\documentclass[11pt]{amsart}

\usepackage{amssymb, amscd}
\usepackage{epsfig, mathtools}
\usepackage[vmargin=1in, hmargin=1.25in]{geometry}
\usepackage[font=small,format=plain,labelfont=bf,up,textfont=it,up]{caption}
\usepackage{microtype}
\usepackage[shortlabels]{enumitem}
\usepackage[backref=page, bookmarks, bookmarksdepth=2, colorlinks=true, linkcolor=blue, citecolor=blue, urlcolor=blue]{hyperref}
\usepackage{etoolbox}
\usepackage{tikz-cd}
\apptocmd{\thebibliography}{\raggedright}{}{}

\makeatletter
\patchcmd{\@maketitle}{\global\topskip42\p@\relax}
  {\global\topskip42\p@\relax \vspace*{-38pt}}
  {}{}
\makeatother

\setenumerate[0]{label=(\alph*)}

\renewcommand*{\backref}[1]{}
\renewcommand*{\backrefalt}[4]{%
    \ifcase #1 (Not cited.)%
    \or        (Cited on page~#2.)%
    \else      (Cited on pages~#2.)%
    \fi}

\newcommand{\arxiv}[1]{\href{http://arxiv.org/abs/#1}{{\tt arXiv:#1}}}

\numberwithin{equation}{section}

\theoremstyle{plain}
\newtheorem{theorem}{Theorem}[section]
\newtheorem{maintheorem}{Theorem}

\newtheorem{maintheoremprime}{Theorem}
\newtheorem{proposition}[theorem]{Proposition}
\newtheorem{lemma}[theorem]{Lemma}

\newtheorem{corollary}[theorem]{Corollary}
\newtheorem{conjecture}[theorem]{Conjecture}

\newtheorem{stepx}{Step}
\newenvironment{step}[1]
 {\renewcommand\thestepx{#1}\stepx}
 {\endstepx}

\newtheorem{claimx}{Claim}
\newenvironment{claim}[1]
 {\renewcommand\theclaimx{#1}\claimx}
 {\endclaimx}

\theoremstyle{definition}
\newtheorem{asm}[theorem]{Assumption}

\newtheorem{defn}[theorem]{Definition}

\newtheorem{notn}[theorem]{Notation}
\newenvironment{notation}[1][]{\begin{notn}[#1]\pushQED{\qed}}{\popQED \end{notn}}

\theoremstyle{remark}
\newtheorem{rmk}[theorem]{Remark}
\newenvironment{remark}[1][]{\begin{rmk}[#1] \pushQED{\qed}}{\popQED \end{rmk}}
\newtheorem{eg}[theorem]{Example}
\newenvironment{example}[1][]{\begin{eg}[#1] \pushQED{\qed}}{\popQED \end{eg}}

\DeclareMathOperator{\Hom}{Hom}

\DeclareMathOperator{\Diff}{Diff}

\DeclareMathOperator{\Image}{Im}

\DeclareMathOperator{\SL}{SL}
\DeclareMathOperator{\Spin}{Spin}

\DeclareMathOperator{\Sp}{Sp}

\DeclareMathOperator{\SO}{SO}

\newcommand\R{\ensuremath{\mathbb{R}}}
\newcommand\C{\ensuremath{\mathbb{C}}}
\newcommand\Z{\ensuremath{\mathbb{Z}}}
\newcommand\Q{\ensuremath{\mathbb{Q}}}

\newcommand\Field{\ensuremath{\mathbb{F}}}

\DeclareMathOperator{\HH}{H}
\newcommand\RH{\ensuremath{\widetilde{\HH}}}
\DeclareMathOperator{\CC}{C}

\newcommand\RC{\ensuremath{\widetilde{\CC}}}

\DeclareMathOperator{\Aut}{Aut}

\DeclareMathOperator{\Interior}{Int}

\newcommand\Set[2]{\ensuremath{\left\{\text{#1 $|$ #2}\right\}}}

\newcommand\cD{\ensuremath{\mathcal{D}}}

\newcommand\cF{\ensuremath{\mathcal{F}}}
\newcommand\cG{\ensuremath{\mathcal{G}}}
\newcommand\cH{\ensuremath{\mathcal{H}}}

\newcommand\cK{\ensuremath{\mathcal{K}}}

\newcommand\cM{\ensuremath{\mathcal{M}}}

\newcommand\cP{\ensuremath{\mathcal{P}}}
\newcommand\cQ{\ensuremath{\mathcal{Q}}}

\newcommand\fB{\ensuremath{\mathfrak{B}}}

\newcommand\fD{\ensuremath{\mathfrak{D}}}

\newcommand\fH{\ensuremath{\mathfrak{H}}}

\newcommand\fR{\ensuremath{\mathfrak{R}}}

\newcommand\bG{\ensuremath{\mathbf{G}}}

\newcommand\bb{\ensuremath{\mathbf{b}}}

\newcommand\bbD{\ensuremath{\mathbb{D}}}

\newcommand\bbL{\ensuremath{\mathbb{L}}}

\newcommand\bbT{\ensuremath{\mathbb{T}}}

\newcommand\bbX{\ensuremath{\mathbb{X}}}
\newcommand\bbY{\ensuremath{\mathbb{Y}}}

\newcommand\bbk{\ensuremath{\Bbbk}}

\newcommand\tT{\ensuremath{\widetilde{T}}}

\newcommand\tX{\ensuremath{\widetilde{X}}}

\newcommand\oU{\ensuremath{\overline{U}}}

\newcommand\ophi{\ensuremath{\overline{\phi}}}
\newcommand\opsi{\ensuremath{\overline{\psi}}}
\newcommand\omu{\ensuremath{\overline{\mu}}}

\newcommand\ub{\ensuremath{\underline{b}}}

\newcommand\up{\ensuremath{\underline{p}}}

\newcommand\uchi{\ensuremath{\underline{\chi}}}

\newcommand\hC{\ensuremath{\widehat{C}}}

\newcommand\hH{\ensuremath{\widehat{H}}}

\newcommand\hU{\ensuremath{\widehat{U}}}

\newcommand\hcD{\ensuremath{\widehat{\cD}}}

\DeclareMathOperator{\Mod}{Mod}
\DeclareMathOperator{\Moduli}{\mathcal{M}}
\newcommand\hMod{\ensuremath{\widehat{\Mod}}}
\DeclareMathOperator{\Torelli}{\mathcal{I}}

\DeclareMathOperator{\PP}{PP}

\newcommand\bbTT{\ensuremath{\mathbb{TT}}}
\newcommand\bbSB{\ensuremath{\mathbb{SB}}}

\DeclareMathOperator{\FLink}{\overrightarrow{Link}}
\DeclareMathOperator{\Simp}{\tt Simp}
\DeclareMathOperator{\tSimp}{\widetilde{\Simp}}
\newcommand\ubbk{\ensuremath{\underline{\bbk}}}

\newcommand\Romega[1]{\ensuremath{\omega^{\fR}_{#1}}}

\DeclareMathOperator{\emptysimp}{[\ ]}

\title[Stable cohomology of the moduli space of curves with level structures]{The stable cohomology of the moduli space of curves with level structures}

\author{Andrew Putman}
\address{Dept of Mathematics; University of Notre Dame; 255 Hurley Hall; Notre Dame, IN 46556}
\email{andyp@nd.edu}
\thanks{AP was supported in part by NSF grant DMS-1811210.}

\date{June 23, 2025}

\begin{document}

\newpage

\begin{abstract}
We prove that in a stable range, the rational cohomology of the moduli space of curves with level structures
is the same as that of the ordinary moduli space of curves.
\end{abstract}

\maketitle
\thispagestyle{empty}

\section{Introduction}
\label{section:introduction}

Let $\Moduli_{g,p}$ be the moduli stack of smooth genus $g$ algebraic curves over
$\C$ equipped with $p$ distinct ordered marked points.
The fundamental group of $\Moduli_{g,p}$ is the
mapping class group $\Mod_{g,p}$ of an oriented genus $g$ surface
$\Sigma_{g,p}$ with $p$ punctures, i.e., the group of isotopy classes of
orientation-preserving diffeomorphisms of $\Sigma_{g,p}$ that fix each puncture.
In fact, $\Moduli_{g,p}$ is a classifying stack for $\Mod_{g,p}$, so
\[\HH^{\bullet}(\Moduli_{g,p};\Q) \cong \HH^{\bullet}(\Mod_{g,p};\Q).\]
There is a rich interplay between the topology of $\Mod_{g,p}$ and the algebraic
geometry of $\Moduli_{g,p}$.  In this paper, we study the cohomology of certain
finite covers of $\Moduli_{g,p}$, or equivalently
finite-index subgroups of $\Mod_{g,p}$.

\subsection{Analogy}
More generally, let $\Sigma_{g,p}^b$ be an oriented genus $g$ surface with $p$ punctures
and $b$ boundary components and let $\Mod_{g,p}^b$ be its mapping class group, i.e., the
group of isotopy classes of orientation-preserving diffeomorphisms of $\Sigma_{g,p}^b$
that fix each puncture and boundary component pointwise.  We will omit $p$ or $b$ if it
vanishes.  There is a fruitful analogy between $\Mod_{g,p}^b$ and arithmetic groups
like $\SL_n(\Z)$ (see, e.g., \cite{BridsonVogtmann}).  This table lists some parallel structures and results:

\medskip
\begin{tabular}{l|l|l}
& $\SL_n(\Z)$ & $\Mod_{g,p}^b$ \\
\hline
natural action                     & vector in $\Z^n$                                               & curve on $\Sigma_{g,p}^b$ \\
associated space                   & locally symmetric space                                        & $\Moduli_{g,p}$ \\
normal form                        & Jordan normal form                                             & Thurston normal form (see \cite{FLP}) \\
Bieri--Eckmann duality             & Borel--Serre \cite{BorelSerreCorners}                          & Harer \cite{HarerVCD} \\
homological stability              & Charney \cite{CharneyStability}, Maazen \cite{MaazenStability} & Harer \cite{HarerStability} \\
calculation of stable $\HH^{\bullet}$ & Borel \cite{BorelStability}                                    & Madsen--Weiss \cite{MadsenWeiss}
\end{tabular}
\medskip

\noindent
Our main theorem gives another entry in this table.  It is related to but different from
homological stability, so we discuss this first.

\subsection{Homological stability}
For simplicity, we restrict to surfaces without punctures.
An embedding $\Sigma_{g}^b \hookrightarrow \Sigma_{g'}^{b'}$
induces a homomorphism
$\Mod_{g}^b \rightarrow \Mod_{g'}^{b'}$ that extends mapping classes by the identity.  Harer \cite{HarerStability}
proved that the induced map
$\HH^k(\Mod_{g'}^{b'}) \rightarrow \HH^k(\Mod_{g}^{b})$
is an isomorphism if $g \gg k$.  The cohomology in
this regime is known as the stable cohomology.  Madsen--Weiss \cite{MadsenWeiss}
proved that rationally it is a polynomial
algebra in classes $\kappa_n \in \HH^{2n}$ called the Miller--Morita--Mumford classes.
Some surveys about this include \cite{GalatiusSurvey, HatcherSurvey, WahlHandbook, WahlSurvey}.

\subsection{Borel stability}
Borel's stability theorem \cite{BorelStability} is about another kind of stability.
Roughly speaking, it says that in a stable range,
the rational cohomology of a lattice $\Gamma$ in a semisimple Lie group $\bG$
depends only on $\bG$, not on $\Gamma$.  In particular, it is unchanged when you replace $\Gamma$ by
a finite-index subgroup.  

For example, $\Gamma = \SL_n(\Z)$ is a lattice in $\bG = \SL_n(\R)$.
For $\ell \geq 2$, define $\SL_n(\Z,\ell)$ be the level-$\ell$ subgroup of $\SL_n(\Z)$, i.e., the kernel of the action of $\SL_n(\Z)$ on
$(\Z/\ell)^n$.  We thus have a short exact sequence
\begin{equation}
\label{eqn:sllevelseq}
1 \longrightarrow \SL_n(\Z,\ell) \longrightarrow \SL_n(\Z) \longrightarrow \SL_n(\Z/\ell) \longrightarrow 1.
\end{equation}
Borel's stability theorem implies that the inclusion $\SL_n(\Z,\ell) \hookrightarrow \SL_n(\Z)$
induces an isomorphism\footnote{We have switched to homology since that is more natural for
the subsequent discussion.} $\HH_k(\SL_n(\Z,\ell);\Q) \cong \HH_k(\SL_n(\Z);\Q)$ for $n \gg k$.
Note that this involves making $\SL_n(\Z)$ smaller by passing to a finite-index subgroup
rather than larger by increasing $n$.
See \cite{CharneyCongruence} and \cite[Theorem C]{PutmanTwistedStability} for direct proofs 
that passing to $\SL_n(\Z,\ell)$ does not change
the stable rational homology.

\subsection{Level-\texorpdfstring{$\ell$}{l} subgroup}
For $\ell \geq 2$, the level-$\ell$ subgroup of $\Mod_{g,p}^b$, denoted
$\Mod_{g,p}^b(\ell)$, is the kernel of the action of $\Mod_{g,p}^b$ on
$\HH_1(\Sigma_{g,p}^b;\Z/\ell)$.  This action preserves the algebraic intersection form,
which is a symplectic form if $p+b \leq 1$.  In that case, we have a short exact sequence
\[1 \longrightarrow \Mod_{g,p}^b(\ell) \longrightarrow \Mod_{g,p}^b \longrightarrow \Sp_{2g}(\Z/\ell) \longrightarrow 1\]
that is analogous to \eqref{eqn:sllevelseq}.  For $p+b \geq 2$, we get a similar exact sequence, but
with a more complicated cokernel.  For $b=0$ and $p \leq 1$, the associated finite cover of
$\Moduli_{g,p}$ is the moduli space $\Moduli_{g,p}[\ell]$ of smooth genus-$g$ curves over $\C$
with $p$ marked points equipped with a full level-$\ell$ structure, i.e., a basis for the $\ell$-torsion in their
Jacobian.\footnote{The subgroup of $\Mod_{g,p}$ corresponding to $\Moduli_{g,p}[\ell]$ is the kernel $\Mod_{g,p}[\ell]$
of the action of $\Mod_{g,p}$ on $\HH_1(\Sigma_g;\Z/\ell)$
coming from the map $\Mod_{g,p} \rightarrow \Mod_g$ that fills in the punctures.  In
general $\Mod_{g,p}[\ell]$ is larger than $\Mod_{g,p}(\ell)$, but they are equal for $p \leq 1$.
We also prove a theorem for $\Mod_{g,p}[\ell]$ (see \S \ref{section:transfer}).}

\subsection{Main theorem}
Since $\Mod_{g,p}^b$ is not a lattice in a Lie group, the only potential analogue of the Borel stability theorem
that might possibly make sense for it would involve passing to a finite-index subgroup like $\Mod_{g,p}^b(\ell)$.  Our main
theorem is about precisely this:

\begin{maintheorem}
\label{maintheorem:mod}
Let $g,p,b \geq 0$ and $\ell \geq 2$.  Then the map
$\HH_k(\Mod_{g,p}^b(\ell);\Q) \rightarrow \HH_k(\Mod_{g,p}^b;\Q)$ induced
by the inclusion $\Mod_{g,p}^b(\ell) \hookrightarrow \Mod_{g,p}^b$ is an isomorphism
if $g \geq 2k^2+7k+2$.
\end{maintheorem}

\subsection{Prior work}
The cases $k=1,2$ of Theorem \ref{maintheorem:mod} were already known.\footnote{Actually, 
the cited papers only handle
the kernel $\Mod_{g,p}^b[\ell]$ of the action of $\Mod_{g,p}^b$ on $\HH_1(\Sigma_g;\Z/\ell)$ coming
from the map $\Mod_{g,p}^b \rightarrow \Mod_g$ that fills in the punctures and glues discs
to the boundary components.  We also prove a
theorem for $\Mod_{g,p}^b[\ell]$ (see \S \ref{section:transfer}), so
even for $k=1,2$ our theorem is stronger than previous work.}
The case $k=1$ was proved by Hain \cite{HainSurvey} using
work of Johnson \cite{JohnsonAbel} on $\HH_1$ of the Torelli subgroup of
$\Mod(\Sigma_g)$.  Hain's proof gives the better stable range $g \geq 3$.
Little is known about
the higher homology groups of the Torelli group, so this approach does
not generalize (but see \S \ref{section:torelli} below).  The case $k=2$ was proved
by Putman \cite{PutmanH2Level}.  The paper \cite{PutmanH2Level} also gives a better bound, namely $g \geq 5$.
We will discuss the relationship between our proof and \cite{PutmanH2Level} below in \S \ref{section:proofsketch}.

\subsection{Necessity of hypotheses}
The hypotheses in Theorem \ref{maintheorem:mod} are necessary:
\begin{itemize}
\item No result like Theorem \ref{maintheorem:mod} can hold
for integral homology.  Indeed, Perron \cite{PerronH1},
Sato \cite{SatoH1}, and Putman \cite{PutmanPicard} identified
exotic torsion elements of $\HH_1(\Mod_{g,p}^b(\ell);\Z)$ that do not
come from $\HH_1(\Mod_{g,p}^b;\Z)$.  Presumably similar torsion
phenomena also occur for higher integral homology groups.  A representation-theoretic
form of stability for this torsion was proved in \cite[Theorem K]{PutmanSam}.
\item Theorem \ref{maintheorem:mod}'s conclusion is false
outside a stable range.  Indeed, Church--Farb--Putman
\cite{ChurchFarbPutmanModVCD} and Morita--Sakasai--Suzuki
\cite{MoritaSakasaiSuzukiModVCD} independently
proved that $\HH^{4g-5}(\Mod(\Sigma_g);\Q)=0$, but
Fullarton--Putman
\cite{FullartonPutman} proved that $\HH^{4g-5}(\Mod(\Sigma_g,\ell);\Q)$ is enormous.
Here $4g-5$ is the virtual 
cohomological dimension of $\Mod(\Sigma_g)$; see \cite{HarerVCD}.
Brendle--Broaddus--Putman \cite{BrendleBroaddusPutmanIsrael}
generalized \cite{FullartonPutman} to $\Mod_{g,p}^b(\ell)$; however, the cohomology of 
$\Mod_{g,p}^b$ in its virtual cohomology dimension is not known in general.
\end{itemize}
On the other hand, the stable range $g \geq 2k^2+7k+2$ can likely be improved.
New ideas are probably needed to get a linear range, but we have not
tried to optimize the constants.

\begin{remark}
Continuing the analogy with $\SL_n(\Z)$, both of the above caveats also apply
to its homology.  For exotic torsion in the stable homology of its finite-index subgroups,
see \cite[Theorem 1.1]{LeeSzczarba} and \cite[Theorem H]{PutmanSam}, and for nonstability outside the stable range see 
\cite{BMSW, ChurchFarbPutmanConjecture, ChurchPutmanCodimension1, LeeSzczarba}
for results at full level, and \cite{LeeSzczarba, MillerPatztPutman, Paraschivescu, Schwermer}
for results at level $\ell \geq 2$.
\end{remark}

\subsection{Other finite-index subgroups, I}
\label{section:transfer}

If $G$ is a finite-index subgroup of $\Gamma$, the transfer map (see, e.g., \cite[\S III.9]{BrownCohomology})
implies that the inclusion $G \hookrightarrow \Gamma$ induces a surjection
$\HH_k(G;\Q) \twoheadrightarrow \HH_k(\Gamma;\Q)$ for all $k$.  Therefore if we are in a regime where
the map
$\HH_k(\Mod_{g,p}^b(\ell);\Q) \rightarrow \HH_k(\Mod_{g,p}^b;\Q)$
is an isomorphism, then for any intermediate subgroup
\[\Mod_{g,p}^b(\ell) \subset G \subset \Mod_{g,p}^b\]
the map
$\HH_k(G;\Q) \rightarrow \HH_k(\Mod_{g,p}^b;\Q)$
is also an isomorphism.  Theorem \ref{maintheorem:mod} thus 
implies a similar theorem for subgroups of $\Mod_{g,p}^b$ containing
some $\Mod_{g,p}^b(\ell)$.  Here are two examples:

\begin{example}
The group $\Mod_{g,p}^b$ acts on $\HH_1(\Sigma_g;\Z/\ell)$ via the map $\Mod_{g,p}^b \rightarrow \Mod_g$ that
fills in the punctures and glues discs to the boundary components.
The kernel $\Mod_{g,p}^b[\ell]$ of this action satisfies
\[\Mod_{g,p}^b(\ell) \subset \Mod_{g,p}^b[\ell] \subset \Mod_{g,p}^b.\]
In the literature, ``level-$\ell$ subgroup'' often means
$\Mod_{g,p}^b[\ell]$ rather than $\Mod_{g,p}^b(\ell)$.
\end{example}

\begin{example}
A {\em spin structure} on $\Sigma_g$ is a $\Spin(2)$-structure on its tangent bundle, where
$\Spin(2)$ is the canonical double-cover of $\SO(2)$.
Let $\omega(-,-)$ be the algebraic intersection pairing on $\HH_1(\Sigma_g;\Field_2)$.
Johnson \cite{JohnsonSpin} showed that spin structures on $\Sigma_g$ can
be identified with quadratic forms $q$ on $\HH_1(\Sigma_g;\Field_2)$ that refine $\omega$, i.e., functions
$q\colon \HH_1(\Sigma_g;\Field_2) \rightarrow \Field_2$ such that
\[q(x+y) = q(x) + q(y) + \omega(x,y) \quad \text{for all $x,y \in \HH_1(\Sigma_g;\Field_2)$.}\]
Such $q$ are classified up to isomorphism by their $\Field_2$-valued Arf invariant.  
The group $\Mod_g$ acts transitively on the set of spin structures on $\Sigma_g$ with
a fixed Arf invariant.  If $\sigma$ is a spin structure on $\Sigma_g$, then
the stabilizer subgroup $\Mod_g(\sigma)$ of $\sigma$ in $\Mod_g$ is called a
{\em spin mapping class group}\footnote{Be warned that in the literature there is another
group that is often called the spin mapping class group.  This group is a $\Z/2$-extension of
$\Mod_g(\sigma)$, and does not lie in $\Mod_g$.  See, e.g., \cite{GalatiusSpin}.} (see, e.g., \cite{HarerStableSpin, HarerH2Spin}).
We have
\[\Mod_g(2) \subset \Mod_g(\sigma) \subset \Mod_g,\]
so our theorem implies a similar result for $\Mod_g(\sigma)$.
\end{example}

\subsection{Other finite-index subgroups, II}
\label{section:finiteindexii}

It is natural to wonder if something like Theorem \ref{maintheorem:mod} holds for all finite-index subgroups, not
just the level-$\ell$ ones.  For $\HH_1$, this is a conjecture of Ivanov \cite{IvanovConjecture} that has been the subject
of a large amount of work; see, e.g., \cite{ErshovHe, PutmanFiniteIndexNote, PutmanWieland}.  These papers prove this
in many cases, but Ivanov's conjecture remains open in general.  For $k \geq 2$, nothing is known about the stable
$\HH_k$ of finite-index subgroups of the mapping class group other than Theorem \ref{maintheorem:mod}.

\subsection{Torelli group}
\label{section:torelli}

The intersection of the $\Mod_{g,p}^b(\ell)$ as $\ell$ ranges over integers $\ell \geq 2$ is
the Torelli group, i.e., the kernel $\Torelli_{g,p}^b$ of the action of $\Mod_{g,p}^b$
on $\HH_1(\Sigma_{g,p}^b;\Z)$.  Little is known about the homology of $\Torelli_{g,p}^b$.  Indeed,
while Johnson \cite{JohnsonAbel} calculated\footnote{Johnson's work covers the cases where $p+b \leq 1$.  See
\cite{PutmanJohnsonKernel} for how to generalize this to surfaces with multiple punctures and boundary components
(at least rationally).} $\HH_1(\Torelli_{g,p}^b)$ and showed that it was
finitely generated for $g \geq 3$, aside from a few low-complexity cases
it is not known if $\HH_2(\Torelli_{g,p}^b)$ is finitely generated.  It is unclear
if Theorem \ref{maintheorem:mod} implies anything about the homology of $\Torelli_{g,p}^b$.

However, sufficient regularity results about the homology of Torelli would imply Theorem \ref{maintheorem:mod}.
To explain this, we restrict for simplicity to closed surfaces.  Let
\[\Sp_{2g}(\Z,\ell) = \ker(\Sp_{2g}(\Z) \rightarrow \Sp_{2g}(\Z/\ell))\]
be the level-$\ell$ subgroup of $\Sp_{2g}(\Z)$.  The commutative
diagram of short exact sequences
\[\begin{tikzcd}[row sep=scriptsize]
1 \arrow{r} & \Torelli_g \arrow{r} \arrow{d}{=} & \Mod_g(\ell) \arrow{r} \arrow{d} & \Sp_{2g}(\Z,\ell) \arrow{r} \arrow{d} & 1 \\
1 \arrow{r} & \Torelli_g \arrow{r}              & \Mod_g       \arrow{r}           & \Sp_{2g}(\Z)      \arrow{r}           & 1
\end{tikzcd}\]
induces a map between the corresponding Hochschild--Serre spectral sequences.
To prove Theorem \ref{maintheorem:mod} (though perhaps with a different bound), it is enough to prove that this map between
spectral sequences is an isomorphism in a range, i.e., that for $g$ large we
have
\[\HH_p(\Sp_{2g}(\Z,\ell);\HH_q(\Torelli_g;\Q)) \cong \HH_p(\Sp_{2g}(\Z);\HH_q(\Torelli_g;\Q)) \quad \text{for $p+q \leq k$}.\]
By the version of the Borel stability theorem with twisted coefficients \cite{BorelStability2}, this would
be true if the following folklore conjecture holds:

\begin{conjecture}
\label{conjecture:torellifinite}
For each $k$, there exists some $G_k$ such that for $g \geq G_k$, the homology group
$\HH_k(\Torelli_g;\Q)$ is finite-dimensional and the action of $\Sp_{2g}(\Z)$ on it
extends to a rational representation of the algebraic group $\Sp_{2g}(\Q)$.
\end{conjecture}

Johnson's aforementioned work on $\HH_1(\Torelli_g)$ shows that this holds for $k=1$ with $G_1 = 3$, but it is
open for all $k \geq 2$.  One can view Theorem \ref{maintheorem:mod} as evidence for Conjecture
\ref{conjecture:torellifinite}.

\subsection{Automorphism groups of free groups}

For a free group $F_n$, its automorphism group $\Aut(F_n)$ shares many
features with $\Mod_{g,p}^b$, so it is also natural to hope that something like Theorem \ref{maintheorem:mod} holds
for $\Aut(F_n)$.  A deep theorem of Galatius \cite{GalatiusAut} says that
\[\HH_k(\Aut(F_n);\Q) = 0 \quad \text{for $n \gg k$},\]
so the natural conjecture is that in a stable range, the rational homology of at least the level-$\ell$ subgroup
of $\Aut(F_n)$ vanishes.  

This is known for $k=1$.  Indeed, a deep theorem of 
Kaluba--Kielak--Nowak (\cite{KKNAutT}, see also \cite{KNOAutT})
says that $\Aut(F_n)$ has Kazhdan's Property (T) for $n \geq 5$, which implies that
$\HH_1(\Gamma;\Q)=0$ for {\em all}
finite-index subgroups $\Gamma$ of $\Aut(F_n)$.  Recall from \S \ref{section:finiteindexii} that Ivanov
conjectured something similar for the mapping class group.\footnote{It is still not known if the mapping class group
has Kazhdan's Property (T).}
Day--Putman \cite[Theorem D]{DayPutmanH2IA} proved that the rational $\HH_2$
of the level-$\ell$ subgroup of $\Aut(F_n)$ is $0$.  We expect that the techniques used to prove
Theorem \ref{maintheorem:mod} could be useful for extending this to the higher
$\HH_k$.

\subsection{Sketch of proof}
\label{section:proofsketch}

We now sketch the proof of Theorem \ref{maintheorem:mod}, focusing for simplicity
on the key case of $\Mod_g^1(\ell)$.  The starting point is the following basic
fact about group homology, which strengthens the observation at the start
of \S \ref{section:transfer} above.  Let $G$ be a finite-index {\em normal} subgroup
of a group $\Gamma$.  Using the transfer
map (see, e.g., \cite[\S III.9]{BrownCohomology}), one can show that
\begin{equation}
\label{eqn:coinvariants}
\HH_k(\Gamma;\Q) = \left(\HH_k\left(G;\Q\right)\right)_{\Gamma},
\end{equation}
where $\Gamma$ denotes the coinvariants
of the action of $\Gamma$ on $\HH_k(G;\Q)$ induced by the conjugation action of $\Gamma$ on $G$. 
Thus $\HH_k(G;\Q) \cong \HH_k(\Gamma;\Q)$ precisely when $\Gamma$ acts trivially on
$\HH_k(G;\Q)$.

Applying this to the finite-index normal subgroup $\Mod_g^1(\ell)$ of $\Mod_g^1$, we see that
the following are equivalent:
\begin{itemize}
\item $\HH_k(\Mod_g^1(\ell);\Q) \cong \HH_k(\Mod_g^1;\Q)$.
\item $\Mod_g^1$ acts trivially on $\HH_k(\Mod_g^1(\ell);\Q)$.
\end{itemize}
We check the second condition for $g \gg k$.  Since $\Mod_g^1$ is generated by Dehn twists $T_{\gamma}$ about nonseparating simple
closed curves $\gamma$, it is enough to prove that these $T_{\gamma}$ act
trivially on $\HH_k(\Mod_g^1(\ell);\Q)$.  Embed $\Sigma_{g-1}^1$ into
$\Sigma_g^1$ as follows:\\
\centerline{\psfig{file=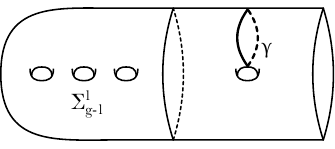,scale=1}}
Since $T_{\gamma}$ commutes with mapping classes supported on $\Sigma_{g-1}^1$,
it acts trivially on the image of $\HH_k(\Mod_{g-1}^1(\ell);\Q)$ in
$\HH_k(\Mod_g^1(\ell);\Q)$.  We deduce that it is enough to prove the
following weaker version of Theorem \ref{maintheorem:mod}:

\begin{maintheoremprime}
\label{maintheorem:modstability}
Let $g \geq 0$ and $\ell \geq 2$.  Then the map
$\HH_k(\Mod_{g-1}^1(\ell);\Q) \rightarrow \HH_k(\Mod_{g}^1(\ell);\Q)$ induced
by the above inclusion $\Sigma_{g-1}^1 \hookrightarrow \Sigma_g^1$ is a surjection
if $g \geq 2k^2+7k+2$.
\end{maintheoremprime}

This resembles a homological stability theorem (or at least the surjective half of one),
and it is natural to try to use the well-developed homological stability machine (see, e.g., \cite{RandalWilliamsWahl})
to prove it.  However, you immediately run into a fundamental problem: the input to
this machine is an action of $\Mod_g^1(\ell)$ on a highly-connected simplicial complex
$\bbX$, and one of the basic properties you need is that $\Mod_g^1(\ell)$ acts
transitively on the vertices.  If you use one of the simplicial complexes
used to prove homological stability for $\Mod_g^1$, this fails.\footnote{And
this cannot be easily avoided since the homological stability machine naturally gives
theorems about integral homology, while Theorem \ref{maintheorem:modstability} only holds rationally.}

However, the machine does give a weaker conclusion: rather than saying that a single
$\HH_k(\Mod_{g-1}^1(\ell);\Q)$ surjects onto $\HH_k(\Mod_g^1(\ell);\Q)$, it implies that
if you take the direct sum over {\em all} embeddings of $\Sigma_{g-1}^1$ into $\Sigma_g^1$, then
you do get a surjective map:\footnote{In fact, you can take the direct sum over orbit
representatives of the action of $\Mod_g^1(\ell)$ on the set of such embeddings.
If we were working with $\Mod_g^1$ instead of $\Mod_g^1(\ell)$, then the ``change
of coordinates'' principle from \cite[\S 1.3.2]{FarbMargalitPrimer} would show that there
is a single such orbit.} 
\[\bigoplus_{\Sigma_{g-1}^1 \hookrightarrow \Sigma_g^1} \HH_k(\Mod_{g-1}^1(\ell);\Q) \twoheadrightarrow \HH_k(\Mod_g^1(\ell);\Q).\]
It is therefore enough to show that each term in this direct sum has the same image.  This
requires an elaborate induction, and in particular requires proving not just
Theorem \ref{maintheorem:modstability}, but also a twisted analogue of Theorem \ref{maintheorem:modstability}
with coefficients in certain rather complicated coefficient systems (tensor powers of the standard representation
and Prym representations; see
\S \ref{section:twistedcoefficientsstd} and \S \ref{section:twistedcoefficientsprym} below).

\begin{remark}
The above outline resembles the proof of the case $k=2$ of Theorem \ref{maintheorem:mod} proved
by the author long ago in \cite{PutmanH2Level}.  Two new developments since then allowed
us to prove the general case:
\begin{itemize}
\item The author's work on twisted homological stability in \cite{PutmanTwistedStability}, which
gives a flexible tool for incorporating twisted coefficients into homological stability proofs.  There
is an earlier approach to this due to Dwyer \cite{DwyerTwisted}, but it seems hard to 
use it
in our proof.
\item The author's work on stability properties of ``partial Torelli groups'' in \cite{PutmanPartialTorelli}, 
which forms the basis for the elaborate induction discussed above as well as the simplicial complexes used in this paper.\qedhere
\end{itemize}
\end{remark}

\subsection{Standard representation}
\label{section:twistedcoefficientsstd}

The general twisted version of Theorem \ref{maintheorem:mod} that we will prove is a little
technical, so we close this introduction by stating two special cases of it that we think
are of independent interest.  The first involves representations built from $\HH_1(\Sigma_{g,p}^b;\Q)$:

\begin{maintheorem}
\label{maintheorem:modstd}
Let $g,p,b \geq 0$ and $\ell \geq 2$.  Then for $r \geq 0$, the map
\[\HH_k\left(\Mod_{g,p}^b\left(\ell\right);\HH_1(\Sigma_{g,p}^b;\Q)^{\otimes r}\right) \rightarrow \HH_k\left(\Mod_{g,p}^b;\HH_1(\Sigma_{g,p}^b;\Q)^{\otimes r}\right)\]
is an isomorphism if $g \geq 2(k+r)^2+7k+6r+2$.
\end{maintheorem}

Note that for $r=0$ this reduces to Theorem \ref{maintheorem:mod}. In particular, setting $r=0$
we get the bound $g \geq 2k^2+7k+2$ from that theorem.  We will
also prove a version of Theorem \ref{maintheorem:modstd} with coefficients
in $\HH_1(\Sigma_g;\Q)^{\otimes r}$ rather than $\HH_1(\Sigma_{g,p}^b;\Q)^{\otimes r}$.
See Theorem \ref{maintheorem:modstd2} in \S \ref{section:closed}.

\subsection{Prym representations}
\label{section:twistedcoefficientsprym}

The other representation we need to handle is the Prym representation, which is defined
as follows.  Assume that $p+b \geq 1$.
Let $\pi\colon S \rightarrow \Sigma_{g,p}^b$
be the regular cover with deck group $\HH_1(\Sigma_g;\Z/\ell)$ coming from the group homomorphism
\[\pi_1(\Sigma_{g,p}^b) \rightarrow \HH_1(\Sigma_{g,p}^b;\Z/\ell) \rightarrow \HH_1(\Sigma_g;\Z/\ell),\]
where the second map glues discs to all the boundary components and fills in all the punctures.  
Since $\Mod_{g,p}^b(\ell)$ acts trivially on $\HH_1(\Sigma_g;\Z/\ell)$ and $p+b \geq 1$, covering
space theory allows us to lift elements of $\Mod_g^1(\ell)$ to mapping classes of $S$ fixing the punctures
and boundary components pointwise.\footnote{This requires $p+d \geq 1$ to ensure that
our lift is well-defined; otherwise, it would only be well-defined up to the action of the deck group.}

This gives us an action of $\Mod_{g,p}^b(\ell)$ on
$\fH_{g,p}^b(\ell;\Q) \coloneqq \HH_1(S;\Q)$.
These representations are called Prym representations.  They were first studied
by Looijenga \cite{LooijengaPrym}, who (essentially) determined their image.  The
map $\pi\colon S \rightarrow \Sigma_{g,p}^b$ induces a map
$\fH_{g,p}^b(\ell;\Q) \rightarrow \HH_1(\Sigma_{g,p}^b;\Q)$.  Our result
is as follows.  Note that our bound here is the same as in the case $r=1$ of Theorem \ref{maintheorem:modstd}:

\begin{maintheorem}
\label{maintheorem:modprym}
Let $g,p,b \geq 0$ and $\ell \geq 2$ be such that $p+b \geq 1$.  Then the map
\[\HH_k\left(\Mod_{g,p}^b\left(\ell\right);\fH_{g,p}^b\left(\ell;\Q\right)\right) \rightarrow \HH_k\left(\Mod_{g,p}^b;\HH_1(\Sigma_{g,p}^b;\Q)\right)\]
is an isomorphism if $g \geq 2(k+1)^2+7k+8$.
\end{maintheorem}

\begin{remark}
We proved the case $k=1$ and $(b,p) = (1,0)$ of this by a brute force calculation in \cite[Theorem C]{PutmanAbelianCovers}, which allowed
us to prove the case $k=2$ of Theorem \ref{maintheorem:mod} in \cite{PutmanH2Level}.  One
of the main insights of the present paper is that one can simultaneously prove
Theorems \ref{maintheorem:mod} and \ref{maintheorem:modprym} with almost no calculations.
\end{remark}

\begin{remark}
The result \cite[Theorem B]{PutmanAbelianCovers} might appear to say that Theorem \ref{maintheorem:modprym} is false for
$\Mod_{g,1}(\ell)$.  However, the Prym representation covered by \cite[Theorem B]{PutmanAbelianCovers} is slightly
different from the one in Theorem \ref{maintheorem:modprym} since it involves the homology of the $\HH_1(\Sigma_g;\Z/\ell)$-cover
$S \rightarrow \Sigma_g$ rather than the cover $S \rightarrow \Sigma_{g,1}$.
\end{remark}

\begin{remark}
In fact, what we need is something similar to Theorem \ref{maintheorem:modprym} for tensor powers of the Prym representations.
Unfortunately, the naive analogue of Theorem \ref{maintheorem:modprym} for higher tensor powers of the Prym representations is
false even for $\HH_0$,\footnote{Recall that $\HH_0$ calculates the coinvariants.  The
$\Mod_{g,1}$-coinvariants of $\HH_1(\Sigma_{g,1};\Q)^{\otimes 2}$ are $\Q$, which is
detected by the algebraic intersection pairing $\HH_1(\Sigma_{g,1};\Q)^{\otimes 2} \rightarrow \Q$.
The Reidemeister pairing described below in \S \ref{section:reidemeister} gives a 
$\Mod_{g,1}(\ell)$-invariant surjective map 
$\fH_{g,1}(\ell;\Q)^{\otimes 2} \rightarrow \Q[H]$ with $H = \HH_1(\Sigma_{g,1};\Z/\ell)$,
so the $\Mod_{g,1}(\ell)$-coinvariants of $\fH_{g,1}(\ell;Q)^{\otimes 2}$ are 
larger than just $\Q$.  In fact, these coinvariants are exactly $\Q[H]$.}
and formulating the correct version is a bit subtle.  See Theorem \ref{maintheorem:generalcase}
in \S \ref{section:proof} for details.
\end{remark}

\subsection{Outline of paper}
The first part of the paper (\S \ref{section:levell}--\ref{section:stabilitymachine})
discusses background and establishes
some basic results.  We start in \S \ref{section:levell} with some group-theoretic properties of
the group $\Mod_{g,p}^b(\ell)$.  We then turn to the twisted homological
stability machine from \cite{PutmanTwistedStability}.  This is contained in two sections:
\S \ref{section:semisimplicial} is devoted to basic facts about simplicial complexes and
their homology, and \S \ref{section:stabilitymachine} isolates the part of the machine
that we need.  The input to this machine is a simplicial complex equipped with
a ``coefficient system''.  

The next part (\S \ref{section:tetheredtori}--\ref{section:partialstability}) 
discusses some tools needed to apply our machine to
$\Mod_{g,p}^b(\ell)$.
In \S \ref{section:tetheredtori}, we introduce the simplicial complex
we will use (the ``complex of tethered $H$-orthogonal tori''), and in \S \ref{section:prymrepresentations}
we discuss the Prym representations and show how to incorporate them into a coefficient
system on this complex.
The action of $\Mod_g^1(\ell)$
on the Prym representation preserves a bilinear form called the Reidemeister pairing, and
\S \ref{section:reidemeister} develops its properties.  

After an interlude in \S \ref{section:partialstability} on the author's results
from \cite{PutmanPartialTorelli} about the partial Torelli groups,
the final part (\S \ref{section:proof}--\ref{section:closed}) contains the proofs
of our main theorems.  We first handle non-closed surfaces in 
\S \ref{section:proof}, and we then derive our results for closed surfaces
in \S \ref{section:closed}.

\subsection{Acknowledgments.}
I would like to thank Tom Church, Dan Margalit, and Xiyan Zhong for helpful comments on previous drafts of this
paper.

\section{Basic properties of the level-\texorpdfstring{$\ell$}{l} subgroup}
\label{section:levell}

We start by discussing some basic facts about the mapping class group and its subgroups.

\subsection{Birman exact sequence I: mapping class group}
Let $x_0$ be a puncture of $\Sigma_{g,p+1}^b$.  Let $\phi\colon \Mod_{g,p+1}^b \rightarrow \Mod_{g,p}^b$
be the map that fills in $x_0$.  Except in some degenerate cases,
the kernel of $\phi$ is the point-pushing subgroup $\PP_{x_0}(\Sigma_{g,p}^b)$.  Elements
of $\PP_{x_0}(\Sigma_{g,p}^b)$ push the puncture $x_0$ around the surface.  Keeping track of
the path traced out by $x_0$ gives an isomorphism
\[\PP_{x_0}(\Sigma_{g,p}^b) \cong \pi_1(\Sigma_{g,p}^b,x_0).\]
To keep our notation from being unmanageable, we will often omit the basepoint and just
write $\pi_1(\Sigma_{g,p}^b)$.  This is all summarized in the following theorem.
See \cite[\S 4.2]{FarbMargalitPrimer} for a textbook reference.

\begin{theorem}[{Birman exact sequence \cite{BirmanThesis}}]
\label{theorem:birmanpuncture}
Fix some $g,p,b \geq 0$ such that $\pi_1(\Sigma_{g,p}^b)$ is nonabelian, and let $x_0$ be a puncture of $\Sigma_{g,p+1}^b$.  There
is then a short exact sequence
\[1 \longrightarrow \PP_{x_0}(\Sigma_{g,p}^b) \longrightarrow \Mod_{g,p+1}^b \stackrel{\phi}{\longrightarrow} \Mod_{g,p}^b \longrightarrow 1,\]
where $\PP_{x_0}(\Sigma_{g,p}^b) \cong \pi_1(\Sigma_{g,p}^b)$.
\end{theorem}

The following lemma describes the effect of $\PP_{x_0}(\Sigma_{g,p}^b)$ on $\HH_1(\Sigma_{g,p+1}^b)$:

\begin{lemma}
\label{lemma:pointpushh1}
Fix some $g,p,b \geq 0$ such that $\pi_1(\Sigma_{g,p}^b)$ is nonabelian, and let $x_0$ be a puncture of $\Sigma_{g,p+1}^b$.
Let $\bbk$ be a commutative ring.
Let $\rho_1\colon \HH_1(\Sigma_{g,p+1}^b;\bbk) \rightarrow \HH_1(\Sigma_{g,p}^{b};\bbk)$ be the map that fills
in $x_0$ and 
let $\rho_2\colon \PP_{x_0}(\Sigma_{g,p}^{b}) \rightarrow \HH_1(\Sigma_{g,p}^{b};\bbk)$
be the composition
\[\PP_{x_0}(\Sigma_{g,p}^b) \cong \pi_1(\Sigma_{g,p}^{b}) \longrightarrow \HH_1(\Sigma_{g,p}^{b};\bbk).\]
Let $\omega(-,-)$ be the algebraic intersection pairing on $\HH_1(\Sigma_{g,p}^{b};\bbk)$
and let $\zeta \in \HH_1(\Sigma_{g,p}^{b};\bbk)$ be the homology class of
a loop around $x_0$, oriented such that $x_0$ is to its right.
Then for $\gamma \in \PP_{x_0}(\Sigma_{g,p}^b)$ and $z \in \HH_1(\Sigma_{g,p+1}^b;\bbk)$, we
have
\[\gamma(z) = z + \omega(\rho_1(z),\rho_2(\gamma)) \cdot \zeta.\]
\end{lemma}
\begin{proof}
It is enough to check this on $\gamma \in \PP_{x_0}(\Sigma_{g,p}^b) \cong \pi_1(\Sigma_{g,p}^{b})$ and
$z \in \HH_1(\Sigma_{g,p+1}^b;\bbk)$ that can be
represented by simple closed curves.  For these, it is immediate from the following picture:\\
\centerline{\psfig{file=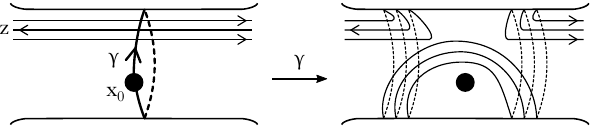,scale=1}}
Here $z \in \HH_1(\Sigma_{g,p+1}^b;\bbk)$ is represented by a cycle that intersects $\gamma$ transversely.
\end{proof}

Next, fix a boundary component $\partial$ of $\Sigma_{g,p}^{b+1}$.  Gluing a punctured disc to $\partial$ and
extending mapping classes over it by the identity, we get a homomorphism $\psi\colon \Mod_{g,p}^{b+1} \rightarrow \Mod_{g,p+1}^b$.
The following folklore result shows that except in degenerate cases the kernel of $\psi$ is the infinite
cyclic subgroup generated by $T_{\partial}$:

\begin{proposition}[{\cite[Proposition 3.19]{FarbMargalitPrimer}}]
\label{proposition:cap}
Fix some $g,p,b \geq 0$ such that $\pi_1(\Sigma_{g,p}^{b+1})$ is nonabelian, and let $\partial$ be a
boundary component of $\Sigma_{g,p}^{b+1}$.  Then there is a central extension
\[1 \longrightarrow \Z \longrightarrow \Mod_{g,p}^{b+1} \stackrel{\psi}{\longrightarrow} \Mod_{g,p+1}^b \longrightarrow 1,\]
where the central $\Z$ is generated by the Dehn twist $T_{\partial}$.
\end{proposition}

\subsection{Partial level-\texorpdfstring{$\ell$}{l} subgroups}
\label{section:partial}

Our proofs will use results about the ``partial Torelli groups'' introduced by the
author in \cite{PutmanPartialTorelli}.  To avoid technicalities, we will only discuss
the special cases of these results needed for our work.\footnote{Unfortunately, the proofs
in \cite{PutmanPartialTorelli} do not simplify much if you restrict to these cases.  We will later
discuss how
to relate the definition we give here to the one in \cite{PutmanPartialTorelli} (see the proof of Theorem \ref{theorem:stability}).}
A subgroup $H < \HH_1(\Sigma_{g,p}^b;\Z/\ell)$ is a symplectic subgroup if the algebraic intersection pairing
\[\omega\colon \HH_1(\Sigma_{g,p}^b;\Z/\ell) \times \HH_1(\Sigma_{g,p}^b;\Z/\ell) \longrightarrow \Z/\ell\]
restricts to a nondegenerate pairing\footnote{Here nondegenerate means that it identifies $H$ with its dual $H^{\vee} = \Hom(H,\Z/\ell)$.} on $H$.  Such an $H$ is of the form $H \cong (\Z/\ell)^{2h}$ for
some $h \geq 0$ called its genus.  We remark that if $p+b \geq 2$ then the algebraic intersection pairing
on $\HH_1(\Sigma_{g,p}^b;\Z/\ell)$ is degenerate, so in that case $\HH_1(\Sigma_{g,p}^b;\Z/\ell)$ is not a symplectic
subgroup of itself.  For a symplectic subgroup $H$ of $\HH_1(\Sigma_{g,p}^b;\Z/\ell)$, the associated
partial level-$\ell$ subgroup, denoted $\Mod_{g,p}^b(H)$, is the group of all
$f \in \Mod_{g,p}^b$ such that $f(x) = x$ for all $x \in H$.

\begin{example}
If $H = 0$, then $\Mod_{g,p}^b(H) = \Mod_{g,p}^b$.
\end{example}

\begin{example}
\label{example:wholething}
If $H$ is a genus-$g$ symplectic subgroup of $\HH_1(\Sigma_{g,p}^b;\Z/\ell)$, then $\Mod_{g,p}^b(H) = \Mod_{g,p}^b(\ell)$.
The point here is that $\Mod_{g,p}^b$ automatically acts trivially on the subgroup $\fB \cong (\Z/\ell)^{p+b-1}$ of $\HH_1(\Sigma_{g,p}^b;\Z/\ell)$
generated by loops surrounding the punctures and boundary components,\footnote{We have $\fB \cong (\Z/\ell)^{p+b-1}$ and not $(\Z/\ell)^{p+b}$ since if you orient them correctly, the sum
of the homology classes of all the loops surrounding the punctures and boundary components is zero.} and $\HH_1(\Sigma_{g,p}^b;\Z/\ell) = \fB \oplus H$.  Thus
acting trivially on $\HH_1(\Sigma_{g,p}^b;\Z/\ell)$ is equivalent to acting trivially on its subgroup $H$.
\end{example}

\subsection{Conventions about symplectic subgroups}
\label{section:symplecticconventions}

Let $H$ be a symplectic subgroup of $\HH_1(\Sigma_{g,p}^b;\Z/\ell)$.  We will often need to relate $\Mod_{g,p}^b(H)$
to the partial level-$\ell$ subgroup on other surfaces.  Technically, the symplectic subgroup on the other surface
is different from $H$; however, there is often a canonical way to identify them.
In this case, we will use the same letter $H$ for both subgroups.  Here are two examples.

\begin{example}
Let $\iota\colon \Sigma_{g,p}^b \hookrightarrow \Sigma_{g',p'}^{b}$ be an embedding.
The kernel of $\iota_{\ast}\colon \HH_1(\Sigma_{g,p}^b;\Z/\ell) \rightarrow \HH_1(\Sigma_{g',p'}^{b'};\Z/\ell)$
is contained in the subgroup generated by loops surrounding boundary components and punctures.  Thus
$\ker(\iota_{\ast}) \cap H = 0$, so $\iota_{\ast}$ maps $H$ isomorphically to a symplectic
subgroup of $\HH_1(\Sigma_{g',p'}^{b'};\Z/\ell)$ that we will also call $H$.  As
long as $\iota$ takes punctures to either points or punctures, 
we
have a map $\Mod_{g,p}^{b}(H) \rightarrow \Mod_{g',p'}^{b'}(H)$ that extends mapping classes
by the identity.
\end{example}

\begin{example}
Let $\iota\colon \Sigma_{g',p'}^{b'} \hookrightarrow \Sigma_{g,p}^b$ be an embedding.  Assume that that
$\iota_{\ast}\colon \HH_1(\Sigma_{g',p'}^{b'};\Z/\ell) \rightarrow \HH_1(\Sigma_{g,p}^{b};\Z/\ell)$ is
injective and the image of $\iota_{\ast}$ contains $H$.  Then using $\iota_{\ast}$,
 we can identify $H$ with a symplectic subgroup of
$\HH_1(\Sigma_{g',p'}^{b'};\Z/\ell)$ that we will also call $H$.  Again, as
long as $\iota$ takes punctures to either points or punctures we have a map
$\Mod_{g',p'}^{b'}(H) \rightarrow \Mod_{g,p}^{b}(H)$ that extends mapping classes
by the identity.
\end{example}

We will only use this convention when it is clear what it means. 

\subsection{Birman exact sequence II: partial level-\texorpdfstring{$\ell$}{l} subgroups}

A version of the Birman exact sequence for the groups $\Mod_{g,p}^b(\ell)$ was proved by
Brendle, Broaddus, and the author in
\cite[Theorem 3.1]{BrendleBroaddusPutmanIsrael}, building on work of the author for the Torelli group in \cite{PutmanCutPasteTorelli}.
For the partial level-$\ell$ subgroups, the appropriate theorem is as follows.  The statement
of this theorem uses the conventions from \S \ref{section:symplecticconventions}.

\begin{theorem}[{Partial mod-$\ell$ Birman exact sequence}]
\label{theorem:birmanpartial}
Fix some $g,p,b \geq 0$ such that $\pi_1(\Sigma_{g,p}^b)$ is nonabelian, and let
$x_0$ be a puncture of $\Sigma_{g,p+1}^b$.  Let $\ell \geq 2$ and let
$H$ be a symplectic subgroup of $\HH_1(\Sigma_{g,p+1}^b;\Z/\ell)$.
There is then a short exact sequence
\[1 \longrightarrow \PP_{x_0}(\Sigma_{g,p}^b,H) \longrightarrow \Mod_{g,p+1}^b(H) \longrightarrow \Mod_{g,p}^{b}(H) \longrightarrow 1,\]
where $\PP_{x_0}(\Sigma_{g,p}^b,H)$ is as follows:
\begin{itemize}
\item If $p=b=0$, then $\PP_{x_0}(\Sigma_{g,p}^b,H) = \PP_{x_0}(\Sigma_{g,p}^b) \cong \pi_1(\Sigma_{g,p}^{b})$.
\item If $p+b \geq 1$, then $\PP_{x_0}(\Sigma_{g,p}^b,H)$ is the kernel of the composition
\[\PP_{x_0}(\Sigma_{g,p}^b) \cong \pi_1(\Sigma_{g,p}^{b}) \rightarrow \HH_1(\Sigma_{g,p}^{b};\Z/\ell) = H \oplus H^{\perp} \stackrel{\text{proj}}{\rightarrow} H.\]
Here $H^{\perp}$ is the orthogonal complement of $H$ with respect to the algebraic intersection pairing.\qedhere
\end{itemize}
\end{theorem}
\begin{proof}
The proof is nearly identical to that of \cite[Theorem 3.1]{BrendleBroaddusPutmanIsrael}, so we will just sketch it.
Letting
\[\PP_{x_0}(\Sigma_{g,p}^b,H) = \PP_{x_0}(\Sigma_{g,p}^b) \cap \Mod_{g,p+1}^b(H),\]
it is easy to see that the Birman exact sequence
\[1 \longrightarrow \PP_{x_0}(\Sigma_{g,p}^b) \longrightarrow \Mod_{g,p+1}^b \longrightarrow \Mod_{g,p}^{b} \longrightarrow 1\]
from Theorem \ref{theorem:birmanpuncture} restricts
to a short exact sequence
\[1 \longrightarrow \PP_{x_0}(\Sigma_{g,p}^b,H) \longrightarrow \Mod_{g,p+1}^b(H) \longrightarrow \Mod_{g,p}^{b}(H) \longrightarrow 1.\]
The nontrivial thing is to identify $\PP_{x_0}(\Sigma_{g,p}^b,H)$, which
follows\footnote{The reason there is a difference between the cases $p=b=0$ and $p+b \geq 1$
is that if $\zeta$ the homology class of a loop surrounding $x_0$, then $\zeta = 0$ if $p=b=0$,
while $\zeta \neq 0$ if $p+b \geq 1$.} from Lemma \ref{lemma:pointpushh1}.
\end{proof}

\begin{remark}
\label{remark:birmanlevel}
If $H$ is a genus-$g$ symplectic subgroup of $\HH_1(\Sigma_{g,p}^n;\Z/\ell)$ and thus $\Mod_{g,p}^b(H) = \Mod_{g,p}^b(\ell)$ (see
Example \ref{example:wholething}), then we will write $\PP_{x_0}(\Sigma_{g,p}^b,\ell)$ for
$\PP_{x_0}(\Sigma_{g,p}^b,H)$.  Theorem \ref{theorem:birmanpartial} thus gives an exact sequence
\[1 \longrightarrow \PP_{x_0}(\Sigma_{g,p}^b,\ell) \longrightarrow \Mod_{g,p+1}^b(\ell) \longrightarrow \Mod_{g,p}^b(\ell) \longrightarrow 1\]
with $\PP_{x_0}(\Sigma_{g,p}^b,\ell)$ the kernel of the map
\[\PP_{x_0}(\Sigma_{g,p}^b) \cong \pi_1(\Sigma_{g,p}^b) \rightarrow \HH_1(\Sigma_{g,p}^b;\Z/\ell) \rightarrow \HH_1(\Sigma_g;\Z/\ell).\qedhere\]
\end{remark}

Since Dehn twists about boundary components always lie in $\Mod_{g,p}^{b+1}(H)$, Proposition
\ref{proposition:cap} immediately implies the following:

\begin{proposition}
\label{proposition:cappartial}
Fix some $g,p,b \geq 0$ such that $\pi_1(\Sigma_{g,p}^{b+1})$ is nonabelian, and let
$\partial$ be a boundary component of $\Sigma_{g,p}^{b+1}$.  Let $\ell \geq 2$ and let
$H$ be a symplectic subgroup of $\HH_1(\Sigma_{g,p}^{b+1};\Z/\ell)$. 
Then there is a central extension
\[1 \longrightarrow \Z \longrightarrow \Mod_{g,p}^{b+1}(H) \longrightarrow \Mod_{g,p+1}^b(H) \longrightarrow 1,\]
where the central $\Z$ is generated by the Dehn twist $T_{\partial}$.
\end{proposition}

\begin{remark}
\label{remark:caplevel}
Again, taking $H$ to be a genus-$g$ symplectic subgroup we get a central extension
\[1 \longrightarrow \Z \longrightarrow \Mod_{g,p}^{b+1}(\ell) \longrightarrow \Mod_{g,p+1}^b(\ell) \longrightarrow 1.\qedhere\]
\end{remark}

\subsection{Generating the partial level-\texorpdfstring{$\ell$}{l} subgroups}

The following lemma describes the difference between the level-$\ell$ subgroup and the partial level-$\ell$ subgroup.
This lemma is true for all surfaces $\Sigma_{g,p}^b$, but we will only need the
case $\Sigma_g^1$, for which the proof is a bit easier.

\begin{lemma}
\label{lemma:partialgen}
Let $g \geq 0$ and $\ell \geq 2$, and
let $H$ be a symplectic subgroup of $\HH_1(\Sigma_{g}^1;\Z/\ell)$.
Then $\Mod_{g}^1(H)$ is generated by $\Mod_{g}^1(\ell)$ along with the set of all Dehn twists
$T_{\gamma}$ such that\footnote{Here we are abusing notation -- to define $[\gamma] \in \HH_1(\Sigma_{g}^1;\Z/\ell)$,
we must first orient $\gamma$.  Changing this orientation replaces $[\gamma]$ by $-[\gamma]$, and
thus does not affect whether $[\gamma] \in H^{\perp}$.  We will make similar abuses
of notation throughout the paper.} $[\gamma] \in H^{\perp}$.  In fact, such $T_{\gamma}$ act
on $H^{\perp}$, and it is enough to take any set of such $T_{\gamma}$ that map to a
generating set for $\Sp(H^{\perp})$.
\end{lemma}
\begin{proof}
If $H$ has genus $h$, then $\Sp(H^{\perp}) \cong \Sp_{2(g-h)}(\Z/\ell)$.  This group can be embedded in
$\Sp_{2g}(\Z)$ as the subgroup of symplectic automorphisms acting trivially on $H$.
The short exact sequence
\[1 \longrightarrow \Mod_g^1(\ell) \longrightarrow \Mod_g^1 \longrightarrow \Sp_{2g}(\Z/\ell) \longrightarrow 1\]
restricts to an exact sequence of the form
\[1 \longrightarrow \Mod_g^1(\ell) \longrightarrow \Mod_g^1(H) \longrightarrow \Sp(H^{\perp}).\]
Dehn twists $T_{\gamma}$ such that $[\gamma] \in H^{\perp}$ map to symplectic transvections in
$\Sp(H^{\perp})$.
Moreover, for every $v \in H^{\perp}$ that is primitive\footnote{That is, such that there does not exist some $w \in H^{\perp}$
and a non-unit $\lambda \in \Z/\ell$ with $v = \lambda w$.} there exists an oriented simple closed curve
$\gamma$ on $\Sigma_g^1$ with $[\gamma] = v$; see\footnote{This reference proves that primitive elements of
$\HH_1(\Sigma_g^1) \cong \Z^{2g}$ can be represented by simple closed curves.  Since primitive elements
of $\HH_1(\Sigma_g^1;\Z/\ell) \cong (\Z/\ell)^{2g}$ can be lifted to primitive elements of $\HH_1(\Sigma_g^1)$, this
implies the corresponding fact for $\HH_1(\Sigma_g^1;\Z/\ell)$.}
\cite[Proposition 6.2]{FarbMargalitPrimer}.  Symplectic transvections about such elements generate
$\Sp(H^{\perp}) \cong \Sp_{2(g-h)}(\Z/\ell)$.  We conclude
that the map $\Mod_g^1(H) \rightarrow \Sp(H^{\perp})$ is surjective, and moreover that
$\Mod_{g}^1(H)$ is generated by $\Mod_{g}^1(\ell)$ along with the set of all Dehn twists
$T_{\gamma}$ such that $[\gamma] \in H^{\perp}$, as desired.
\end{proof}

\begin{remark}
\label{remark:specificgen}
In Lemma \ref{lemma:partialgen}, we can take
\[S = \{T_{\alpha_1},\ldots,T_{\alpha_{g-h}},T_{\beta_1},\ldots,T_{\beta_{g-h}},T_{\gamma_1},\ldots,T_{\gamma_{g-h-1}}\},\]
where the $\alpha_i$ and $\beta_i$ and $\gamma_i$ are as follows:\\
\centerline{\psfig{file=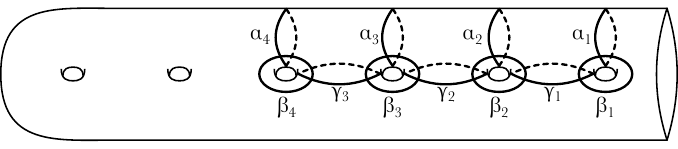,scale=1}}
Here $H$ consists of all elements of homology orthogonal to the
curves about whose twists are in $S$, so $H$ is supported on the handles on the left side of the figure
that have no $S$-curves around them.  This from the fact that
$\Mod_{g-h}$ surjects onto $\Sp_{2(g-h)}(\Z/\ell)$ and the fact that twists about the
curves in the following figure generate $\Mod_{g-h}$:\\
\centerline{\psfig{file=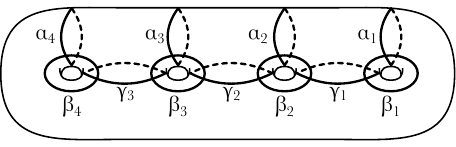,scale=1}}
See \cite[\S 4.4]{FarbMargalitPrimer}.
\end{remark}

\subsection{Subsurface stabilizers}
\label{section:subsurfacestabilizers}

Let $j\colon \Sigma_g^2 \hookrightarrow \Sigma_{g+1}^1$ be the following embedding:\\
\centerline{\psfig{file=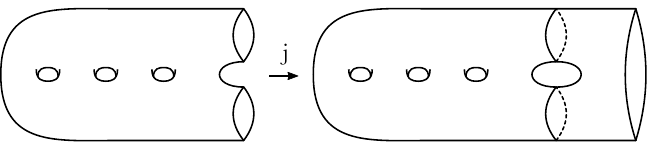,scale=1}}
There is an induced map $j_{\ast}\colon \Mod_g^2 \rightarrow \Mod_{g+1}^1$ that
extends mapping classes on $\Sigma_g^2$ to $\Sigma_{g+1}^1$ by the identity.  Define
\[\hMod_g^2(\ell) = \Set{$f \in \Mod_g^2$}{$j_{\ast}(f) \in \Mod_{g+1}^1(\ell)$}.\]
We have $\hMod_g^2(\ell) \subset \Mod_g^2(\ell)$, but it is not the case
that $\hMod_g^2(\ell) = \Mod_g^2(\ell)$.  For instance, if
$\partial$ is one of the boundary components of $\Sigma_g^2$ then $T_{\partial} \in \Mod_g^2(\ell)$
but $T_{\partial} \notin \hMod_g^2(\ell)$.  However, this is the only difference between
these two groups:

\begin{lemma}
\label{lemma:justboundary}
Let $g \geq 0$ and $\ell \geq 2$, and let $\partial$ be a boundary component
of $\Sigma_g^2$.  Then for all $f \in \Mod_g^2(\ell)$, there exists some $n \in \Z$
such that $T_{\partial}^n f \in \hMod_g^2(\ell)$.
\end{lemma}
\begin{proof}
Let $S \cong \Sigma_g^1$ and $\alpha$ be as in the following figure:\\
\centerline{\psfig{file=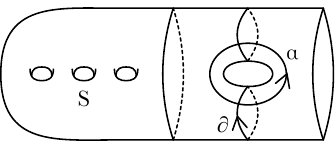,scale=1}}
Identify $\HH_1(S;\Z/\ell)$ with its image
in $\HH_1(\Sigma_{g+1}^1)$.  Since $\HH_1(S;\Z/\ell)$ injects into $\HH_1(\Sigma_g^2;\Z/\ell)$
and $f \in \Mod_g^2(\ell)$, the mapping class $f$ acts trivially on $\HH_1(S;\Z/\ell)$.  Since
$f$ fixes $\partial$, it also fixes $[\partial] \in \HH_1(\Sigma_{g+1}^1;\Z/\ell)$.  It follows that $f$ takes
$[\alpha]$ to an element of $\HH_1(\Sigma_{g+1}^1;\Z/\ell)$ that is orthogonal to $\HH_1(S;\Z/\ell)$ and
has algebraic intersection number $1$ with $[\partial]$.  This implies that $f([\alpha]) = [\alpha] + n [\partial]$
for some $n \in \Z$.  We deduce that $T_{\partial}^{-n} f$ fixes $[\alpha]$ as well as $[\partial]$ and $\HH_1(S;\Z/\ell)$.
These generate $\HH_1(\Sigma_{g+1}^1;\Z/\ell)$, so we conclude that $T_{\partial}^{-n} f \in \hMod_g^2(\ell)$.
\end{proof}

\begin{corollary}
\label{corollary:hatcap}
Let $g \geq 1$ and $\ell \geq 2$, and let $\partial$ be a boundary component of
$\Sigma_g^2$.  Let $V$ be a $\Q$-vector space equipped with an action of $\Mod_g^2(\ell)$ such that
$T_{\partial}$ acts trivially on $V$.  Then $\HH_k(\hMod_g^2(\ell);V) \cong \HH_k(\Mod_g^2(\ell);V)$
for all $k \geq 0$.
\end{corollary}

The proof of this corollary uses the following standard lemma, which follows
from the existence of the transfer map and will be used many times in this paper.

\begin{lemma}[{Transfer map lemma, see, e.g., \cite[\S III.9]{BrownCohomology}}]
\label{lemma:transfer}
Let $G$ be a finite-index subgroup of $\Gamma$.  For a field $\bbk$ of characteristic $0$,
let $V$ be a $\bbk$-vector space equipped with an action of $\Gamma$.  Then for all $k$,
the map $\HH_k(G;V) \rightarrow \HH_k(\Gamma;V)$ is surjective.  If $G$ is also a normal
subgroup of $\Gamma$, then $\Gamma$ acts on $\HH_k(G;V)$ and $\HH_k(\Gamma;V) = \HH_k(G;V)_{\Gamma}$,
where the subscript means we are taking the $\Gamma$-coinvariants.\footnote{The action
of $\Gamma$ on $\HH_k(G;V)$ factors through $Q = \Gamma/G$, and sometimes it will
be more convenient to write this as $\HH_k(G;V)_Q$.}
\end{lemma}

\begin{proof}[Proof of Corollary \ref{corollary:hatcap}]
Since $T_{\partial} \in \Mod_g^2(\ell)$ is central and $T_{\partial}^{\ell}$ is the smallest power of $T_{\partial}$ 
lying in $\hMod_g^2(\ell)$, Lemma \ref{lemma:justboundary} implies that $\hMod_g^2(\ell)$ is a finite-index normal subgroup
of $\Mod_g^2(\ell)$ and that $\Mod_g^2(\ell) / \hMod_g^2(\ell)$ is a cyclic group of order $\ell$ generated by
the image of $T_{\partial}$.  By the transfer map lemma (Lemma \ref{lemma:transfer}), we deduce that
\[\HH_k(\Mod_g^2(\ell);V) \cong \HH_k(\hMod_g^2(\ell);V)_{\Mod_g^2(\ell)}.\]
Since inner automorphisms act
trivially on homology (see, e.g., \cite[Proposition III.8.1]{BrownCohomology}), the group $\hMod_g^2(\ell)$
acts trivially on $\HH_k(\hMod_g^2(\ell);V)$.  Since $T_{\partial}$ is a central element of $\Mod_g^2(\ell)$
acting trivially on $V$, it also acts trivially on $\HH_k(\hMod_g^2(\ell);V)$.  We conclude that
$\Mod_g^2(\ell)$ acts trivially on $\HH_k(\hMod_g^2(\ell);V)$, so $\HH_k(\hMod_g^2(\ell);V) \cong \HH_k(\Mod_g^2(\ell);V)$.
\end{proof}

\section{Ordered simplicial complexes and coefficient systems}
\label{section:semisimplicial}

This section contains some topological background material needed to describe the twisted homological
stability machine of \cite{PutmanTwistedStability} in the next section \S \ref{section:stabilitymachine}.  
We will not need to be as general as \cite{PutmanTwistedStability},
so instead of stating things in terms of semisimplicial sets we will work in the easier category
of ``ordered simplicial complexes''.

\subsection{Ordered simplicial complexes}
An {\em ordered simplicial complex} is a CW complex $\bbX$ whose cells are simplices such that
the following hold:
\begin{itemize}
\item The vertices $\bbX^0$ are an arbitrary discrete set.
\item For $p \geq 0$, the set $\bbX^p$ of $p$-simplices consists of
certain ordered sequences $\sigma = [v_0,\ldots,v_p]$,
with the $v_i$ distinct vertices.  The faces of $\sigma$ are obtained by deleting some of the
$v_i$, so for instance the codimension-$1$ faces are of the form 
$[v_0,\ldots,\widehat{v_i},\ldots,v_p]$.
\end{itemize}
These differ from ordinary simplicial complexes in two ways:
\begin{itemize}
\item The vertices making up a simplex have an order.
\item There can be up to $(p+1)!$ simplices of dimension $p$ with the same set of vertices, corresponding
to different orderings.  For instance, there might be distinct edges $[v_0,v_1]$ and
$[v_1,v_0]$ between vertices $v_0$ and $v_1$.
\end{itemize}
A group $G$ acting on $\bbX$ is required
to respect the ordering of the vertices on a simplex, so 
if $\sigma = [v_0,\ldots,v_p]$ is a $p$-simplex
and $g \in G$, then $g \sigma = [g v_0,\ldots,g v_p]$.

\subsection{Cohen--Macaulay complexes}

Let $\bbX$ be an ordered simplicial complex.
For simplices $\sigma = [v_0,\ldots,v_p]$ and $\tau = [w_0,\ldots,w_{q}]$ of $\bbX$,
let 
\[\sigma \cdot \tau = [v_0,\ldots,v_p,w_0,\ldots,w_{q}].\]
This might not be a simplex.
The {\em forward link} of a $p$-simplex $\sigma$ of $\bbX$,
denoted $\FLink_{\bbX}(\sigma)$, is the ordered simplicial complex whose $q$-simplices
are $q$-simplices $\tau$ of $\bbX$ such that $\sigma \cdot \tau$ is a $(p+q+1)$-simplex.
We say that $\bbX$ is {\em weakly forward Cohen--Macaulay} of dimension $n$
if $\bbX$ is $(n-1)$-connected and for all $p$-simplices $\sigma$ of $\bbX$, the forward
link $\FLink_{\bbX}(\sigma)$ is $(n-p-2)$-connected.\footnote{For this, we must
decide how to handle the empty set.  We say that a space $X$ is {\em $n$-connected} if for all $k \leq n$, all continuous maps
$S^k \rightarrow X$ can be extended to $\mathbb{D}^{k+1}$.  We have
$S^{-1} = \emptyset$ but $\mathbb{D}^0 = \{\text{pt}\}$, so for $n \geq -1$ a space
that is $n$-connected must be nonempty.  However, all spaces 
are $n$-connected for $n \leq -2$.}

\subsection{Coefficient systems}
Let $\bbk$ be a commutative ring.  Our next goal is to define coefficient 
systems on ordered simplicial complexes $\bbX$, which informally are natural
associations of $\bbk$-modules to each simplex.  Let $\Simp(\bbX)$ be the poset of simplices of $\bbX$ and let
$\tSimp(\bbX)$ be the poset obtained by adjoining an initial object $\emptysimp$
to $\Simp(\bbX)$.  We will call $\emptysimp$ the $(-1)$-simplex of $\bbX$.

A {\em coefficient system} over $\bbk$ on an ordered simplicial complex $\bbX$ is a contravariant functor
$\cF$ from $\Simp(\bbX)$ to the category of $\bbk$-modules.  We will frequently
omit the $\bbk$ and just talk about coefficient systems on $\bbX$. 
Unpacking this, $\cF$ consists of the following data:
\begin{itemize}
\item For each simplex $\sigma$ of $\bbX$, a $\bbk$-module $\cF(\sigma)$.
\item For each simplex $\sigma$ and each face $\sigma'$ of $\sigma$, 
a $\bbk$-module morphism $\cF(\sigma) \rightarrow \cF(\sigma')$.
\end{itemize}
These must satisfy the evident compatibility conditions. Similarly, an {\em augmented coefficient system} on
$\bbX$ is a contravariant functor $\cF$ from $\tSimp(\bbX)$ to the category of $\bbk$-modules.
The collection of coefficient systems (resp.\ augmented coefficient systems) over $\bbk$ on $\bbX$ forms an 
abelian category whose morphisms are natural transformations.

\begin{notation}
For a simplex $[v_0,\ldots,v_p]$, we will denote $\cF([v_0,\ldots,v_p])$ by $\cF[v_0,\ldots,v_p]$.
In particular, the value of $\cF$ on the $(-1)$-simplex $\emptysimp$ will be
written as $\cF\emptysimp$.
\end{notation}

\begin{example}
We can define a constant coefficient system $\ubbk$ on $\bbX$ with
$\ubbk(\sigma) = \bbk$ for all simplices $\sigma$.  This can
be extended to an augmented coefficient system by setting $\ubbk\emptysimp = \bbk$ for the $(-1)$-simplex $\emptysimp$.
\end{example}

\subsection{Homology}
Let $\bbX$ be an ordered simplicial complex and let
$\cF$ be a coefficient system on $\bbX$.  Define the
simplicial chain complex of $\bbX$ with coefficients in $\cF$ to be the chain complex
$\CC_{\bullet}(\bbX;\cF)$ defined as follows:
\begin{itemize}
\item For $p \geq 0$, we have
\[\CC_p(\bbX;\cF) = \bigoplus_{\sigma \in \bbX^p} \cF(\sigma).\]
\item The boundary map $d\colon \CC_p(\bbX;\cF) \rightarrow \CC_{p-1}(\bbX;\cF)$ is
$d = \sum_{i=0}^p (-1)^i d_i$, where the map $d_i\colon \CC_p(\bbX;\cF) \rightarrow \CC_{p-1}(\bbX;\cF)$ is
as follows.  Consider $\sigma \in \bbX^p$.  Write $\sigma = [v_0,\ldots,v_p]$, and
let $\sigma_i = [v_0,\ldots,\widehat{v_i},\ldots,v_p]$.  Then on the $\cF(\sigma)$ factor of $\CC_n(\bbX;\cF)$,
the map $d_i$ is
\[\cF(\sigma) \longrightarrow \cF(\sigma_i) \hookrightarrow \bigoplus_{\sigma' \in \bbX^{p-1}} \cF(\sigma') = \CC_{p-1}(\bbX;\cF).\]
\end{itemize}
Define
\[\HH_k(\bbX;\cF) = \HH_k(\CC_{\bullet}(\bbX;\cF)).\]
For an augmented coefficient system $\cF$ on $\bbX$, define
$\RC_{\bullet}(\bbX;\cF)$ to be the augmented chain complex defined just like we did above but with
$\RC_{-1}(\bbX;\cF) = \cF\emptysimp$ and define
\[\RH_k(\bbX;\cF) = \HH_k(\RC_{\bullet}(\bbX;\cF)).\]

\begin{example}
For the constant coefficient system $\ubbk$, the homology groups $\HH_k(\bbX;\ubbk)$ and
$\RH_k(\bbX;\ubbk)$ agree with the usual simplicial homology groups of $\bbX$.
\end{example}

\begin{remark}
With our definition, $\RH_{-1}(\bbX;\cF)$ is a quotient of $\cF\emptysimp$.  This quotient can sometimes be
nonzero.  It vanishes precisely when the map
\[\bigoplus_{v \in \bbX^0} \cF[v] \rightarrow \cF\emptysimp\]
is surjective.
\end{remark}

\subsection{Equivariant coefficient systems}
\label{section:equivariant}

Let $\bbX$ be an ordered simplicial complex, let $G$ be a group acting on $\bbX$, and
let $\cF$ be an augmented coefficient system on $\bbX$.
We want to equip $\cF$ with an ``action'' of 
$G$ that is compatible with the $G$-action on $\bbX$.  For simplicity,\footnote{This avoids
the complicated definition in terms of natural transformations from \cite{PutmanTwistedStability}.} 
we will restrict ourselves to $\cF$ such that for all $\sigma,\sigma' \in \tSimp(\bbX)$
with $\sigma' \subset \sigma$, the map $\cF(\sigma) \rightarrow \cF(\sigma')$ is injective.  We
call these {\em injective augmented coefficient systems}.  For such an $\cF$,
the map $\cF(\sigma) \rightarrow \cF\emptysimp$ is injective 
for all $\sigma \in \tSimp(\bbX)$, so we can regard $\cF(\sigma)$ as a
submodule of $\cF\emptysimp$.

A {\em $G$-equivariant} injective augmented coefficient system on $\bbX$ is an injective
augmented coefficient system $\cF$ along with an action of $G$ on $\cF\emptysimp$ such that
for all $\sigma \in \tSimp(\bbX)$, we have
\[g \cF(\sigma) = \cF(g \cdot \sigma) \quad \text{for all $g \in G$}.\]
Here we are regarding $\cF(\sigma)$ as a submodule of $\cF\emptysimp$, so
$g \cF(\sigma)$ is the image of $\cF(\sigma)$ under the action of $g$ on $\cF\emptysimp$.
Letting $G_{\sigma}$ be the stabilizer of $\sigma$, this implies that the
action of $G$ on $\cF\emptysimp$ restricts to an action of $G_{\sigma}$ on $\cF(\sigma)$.

\begin{example}
\label{example:equivariant}
Let $\bbX$ be an ordered simplicial complex with vertex set $V = \bbX^0$.  For a set
$S$, write $\bbk[S]$ for the free $\bbk$-module with basis $S$.  We can then define
an injective augmented coefficient system $\cF$ on $\bbX$ via the formulas
\[\cF[v_0,\ldots,v_p] = \bbk[V \setminus \{v_0,\ldots,v_p\}] \quad \text{and} \quad \cF\emptysimp = \bbk[V].\]
If a group $G$ acts on $\bbX$, then its action on $V$ induces an action on $\cF\emptysimp$, making
$\cF$ into a $G$-equivariant injective augmented coefficient system.
\end{example}

\section{The stability machine}
\label{section:stabilitymachine}

We now discuss some aspects of the homological stability machine with twisted coefficients from \cite{PutmanTwistedStability}.  There is an earlier approach to this due
to Dwyer \cite{DwyerTwisted}, but it seems hard to use it to prove our theorems.

\subsection{Motivation}
What we need is not the homological stability machine itself, but a result that encapsulates
one part of how the inductive step in the machine works.  Consider a group
$G$ acting on an ordered simplicial complex $\bbX$.  The goal is to relate the homology of $G$ to the homology
of stabilizers of simplices of $\bbX$.  The most basic thing one might want is that $\HH_k(G)$ is
``carried'' on the vertex stabilizers in the sense that the map
\begin{equation}
\label{eqn:gsurjective}
\bigoplus_{v \in \bbX^0} \HH_k(G_v) \longrightarrow \HH_k(G)
\end{equation}
is surjective.  For $v \in \bbX^0$ and $g \in G$, we have $g G_v g^{-1} = G_{g v}$, so since inner
automorphisms act trivially on homology the images of $\HH_k(G_v)$ and $\HH_k(G_{g v})$ in $\HH_k(G)$
are equal.  Thus for the sake of verifying surjectivity the above direct sum can be taken to be over representatives
of the $G$-orbits of $\bbX^0$.  

In a typical homological stability proof, the group $G$ acts
transitively on the vertices of $\bbX$ and there is a vertex $v_0$ such that $G_{v_0}$ is the previous
group in our sequence of groups.  In that case, if \eqref{eqn:gsurjective} is surjective then
the map $\HH_k(G_{v_0}) \rightarrow \HH_k(G)$ is surjective, which
is a weak form of homological stability.

\subsection{Fragment of machine}
In our situation, the group $G$ will {\em not} act transitively on the vertices of $\bbX$.  Moreover,
we want to incorporate twisted coefficients, which we do using a $G$-equivariant coefficient system
on $\bbX$.  The result we need is as follows:

\begin{proposition}
\label{proposition:stabilitymachine}
Let $G$ be a group acting on an ordered simplicial complex $\bbX$ and
let $\cM$ be a $G$-equivariant augmented coefficient system on $\bbX$.  For some $k \geq 0$, assume
that the following hold:
\begin{itemize}
\item[(i)]   We have $\RH_i(\bbX;\cM) = 0$ for $-1 \leq i \leq k-1$.
\item[(ii)]  We have $\RH_i(\bbX/G) = 0$ for $-1 \leq i \leq k$.
\item[(iii)] Let $\sigma$ be a simplex of $\bbX$.  Then for $i \geq 1$ the map
\[\HH_{k-i}(G_{\sigma};\cM(\sigma)) \longrightarrow \HH_{k-i}(G;\cM\emptysimp)\]
is an isomorphism if $i-1 \leq \dim(\sigma) \leq i+1$.
\end{itemize}
Then the map
\[\bigoplus_{v \in \bbX^0} \HH_k(G_v;\cM[v]) \longrightarrow \HH_k(G;\cM\emptysimp)\]
is a surjection.
\end{proposition}
\begin{proof}
This can be proved using the spectral sequence \cite[Theorem 5.6]{PutmanTwistedStability} exactly like \cite[Theorem 5.8]{PutmanTwistedStability}.
\end{proof}

\subsection{Vanishing theorem}
For Proposition \ref{proposition:stabilitymachine} to be useful, we need a way to 
verify its first hypothesis, which says that $\RH_i(\bbX;\cM) = 0$ in some range.  
The paper \cite{PutmanTwistedStability} gives a criterion for this.  Letting $\bbX$
be an ordered simplicial complex, it applies to augmented coefficient systems $\cF$ on
$\bbX$ that are polynomial of degree $d \geq -1$.  This is defined
inductively in $d$ as follows:\footnote{The reference \cite{PutmanTwistedStability}
defines what it means to be polynomial of degree $d$ up to dimension $e$.  What
we define here corresponds to $e = \infty$.}
\begin{itemize}
\item A coefficient system $\cF$ is polynomial of degree $-1$ if $\cF(\sigma) = 0$ for all
simplices $\sigma$.  In particular, $\cF\emptysimp = 0$ for the $(-1)$-simplex $\emptysimp$.
\item A coefficient system $\cF$ is polynomial of degree $d \geq 0$ if it satisfies
the following two conditions:
\begin{itemize}
\item The coefficient system $\cF$ is injective in the sense of \S \ref{section:equivariant}.  Recall
that this means that if $\sigma$ is a simplex and $\sigma'$ is a face of $\sigma$, then the map
$\cF(\sigma) \rightarrow \cF(\sigma')$ is injective.
\item Let $w$ be a vertex of $\bbX$.  Let $D_w \cF$ be the coefficient
system on the forward link $\FLink_{\bbX}(w)$ defined by the formula
\[\quad\quad\quad\quad D_w \cF(\sigma) = \frac{\cF(\sigma)}{\Image\left(\cF\left(w \cdot \sigma\right) \rightarrow \cF\left(\sigma\right)\right)} \quad \text{for a simplex $\sigma$ of $\FLink_{\bbX}(w)$}.\]
Then $D_w \cF$ must be polynomial of degree $d-1$.
\end{itemize}
\end{itemize}

\begin{example}
A coefficient system $\cF$ is polynomial of degree $0$ if and only
if it is constant.
\end{example}

\begin{example}
Let $\bbX$ be an ordered simplicial complex with vertex set $V = \bbX^0$.  Let $\cF$
be the augmented coefficient system on $\bbX$ from Example \ref{example:equivariant}, so
\[\cF[v_0,\ldots,v_p] = \bbk[V \setminus \{v_0,\ldots,v_p\}] \quad \text{and} \quad \cF\emptysimp = \bbk[V].\]
We claim that $\cF$ is polynomial of degree $1$.  Since $\cF$ is injective, for all $w \in V$ we must
prove that $D_w \cF$ is polynomial of degree $0$, i.e., constant.  For a simplex
$[v_0,\ldots,v_p]$ of the forward link of $w$, we have
\begin{align*}
D_w \cF[v_0,\ldots,v_p] &= \frac{\bbk[V \setminus \{v_0,\ldots,v_p\}]}{\bbk[V \setminus \{w,v_0,\ldots,v_p\}]}
\cong \bbk[\{w\}] \cong \bbk.
\end{align*}
Thus $D_w \cF \cong \ubbk$, as desired.
\end{example}

The vanishing theorem from \cite{PutmanTwistedStability} is then as follows:

\begin{theorem}[{Vanishing theorem, \cite[Theorem 6.4]{PutmanTwistedStability}\footnote{The statement of \cite[Theorem 6.4]{PutmanTwistedStability} requires $\cF$ to be polynomial of degree $d$ up to dimension $N$.  As we said when we defined polynomiality, what we defined is being polynomial of degree $d$ up to dimension $\infty$, so the ``up to dimension'' part of \cite[Theorem 6.4]{PutmanTwistedStability} is superfluous.}}]
\label{theorem:vanishing}
For some $N \geq -1$ and $d \geq -1$, let $\bbX$ be an ordered simplicial complex 
that is weakly forward Cohen--Macaulay of dimension $N+d+1$. 
Let $\cF$ be an augmented coefficient system on $\bbX$ that is 
polynomial of degree $d$. Then $\RH_i(\bbX;\cF)=0$ for $-1 \leq i \leq N$.
\end{theorem}

\subsection{Strong polynomiality}
\label{section:strongpolynomial}

Let $\bbX$ be an ordered simplicial complex.  
Our next goal (accomplished below in \S \ref{section:tensorproduct} after two sections of
preliminaries) is to study tensor products of polynomial augmented coefficient
systems on $\bbX$.
What we would like to prove is that if $\cF$ and $\cG$ are polynomial of degrees
$d \geq 0$ and $e \geq 0$, then $\cF \otimes \cG$ is polynomial of degree $d+e$.
For this to be true, we need some additional hypotheses. 

An augmented coefficient system $\cF$ is {\em strongly polynomial} of degree $d \geq -1$
if it satisfies the following inductive definition:
\begin{itemize}
\item It is strongly polynomial of degree $-1$ if $\cF(\sigma) = 0$ for all
simplices $\sigma$.  To simplify handling edge cases in inductive proofs,
we will also say such coefficient systems are strongly polynomial of degree
$d$ for all $d \leq -1$. 
\item It is strongly polynomial of degree $d \geq 0$ if it satisfies
the following two conditions:
\begin{itemize}
\item The coefficient system $\cF$ is injective in the sense of \S \ref{section:equivariant}.
\item Let $\tau = [w_0,\ldots,w_{q}]$ be a simplex of $\bbX$.  Set
$\tau' = [w_0,\ldots,w_{q-1}]$, interpreted as the empty $(-1)$-simplex
if $q = 0$.  Let $D_{\tau} \cF$ be the coefficient
system on the forward link $\FLink_{\bbX}(\tau)$ defined by the formula
\[\quad\quad\quad\quad D_{\tau} \cF(\sigma) = \frac{\cF(\tau' \cdot \sigma)}{\Image\left(\cF\left(\tau \cdot \sigma\right) \rightarrow \cF\left(\tau' \cdot \sigma\right)\right)} \quad \text{for a simplex $\sigma$ of $\FLink_{\bbX}(\tau)$}.\]
Then $D_{\tau} \cF$ must be strongly polynomial of degree $d-1$.
\end{itemize}
\end{itemize}

\begin{remark}
\label{remark:negativedegree}
If $\cF$ is strongly polynomial of degree $d$, then it is also strongly polynomial
of degree $d'$ for all $d' \geq d$.  In particular, if $\cF$ vanishes identically
then it is strongly polynomial of degree $d$ for all $d \in \Z$.
\end{remark}

\begin{remark}
This is stronger than simply being polynomial, whose definition only involves
$D_{\tau}$ for $0$-dimensional simplices $\tau = [w]$.
\end{remark}

\subsection{Insertion functors}
\label{section:ataupoly}

Let $\cF$ be an augmented coefficient system on an ordered simplicial complex $\bbX$.  For a simplex
$\tau$ of $\bbX$, let $A_{\tau} \cF$ be the augmented coefficient system on the forward
link $\bbL = \FLink_{\bbX}(\tau)$ defined via the formula
\[A_{\tau} \cF (\sigma) = \cF(\tau \cdot \sigma) \quad\quad \text{for a simplex $\sigma$ of $\bbL$}.\]
Write $\tau = [w_0,\ldots,w_{q}]$, and let $\tau' = [w_0,\ldots,w_{q-1}]$.
If $\cF$ is injective, we have a short exact sequence
\[0 \longrightarrow A_{\tau} \cF \longrightarrow \left(A_{\tau'} \cF\right)|_{\bbL} \longrightarrow D_{\tau} \cF \longrightarrow 0\]
of augmented coefficient systems on $\bbL$.  The following lemma implies that if $\cF$ is strongly
polynomial of degree $d$, then so are $A_{\tau} \cF$ and $(A_{\tau'} \cF)|_{\bbL}$.  This can fail
if $\cF$ is only polynomial.

\begin{lemma}
\label{lemma:ataupoly}
Let $\bbX$ be an ordered simplicial complex and let $\cF$ be an augmented coefficient system on $\bbX$
that is strongly polynomial of degree $d \geq -1$.  The following hold:
\begin{itemize}
\item[(i)] For all subcomplexes $\bbY$ of $\bbX$, the coefficient system $\cF|_{\bbY}$ is strongly
polynomial of degree $d$.
\item[(ii)] Let $\tau$ be a simplex of $\bbX$ and let $\bbL = \FLink_{\bbX}(\tau)$.  Then $A_{\tau} \cF$ is
strongly polynomial of degree $d$.
\end{itemize}
\end{lemma}
\begin{proof}
Both (i) and (ii) are trivial if $d=-1$, so we can assume that $d \geq 0$.
For (i), in \cite[Lemma 6.3]{PutmanTwistedStability} it is proved that if
$\cF$ is assumed merely to be polynomial of degree $d$, then so is $\cF|_{\bbY}$.  The
same proof works for strong polynomiality.

For (ii), since $\cF$ is injective, so is $A_{\tau} \cF$.  Also, if
$\kappa$ is a simplex of $\bbL$, then $D_{\kappa} A_{\tau} \cF = D_{\tau \cdot \kappa} \cF$.
Since $\cF$ is strongly polynomial of degree $d$, this is strongly polynomial of degree
$(d-1)$.  Together, these two observations imply that $A_{\tau} \cF$ is strongly
polynomial of degree $d$.
\end{proof}

\subsection{Filtrations of coefficient systems}

It is clear that the collection of augmented coefficient systems that are strongly polynomial
is closed under direct sums.
More generally, we have the following.  To make
its statement easier to parse, we only state it for strongly polynomial coefficient
systems, but it also holds for polynomial ones with a similar proof.

\begin{lemma}
\label{lemma:filtrations}
Let $\bbX$ be an ordered simplicial complex and let $\cF$ be an augmented
coefficient system on $\bbX$.  Assume that $\cF$ has a filtration
\[\cF = \cF_r \supset \cF_{r-1} \supset \cdots \supset \cF_0 = 0\]
such that $\cF_i / \cF_{i+1}$ is strongly polynomial of degree $d \geq -1$ for all
$1 \leq i \leq r$.  Then $\cF$ is strongly polynomial of degree $d$.
\end{lemma}
\begin{proof}
The proof is by induction on $d$.  The base case $d=-1$ is clear, so
assume that $d \geq 0$ and that the lemma is true for
degree $d-1$.  Using another induction on the length of a filtration,
we see that it is enough to prove that if
\[0 \longrightarrow \cK \longrightarrow \cF \longrightarrow \cQ \longrightarrow 0\]
is a short exact sequence of augmented coefficient systems such that $\cK$ and
$\cQ$ are strongly polynomial of degree $d$, then $\cF$ is strongly polynomial of degree $d$.

To see that $\cF$ is injective, let $\sigma$ be a simplex and $\sigma'$ be a face
of $\sigma$.  We then have a commutative diagram
\[\begin{tikzcd}[row sep=small]
0 \arrow{r} & \cK(\sigma)  \arrow{r} \arrow{d} & \cF(\sigma)  \arrow{r} \arrow{d} & \cQ(\sigma)  \arrow{r} \arrow{d} & 0 \\
0 \arrow{r} & \cK(\sigma') \arrow{r}           & \cF(\sigma') \arrow{r}           & \cQ(\sigma') \arrow{r}           & 0
\end{tikzcd}\]
with exact rows.
The maps $\cK(\sigma) \rightarrow \cK(\sigma')$ and $\cQ(\sigma) \rightarrow \cQ(\sigma')$
are injective by assumption, so by the five-lemma $\cF(\sigma) \rightarrow \cF(\sigma')$ is
also injective, as desired.

Now consider a simplex $\tau = [w_0,\ldots,w_{q}]$ of $\bbX$.  Set
$\tau' = [w_0,\ldots,w_{q-1}]$.  We know that $D_{\tau} \cK$ and $D_{\tau} \cQ$ are strongly polynomial
of degree $d-1$, and we must prove that $D_{\tau} \cF$ is as well.  For a simplex $\sigma$
of $\FLink_{\bbX}(\tau)$, we have a commutative diagram
\[\begin{tikzcd}[row sep=small]
            & 0 \arrow{d}                                 & 0 \arrow{d}                                 & 0 \arrow{d}                                 &   \\
0 \arrow{r} & \cK(\tau \cdot \sigma)  \arrow{r} \arrow{d} & \cF(\tau \cdot \sigma)  \arrow{r} \arrow{d} & \cQ(\tau \cdot \sigma)  \arrow{r} \arrow{d} & 0 \\
0 \arrow{r} & \cK(\tau' \cdot \sigma) \arrow{r} \arrow{d} & \cF(\tau' \cdot \sigma) \arrow{r} \arrow{d} & \cQ(\tau' \cdot \sigma) \arrow{r} \arrow{d} & 0 \\
0 \arrow{r} & D_{\tau} \cK(\sigma)    \arrow{r} \arrow{d} & D_{\tau} \cF(\sigma)    \arrow{r} \arrow{d} & D_{\tau} \cQ(\sigma)    \arrow{r} \arrow{d} & 0 \\
            & 0                                           & 0                                           & 0                                           &
\end{tikzcd}\]
The columns are exact since $\cK$ and $\cF$ and $\cQ$ are injective, and the first two
rows are also exact by assumption.  A quick diagram chase (or alternatively,
the snake lemma) shows that the third row is
also exact.  This implies that we have a short exact sequence
\[0 \longrightarrow D_{\tau} \cK \longrightarrow D_{\tau} \cF \longrightarrow D_{\tau} \cQ \longrightarrow 0\]
of augmented coefficient systems.  Since $D_{\tau} \cK$ and $D_{\tau} \cQ$ are strongly polynomial
of degree $d-1$, our inductive hypothesis implies that $D_{\tau} \cF$ is as well, as desired.
\end{proof}

\subsection{Tensor products of coefficient systems}
\label{section:tensorproduct}

Using Lemma \ref{lemma:filtrations}, we will prove the following.  We will apply
it with $\bbk$ a field, in which case its flatness assumptions are automatic.
When $d+e < -1$, the statement of this lemma uses the conventions about
strong polynomiality in negative degrees from \S \ref{section:strongpolynomial} (c.f.
Remark \ref{remark:negativedegree}).

\begin{lemma}
\label{lemma:tensorpolynomial}
Let $\bbX$ be an ordered simplicial complex and 
let $\cF$ and $\cG$ be augmented coefficient systems on $\bbX$ over a commutative
ring $\bbk$.  Assume the following hold:
\begin{itemize}
\item For all simplices $\sigma$ of $\bbX$, the $\bbk$-modules $\cF(\sigma)$
and $\cG(\sigma)$ are flat.
\item The augmented coefficient systems $\cF$ and $\cG$ are strongly polynomial
of degrees $d \geq -1$ and $e \geq -1$, respectively.
\end{itemize}
Then $\cF \otimes \cG$ is strongly polynomial of degree $d+e$.
\end{lemma}
\begin{proof}
The proof will be by induction on $d$ and $e$.  The base cases are when
either $d = -1$ or $e = -1$ (or both).  In other words, at least one of
$\cF$ and $\cG$ is identically $0$.  Therefore $\cF \otimes \cG$
is also identically $0$, and is thus strongly polynomial of degree
$-1$.  This implies that it is strongly polynomial of any degree
whatsoever, and in particular is strongly polynomial of degree $d+e$.

Assume now that $d \geq 0$ and $e \geq 0$, and that the lemma
is true whenever one of them is smaller.  
We first prove that $\cF \otimes \cG$ is an injective augmented coefficient system.  Let $\sigma$
be a simplex and let $\sigma'$ be a face of $\sigma$.  The map
$(\cF \otimes \cG)(\sigma) \rightarrow (\cF \otimes \cG)(\sigma')$ can be factored
as
\[(\cF \otimes \cG)(\sigma) = \cF(\sigma) \otimes \cG(\sigma) \rightarrow \cF(\sigma') \otimes \cG(\sigma) \rightarrow \cF(\sigma') \otimes \cG(\sigma') = (\cF \otimes \cG)(\sigma').\]
The first arrow is injective since $\cF$ is injective and $\cG(\sigma)$ is flat, and the second arrow
is injective since $\cF(\sigma')$ is flat and $\cG$ is injective.  It follows that
the map $(\cF \otimes \cG)(\sigma) \rightarrow (\cF \otimes \cG)(\sigma')$ is injective,
as desired.

Now consider a simplex $\tau = [w_0,\ldots,w_{q}]$ of $\bbX$.  Set $\tau' = [w_0,\ldots,w_{q-1}]$.
We must prove that the augmented coefficient
system $D_{\tau} (\cF \otimes \cG)$ on $\bbL = \FLink_{\bbX}(\tau)$ is strongly polynomial of degree $d+e-1$.
Using the notation from \S \ref{section:ataupoly}, we have short exact sequences of augmented coefficient
systems
\[0 \longrightarrow A_{\tau} \cF \longrightarrow (A_{\tau'} \cF)|_{\bbL} \longrightarrow D_{\tau} \cF \longrightarrow 0\]
and
\[0 \longrightarrow A_{\tau} \cG \longrightarrow (A_{\tau'} \cG)|_{\bbL} \longrightarrow D_{\tau} \cG \longrightarrow 0\]
on $\bbL$.  By Lemma \ref{lemma:ataupoly}, the augmented coefficient systems $A_{\tau} \cF$ and $(A_{\tau'} \cF)|_{\bbL}$
(resp.\ $A_{\tau} \cG$ and $(A_{\tau'} \cG)|_{\bbL}$) are strongly polynomial of degree $d$ (resp.\ $e$).
Using our flatness assumptions, we have a filtration
\[0 
\subset 
\left(A_{\tau} \cF\right) \otimes \left(A_{\tau} \cG\right)
\subset
(A_{\tau'} \cF)|_{\bbL} \otimes \left(A_{\tau} \cG\right)
\subset
(A_{\tau'} \cF)|_{\bbL} \otimes (A_{\tau'} \cG)|_{\bbL}\]
of coefficient systems.  The associated graded of this filtration consists of the following:
\begin{itemize}
\item $\left(A_{\tau} \cF\right) \otimes \left(A_{\tau} \cG\right)$.  Since $A_{\tau} \cF$ is strongly polynomial
of degree $d$ and $A_{\tau} \cG$ is strongly polynomial of degree $e$, we cannot apply our inductive
hypothesis to this (but we will soon quotient it out, so this will not matter).
\item $\left(D_{\tau} \cF\right) \otimes \left(A_{\tau} \cG\right)$.  Since $D_{\tau} \cF$ is strongly polynomial
of degree $(d-1)$ and $A_{\tau} \cG$ is strongly polynomial of degree $e$, our inductive hypothesis
says that this is strongly polynomial of degree $d+e-1$.
\item $(A_{\tau'} \cF)|_{\bbL} \otimes \left(D_{\tau} \cG\right)$.  Since $(A_{\tau'} \cF)|_{\bbL}$ is strongly polynomial of degree $d$
and $D_{\tau} \cG$ is strongly polynomial of degree $e-1$, our inductive hypothesis says that
this is strongly polynomial of degree $d+e-1$.
\end{itemize}
From this, we see that
\[D_{\tau} \left(\cF \otimes \cG\right) = \left((A_{\tau'} \cF)|_{\bbL} \otimes (A_{\tau'} \cG)|_{\bbL}\right) / \left(\left(A_{\tau} \cF\right) \otimes \left(A_{\tau} \cG\right)\right)\]
has a filtration whose associated graded terms are strongly polynomial of degree $d+e-1$.  By
Lemma \ref{lemma:filtrations}, we deduce that $D_{\tau} \left(\cF \otimes \cG\right)$
is strongly polynomial of degree $d+e-1$, as desired.
\end{proof}

\section{The complex of tethered tori}
\label{section:tetheredtori}

In this section, we introduce an ordered simplicial complex upon which $\Mod_{g}^1(\ell)$ acts and study
its basic properties.  In the next section, we will introduce a $\Mod_g^1(\ell)$-equivariant
coefficient system on it and prove it is strongly polynomial.

\subsection{Tori and tethered tori.}
Let $\tau(\Sigma_1^1)$ be the result of gluing $[0,1]$ to $\Sigma_1^1$ by identifying $1 \in [0,1]$ with a
point of $\partial \Sigma_1^1$.  The subset $[0,1] \subset \tau(\Sigma_1^1)$ will be called the {\em tether}
and the point $0 \in [0,1] \subset \tau(\Sigma_1^1)$ will be called the {\em initial point} of the tether.
For an open interval $I \subset \partial \Sigma_{g}^1$, an {\em $I$-tethered torus} in $\Sigma_{g}^1$ is
an embedding $\iota\colon \tau(\Sigma_1^1) \rightarrow \Sigma_{g}^1$ taking the initial point of the tether
to a point of $I$ such that the restriction of $\iota$ to $\Sigma_1^1$ is orientation-preserving:\\
\centerline{\psfig{file=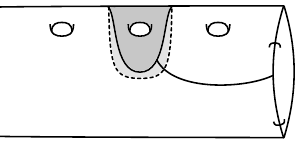,scale=1}}
We will always consider $I$-tethered tori up to isotopy.\footnote{These are isotopies through $I$-tethered
tori, so the initial point of the tether can move within $I$.}
An $I$-tethered torus $\iota\colon \tau(\Sigma_1^1) \rightarrow \Sigma_{g}^1$ is said to be {\em orthogonal}
to a symplectic subgroup $H \subset \HH_1(\Sigma_{g}^1;\Z/\ell)$
if all elements of the image of
\[\HH_1(\Sigma_1^1;\Z/\ell) \stackrel{\cong}{\longrightarrow} \HH_1(\tau(\Sigma_1^1);\Z/\ell) \stackrel{\iota_{\ast}}{\longrightarrow} \HH_1(\Sigma_{g}^1;\Z/\ell)\]
are orthogonal to $H$ under the algebraic intersection pairing.

\subsection{Complex of tethered tori.}
\label{section:tetheredtoruscomplex}

Fix a symplectic subgroup $H \subset \HH_1(\Sigma_{g}^1;\Z/\ell)$ and an open interval $I \subset \partial \Sigma_{g}^1$.
The {\em complex of $I$-tethered $H$-orthogonal tori} in $\Sigma_{g}^1$, denoted $\bbTT_g^1(I,H)$, is the ordered
simplicial complex whose $p$-simplices are ordered sequences $[\iota_0,\ldots,\iota_p]$ as follows:
\begin{itemize}
\item Each $\iota_i$ is the isotopy class of an $I$-tethered torus that is orthogonal to $H$.
\item The $\iota_i$ can be isotoped so as to be disjoint.
\item The $\iota_i$ are ordered using the order in which their tethers leave $I$, which is oriented such that
the surface is to its right.
\end{itemize}
For instance, a $2$-simplex might look like this:\\
\centerline{\psfig{file=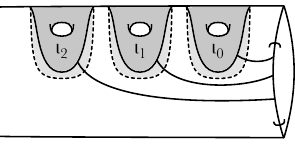,scale=1}}
If $H=0$, then we will sometimes omit it from our notation and simply write $\bbTT_g^1(I)$.
The complex $\bbTT_g^1(I)$ was introduced by Hatcher--Vogtmann \cite{HatcherVogtmannTethers}, who
proved that it was $\frac{g-3}{2}$-connected.  The author generalized this to
$\bbTT_g^1(I,H)$ as follows:

\begin{theorem}[{\cite[Theorem 3.8]{PutmanPartialTorelli}}]
\label{theorem:tetheredtoricon}
Fix $g,h \geq 0$ and $\ell \geq 2$.  Let $I$ be an open interval in $\partial \Sigma_{g}^1$ and let
$H$ be a genus-$h$ symplectic subgroup of $\HH_1(\Sigma_{g}^1;\Z/\ell)$.  Then
$\bbTT_g^1(I,H)$ is $\frac{g-(4h+3)}{2h+2}$-connected.
\end{theorem}

This has the following corollary:

\begin{corollary}
\label{corollary:tetheredtoricm}
Fix $g,h \geq 0$ and $\ell \geq 2$.  Let $I$ be an open interval in $\partial \Sigma_{g}^1$ and let
$H$ be a genus-$h$ symplectic subgroup of $\HH_1(\Sigma_{g}^1;\Z/\ell)$.  Then
$\bbTT_g^1(I,H)$ is weakly forward Cohen--Macaulay of dimension $\frac{g-(4h+3)}{2h+2}+1$.
\end{corollary}
\begin{proof}
Theorem \ref{theorem:tetheredtoricon} says that $\bbTT_g^1(I,H)$ is $\frac{g-(4h+3)}{2h+2}$-connected.
Let $\sigma = [\iota_0,\ldots,\iota_p]$ be a $p$-simplex and let $\bbL$ be the forward link of $\sigma$.  As
in the following figure, let $S$ be the result of deleting the interiors of the $\iota_i(\Sigma_1^1)$ from $\Sigma_{g}^1$
and then cutting open the resulting surface along the tethers:\\
\centerline{\psfig{file=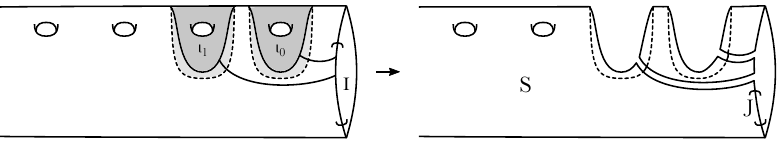,scale=1}}
We thus have $S \cong \Sigma_{g-p-1}^1$, and $\HH_1(S;\Z/\ell)$ can be identified with a subgroup of $\HH_1(\Sigma_{g}^1;\Z/\ell)$ containing
$H$.  Identify $S$ with $\Sigma_{g-p-1}^1$ and $H$ with a subgroup of $\HH_1(\Sigma_{g-p-1}^1;\Z/\ell)$.  Letting $J \subset \partial \Sigma_{g-p-1}^1$ be the interval indicated in
the above figure, the forward link $\bbL$ is isomorphic to $\bbTT_{g-p-1}^1(J,H)$, see here:\\
\centerline{\psfig{file=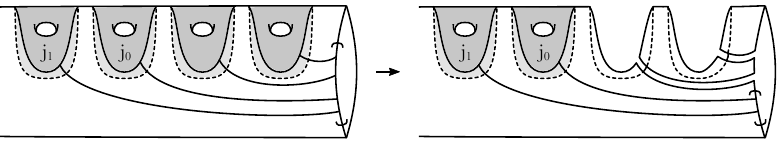,scale=1}}
It thus follows from Theorem \ref{theorem:tetheredtoricon} that $L$ is
\[\frac{(g-p-1)-(4h+3)}{2h+2} = \frac{g-(4h+3)}{2h+2} - \frac{p+1}{2h+2} \geq \frac{g-(4h+3)}{2h+2} - p - 1\]
connected, as desired.
\end{proof}

\subsection{Realizing symplectic bases}
We next describe the quotient of $\bbTT_g^1(I,H)$ by $\Mod_{g}^1(\ell)$,
which requires some preliminaries.  Let $\omega(-,-)$ be the algebraic intersection pairing on
$\HH_1(\Sigma_g^1;\Z/\ell)$.  A {\em symplectic basis} of $\HH_1(\Sigma_{g}^1;\Z/\ell)$ is a set of elements
$\{a_1,b_1,\ldots,a_g,b_g\}$ of $\HH_1(\Sigma_{g}^1;\Z/\ell)$ such that
\[\omega(a_i,a_j) = \omega(b_i,b_j) = 0 \quad \text{and} \quad \omega(a_i,b_j) = \delta_{ij} \quad \text{for $1 \leq i,j \leq g$}.\]
This implies that the set $\{a_1,b_1,\ldots,a_g,b_g\}$
is a basis of the free $\Z/\ell$-module $\HH_1(\Sigma_{g}^1;\Z/\ell)$.  
A {\em geometric realization} of a symplectic basis $\{a_1,b_1,\ldots,a_{g},b_{g}\}$
is a collection of oriented simple closed curves $\{\alpha_1,\beta_1,\ldots,\alpha_{g},\beta_{g}\}$ on
$\Sigma_{g}^1$ satisfying
\[[\alpha_i] = a_i \quad \text{and} \quad [\beta_i] = b_i \quad \text{for $1 \leq i \leq g$}\]
such that the curves $\{\alpha_1,\beta_1,\ldots,\alpha_{g},\beta_{g}\}$ are all pairwise disjoint
except that each $\alpha_i$ intersects $\beta_i$ exactly once.  See here:\\
\centerline{\psfig{file=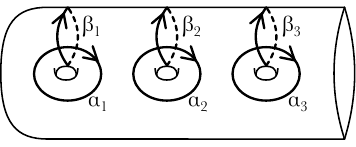,scale=1}}
The following lemma shows that geometric realizations always exist:

\begin{lemma}
\label{lemma:extendrealization}
Fix some $g \geq 0$ and $\ell \geq 2$.  Then every symplectic basis $\{a_1,b_1,\ldots,a_g,b_g\}$
of $\HH_1(\Sigma_{g}^1;\Z/\ell)$ has a geometric realization.
\end{lemma}
\begin{proof}
A similar statement was proved in \cite[Lemma A.3]{PutmanCutPasteTorelli} for
symplectic bases of $\HH_1(\Sigma_g;\Z) \cong \Z^{2g}$, and the same proof works for
$\HH_1(\Sigma_g^1;\Z/\ell)$.
\end{proof}

\subsection{Identifying the quotient}
Let $V$ be a free $\Z/\ell$-module equipped with a symplectic form $\omega(-,-)$.
Define $\bbSB(V)$ to be the ordered simplicial complex whose $(p-1)$-simplices are
ordered tuples $[(a_1,b_1),\ldots,(a_p,b_p)]$ with $a_i,b_j \in V$
such that
\[\omega(a_i,a_j) = \omega(b_i,b_j) = 0 \quad \text{and} \quad \omega(a_i,b_j) = \delta_{ij} \quad \text{for $1 \leq i,j \leq p$}.\]
We then have the following.

\begin{lemma}
\label{lemma:identifyquotient}
Fix some $g \geq 0$ and $\ell \geq 2$.  Let $I$ be an open interval in $\partial \Sigma_{g}^1$ and let
$H$ be a symplectic subgroup of $\HH_1(\Sigma_{g}^1;\Z/\ell)$.  Then
$\bbTT_g^1(I,H) / \Mod_{g}^1(\ell) \cong \bbSB(H^{\perp})$.
\end{lemma}
\begin{proof}
Fix simple closed oriented curves $A$ and $B$ in $\Sigma_1^1$ that intersect once with a positive sign.
For an $I$-tethered torus $\iota\colon \tau(\Sigma_1^1) \rightarrow \Sigma_{g}^1$ that is orthogonal
to $H$, we have oriented simple closed curves $\iota(A)$ and $\iota(B)$, and the tuple $([\iota(A)],[\iota(B)])$
of mod-$\ell$ homology classes is a vertex of $\bbSB(H^{\perp})$.  
Define a map of ordered simplicial complexes $\Psi\colon \bbTT_g^1(I,H) \rightarrow \bbSB(H^{\perp})$ as follows.
Consider a $(p-1)$-simplex
$[\iota_1,\ldots,\iota_p]$ of $\bbTT_g^1(I,H)$.  We then define
\[\Psi[\iota_1,\ldots,\iota_p] = [([\iota_1(A)],[\iota_1(B)]),\ldots,([\iota_p(A)],[\iota_p(B)])].\]
The map $\Psi$ is $\Mod_{g}^1(\ell)$-invariant, and to prove that the resulting map
\[\bbTT_g^1(I,H) / \Mod_{g}^1(\ell) \longrightarrow \bbSB(H^{\perp})\]
is an isomorphism it is enough to prove the following two facts.

\begin{claim}{1}
For all simplices $\sigma$ of $\bbSB(H^{\perp})$, there exists a simplex $\tau$ of $\bbTT_g^1(I,H)$ with
$\Psi(\tau) = \sigma$.
\end{claim}

Write $\sigma = [(a_1,b_1),\ldots,(a_p,b_p)]$.  The set $\{a_1,b_1,\ldots,a_p,b_p\}$ can be extended to a
symplectic basis for $\HH_1(\Sigma_{g}^1;\Z/\ell)$, and Lemma \ref{lemma:extendrealization} implies
that this symplectic basis has a geometric realization.  Throwing away some of the curves in this
geometric realization, we 
find simple closed oriented curves $\{\alpha_1,\beta_1,\ldots,\alpha_p,\beta_p\}$ on $\Sigma_{g}^1$
satisfying
\[[\alpha_i] = a_i \quad \text{and} \quad [\beta_i] = b_i \quad \text{for $1 \leq i \leq p$}\]
such that the curves $\{\alpha_1,\beta_1,\ldots,\alpha_{p},\beta_{p}\}$ are all pairwise disjoint
except that each $\alpha_i$ intersects $\beta_i$ exactly once.  As in the following
figure, we can then find a simplex $\tau = \{\iota_1,\ldots,\iota_p\}$ of
$\bbTT_g^1(I)$ such that $\iota_i(A) = \alpha_i$ and $\iota_i(B) = \beta_i$ for
$1 \leq i \leq p$:\\
\centerline{\psfig{file=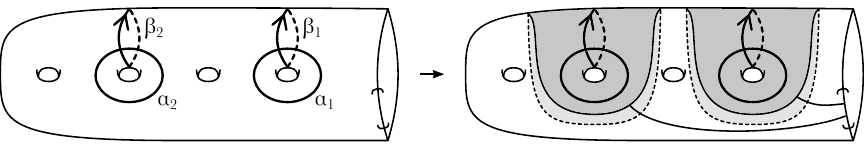,scale=1}}
Since the $a_i$ and $b_i$ all lie in $H^{\perp}$, the simplex $\tau$ lies
in $\bbTT_g^1(I,H)$, and by construction we have $\Psi(\tau) = \sigma$.

\begin{claim}{2}
For all simplices $\tau_1$ and $\tau_2$ of $\bbTT_g^1(I,H)$ such that $\Psi(\tau_1) = \Psi(\tau_2)$, there
exists some $f \in \Mod_{g}^1(\ell)$ such that $f(\tau_1) = \tau_2$.
\end{claim}

The dimensions of $\tau_1$ and $\tau_2$ are the same, say $(p-1)$.
For $r=1,2$ let $\tau_r = [\iota^r_1,\ldots,\iota^r_p]$.  Write
\[\Psi(\tau_1) = \Psi(\tau_2) = [(a_1,b_1),\ldots,(a_p,b_p)].\]
We can extend $\{a_1,b_1,\ldots,a_p,b_p\}$ to a symplectic basis
$\{a_1,b_1,\ldots,a_g,b_g\}$ for $\HH_1(\Sigma_{g}^1;\Z/\ell)$.
For $r=1,2$ let $S_r$ be the result of deleting the interiors of the $\iota^r_i(\Sigma_1^1)$ from $\Sigma_{g}^1$
and then cutting open the resulting surface along the tethers:\\
\centerline{\psfig{file=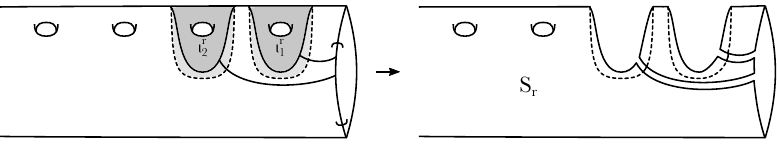,scale=1}}
We thus have $S_r \cong \Sigma_{g-p}^1$.  Identifying $S_r$ with a subsurface of $\Sigma_{g}^1$ in the obvious way
identifies $\HH_1(S_r;\Z/\ell)$ with a subgroup of $\HH_1(\Sigma_{g}^1;\Z/\ell)$, and
$\{a_{p+1},b_{p+1},\ldots,a_g,b_g\}$ is a symplectic basis for $\HH_1(S_r;\Z/\ell)$.  By Lemma
\ref{lemma:extendrealization}, we can geometrically realize this with
curves $\{\alpha^r_{p+1},\beta^r_{p+1},\ldots,\alpha^r_g,\beta^r_g\}$ lying in $S_r$.
Using the ``change of coordinates'' principle from \cite[\S 1.3.2]{FarbMargalitPrimer}, we can
find some $f \in \Mod_{g}^1$ with the following properties:
\begin{itemize}
\item $f(\iota^1_i) = \iota^2_i$ for $1 \leq i \leq p$.  In particular, $f$ fixes
$a_i \in \HH_1(\Sigma_{g}^1;\Z/\ell)$ and $b_i \in \HH_1(\Sigma_{g}^1;\Z/\ell)$ for $1 \leq i \leq p$.
\item $f(\alpha^1_i) = \alpha^2_i$ and $f(\beta^1_i) = \beta^2_i$ for $p+1 \leq i \leq g$.
In particular, $f$ fixes $a_i \in \HH_1(\Sigma_{g}^1;\Z/\ell)$ and $b_i \in \HH_1(\Sigma_{g}^1;\Z/\ell)$
for $p+1 \leq i \leq g$.
\end{itemize}
The first of these properties implies that $f(\tau_1) = \tau_2$.  Since $f$ fixes the
symplectic basis $\{a_1,b_1,\ldots,a_g,b_g\}$ for $\HH_1(\Sigma_{g}^1;\Z/\ell)$,
it lies in $\Mod_{g}^1(\ell)$.  The claim follows.
\end{proof}

\subsection{High connectivity of quotient}
Building on work of Charney \cite{CharneyVogtmann}, Mirzaii--van der Kallen \cite{MirzaiiVanderKallen}
proved the following:

\begin{theorem}[{Mirzaii--van der Kallen \cite[Lemma 7.4]{MirzaiiVanderKallen}}]
\label{theorem:highconnectivitysb}
Fix some $g \geq 0$ and $\ell \geq 2$.  Let $V = \HH_1(\Sigma_{g}^1;\Z/\ell)$.  Then
$\bbSB(V)$ is $\frac{g-5}{2}$-connected.
\end{theorem}

This has the following corollary.

\begin{corollary}
\label{corollary:highconnectivityquotient}
Fix some $g \geq 0$ and $\ell \geq 2$.  Let $I$ be an open interval in $\partial \Sigma_{g}^1$ and let
$H$ be a genus-$h$ symplectic subgroup of $\HH_1(\Sigma_{g}^1;\Z/\ell)$.  Then
$\bbTT_g^1(I,H) / \Mod_{g}^1(\ell)$ is $\frac{g-h-5}{2}$-connected.
\end{corollary}
\begin{proof}
Lemma \ref{lemma:identifyquotient} says that $\bbTT_g^1(I,H) / \Mod_{g}^1(\ell) \cong \bbSB(H^{\perp})$,
and $H^{\perp} \cong \HH_1(\Sigma_{g-h}^1;\Z/\ell)$.  Theorem \ref{theorem:highconnectivitysb} thus
implies that $\bbTT_g^1(I,H) / \Mod_{g}^1(\ell)$ is $\frac{(g-h)-5}{2}$-connected.
\end{proof}

\section{Prym representations}
\label{section:prymrepresentations}

We now discuss the definition and some basic properties of the Prym representations and show
how to encode them by equivariant augmented coefficient systems on the tethered torus complexes $\bbTT_g^1(I,H)$.  
Throughout this section, $\bbk$ is a commutative ring.  Fix some $g \geq 1$ and $\ell \geq 2$, and let\footnote{Here $\cD$ stands for ``deck group''.}
\[\cD = \HH_1(\Sigma_g^1;\Z/\ell) = \HH_1(\Sigma_g;\Z/\ell) \cong (\Z/\ell)^{2g}.\]

\subsection{Surfaces with one boundary component, definition}
We start with surfaces $\Sigma_g^1$ with one boundary component.  In this case, the Prym
representation is defined as follows.  
Let $S_{\cD} \rightarrow \Sigma_g^1$ be the finite regular cover corresponding to
the homomorphism\footnote{Since the target of this homomorphism
is abelian, there is no need to specify a basepoint of $\pi_1(\Sigma_g^1)$; however, if the reader
prefers to be careful about basepoints then they should fix one on $\partial \Sigma_g^1$.}
$\pi_1(\Sigma_g^1) \rightarrow \cD$.  The deck group of this cover
is $\cD$.  By definition, the Prym representation with coefficients in $\bbk$ is 
\[\fH_g^1(\ell;\bbk) = \HH_1(S_{\cD};\bbk).\]
The level-$\ell$ mapping class group $\Mod_g^1(\ell)$ acts on $\fH_g^1(\ell;\bbk)$ via the action
on homology of lifts of mapping classes on $\Sigma_g^1$ to $S_{\cD}$ that fix $\partial S_{\cD}$
pointwise.  

\begin{remark}
It is important that $\Sigma_g^1$ has nonempty boundary.  Otherwise, due to basepoint
issues there would not be a canonical way to lift elements of $\Mod_g^1(\ell)$ to the
cover $S_{\cD}$.
\end{remark}

\begin{remark}
We could extend the action of $\Mod_g^1(\ell)$ on $\fH_g^1(\ell;\bbk)$ to $\Mod_g^1$ since the cover
$S_{\cD} \rightarrow \Sigma_g^1$ is a characteristic cover.\footnote{That is, it corresponds
to a subgroup of $\pi_1(\Sigma_g^1)$ that is preserved by all automorphisms.}  However, the
lifts in that case would only fix a single component of $\partial S_{\cD}$.
\end{remark}

\subsection{Partial Prym representation}
It is unclear how to incorporate the $\fH_g^1(\ell;\bbk)$ into augmented coefficient systems
on $\bbTT_g^1(I)$, and it seems unlikely that any such coefficient system
would be polynomial.  To fix this, we restrict ourselves to the partial Prym
representations,\footnote{In \S \ref{section:intermediate}, we will explain how to relate
the partial Prym representations to the Prym representation.} which are defined as follows.  

Let $H$ be a symplectic
subgroup of $\HH_1(\Sigma_g^1;\Z/\ell)$.  Recall from \S \ref{section:partial} that 
the associated partial level-$\ell$ subgroup, denoted $\Mod_{g}^1(H)$, is the group of all 
$f \in \Mod_{g}^1$ such that $f(x) = x$ for all $x \in H$.  Let $S_H \rightarrow \Sigma_g^1$
be the finite regular cover corresponding to the homomorphism
\[\pi_1(\Sigma_g^1) \rightarrow \HH_1(\Sigma_g^1) = H \oplus H^{\perp} \stackrel{\text{proj}}{\longrightarrow} H.\]
The deck group of this cover is $H$.  Setting $\fH_g^1(H;\bbk) = \HH_1(S_H;\bbk)$, just like for
$\fH_g^1(\ell;\bbk)$ we can define an action of $\Mod_g^1(H)$ on $\fH_g^1(H;\bbk)$ by lifting\footnote{Unlike
for $\fH_g^1(\ell;\bbk)$, this action cannot be extended to $\Mod_g^1$ since this is not a characteristic cover.}
mapping classes to $S_H$.  We will call $\fH_g^1(H;\bbk)$ a partial Prym representation.

\subsection{Coefficient system}
\label{section:coefficientsystem}

Continue to let $H$ be a symplectic subgroup of $\HH_1(\Sigma_g^1;\Z/\ell)$, and let
$\pi\colon S_H \rightarrow \Sigma_g^1$ be the regular cover discussed above.  Fix an
open interval $I$ in $\partial \Sigma_g^1$, and consider a simplex $\sigma = [\iota_0,\ldots,\iota_k]$
of $\bbTT_g^1(I,H)$.  Set
\[X_{\sigma} = \Sigma_g^1 \setminus \text{Nbhd}\left(\partial \Sigma_g^1 \cup \Image\left(\iota_0\right) \cup \cdots \cup \Image\left(\iota_k\right)\right),\]
where $\text{Nbhd}(-)$ denotes an open regular neighborhood of the indicated subset of $\Sigma_g^1$.  See here:\\
\centerline{\psfig{file=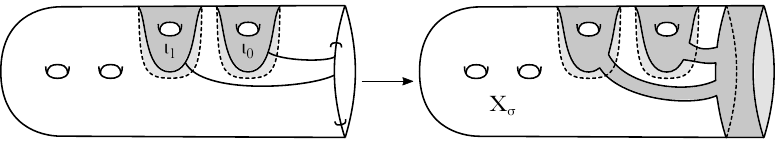,scale=1}}
We thus have $X_{\sigma} \cong \Sigma_{g-k-1}^1$.
Since $\sigma$ is a simplex of $\bbTT_g^1(I,H)$, the map $\pi_1(\Sigma_g^1) \rightarrow H$ used
to define $\pi\colon S_H \rightarrow \Sigma_g^1$ restricts to a surjective map $\pi_1(X_{\sigma}) \rightarrow H$.
It follows that $\tX_{\sigma} = \pi^{-1}(X_{\sigma})$ is a connected submanifold of $S_H$ and
$\tX_{\sigma} \rightarrow X_{\sigma}$ is a finite regular $H$-cover.  Define
an augmented coefficient system $\cH_g^1(H;\bbk)$ on $\bbTT_g^1(I,H)$ via the formula
\[\cH_g^1(H;\bbk)(\sigma) = \HH_1(\tX_{\sigma};\bbk).\]
Our convention is that if $\sigma = \emptysimp$ is the $(-1)$-simplex, then $X_{\sigma} = \Sigma_g^1 \setminus \text{Nbhd}\left(\partial \Sigma_g^1\right)$,
so
\[\cH_g^1(H;\bbk)\emptysimp  = \HH_1(X_{\sigma};\bbk) \cong \HH_1(S_H;\bbk) = \fH_g^1(H;\bbk).\]
Our main result about this coefficient system is that it is strongly polynomial of degree
$1$ (see \S \ref{section:strongpolynomial} for the definition of a strongly polynomial coefficient system):

\begin{lemma}
\label{lemma:hstrongpoly}
Let $g \geq 0$ and $\ell \geq 2$.  Fix a symplectic subgroup $H$ of $\HH_1(\Sigma_g^1;\Z/\ell)$ and an open interval
$I$ in $\partial \Sigma_g^1$.  Then for all commutative rings $\bbk$ the augmented coefficient system $\cH_g^1(H;\bbk)$ on $\bbTT_g^1(I,H)$ is strongly
polynomial of degree $1$.
\end{lemma}
\begin{proof}
It is immediate from the definitions that $\cH_g^1(H;\bbk)$ is injective.  The other condition in
the definition of being strongly polynomial of degree $1$ is as follows.
Let $\tau = [\iota_0,\ldots,\iota_{\ell}]$ be a simplex of $\bbTT_g^1(I,H)$.  Set $\tau' = [\iota_0,\ldots,\iota_{\ell-1}]$, and let
$D_{\tau} \cH_g^1(H;\bbk)$ be the coefficient system on the forward link $\bbL = \FLink_{\bbTT_g^1(I,H)}(\tau)$
defined by the formula
\[D_{\tau} \cH_g^1(H;\bbk)(\sigma) = \frac{\cH_g^1(H;\bbk)(\tau' \cdot \sigma)}{\Image\left(\cH_g^1\left(H;\bbk\right)\left(\tau \cdot \sigma\right) \rightarrow \cH_g^1\left(H;\bbk\right)\left(\tau' \cdot \sigma\right)\right)} \quad \text{for a simplex $\sigma$ of $\bbL$}.\]
We must prove that $D_{\tau} \cH_g^1(H;\bbk)$ is strongly polynomial of degree $0$, i.e., constant.  Expanding out
the above formula for $D_{\tau} \cH_g^1(H;\bbk)(\sigma)$, we see that
\[D_{\tau} \cH_g^1(H;\bbk)(\sigma) = \frac{\HH_1(\tX_{\tau' \cdot \sigma};\bbk)}{\Image\left(\HH_1\left(\tX_{\tau \cdot \sigma};\bbk\right) \rightarrow \HH_1\left(\tX_{\tau' \cdot \sigma};\bbk\right)\right)}.\]
Letting $\pi\colon S_H \rightarrow \Sigma_g^1$ be the regular cover used to define $\cH_g^1(H;\bbk)$, it is immediate that
this is isomorphic to
\[\HH_1\left(\pi^{-1}\left(\Image\left(\iota_{\ell}\right)\right);\bbk\right).\]
The subspace $\pi^{-1}(\Image(\iota_{\ell}))$ of $S_H$ is the disjoint union of $|H|$ copies of a tethered torus $\tau(\Sigma_1^1)$, one
tethered to each component of $\partial S_H$.  Its first homology group injects into $\HH_1(S_H;\bbk)$.  That $D_{\tau} \cH_g^1(H;\bbk)$ is constant follows.
\end{proof}

\subsection{General surfaces, definition}
Our next goal is to relate $\fH_g^1(\ell;\bbk)$ and $\fH_g^1(H;\bbk)$.  For later use, we put these
results in a broader context.
Throughout the rest of this section, fix some $g,b,p \geq 0$ and $\ell \geq 2$ with $b+p \geq 1$.  
Let $S_{\cD} \rightarrow \Sigma_{g,p}^b$ be
the regular cover with deck group $\cD = \HH_1(\Sigma_g;\Z/\ell)$
corresponding to the group homomorphism
\[\pi_1(\Sigma_{g,p}^b) \rightarrow \HH_1(\Sigma_{g,p}^b;\Z/\ell) \rightarrow \HH_1(\Sigma_g;\Z/\ell) = \cD,\]
where the second map fills in the punctures and glues discs to the boundary components.
Define $\fH_{g,p}^b(\ell;\bbk) = \HH_1(S_{\cD};\bbk)$.  The group $\Mod_{g,p}^b(\ell)$ acts
on $\fH_{g,p}^b(\ell;\bbk)$ as before.

\begin{remark}
\label{remark:boundarypunctures}
At the level of homology, there is no difference between boundary components and punctures, so $\fH_{g,p}^b(\ell;\bbk) \cong \fH_{g,p+b}(\ell;\bbk)$.
\end{remark}

\subsection{Decomposition}
\label{section:decompositionprym}

We now specialize $\bbk$ to the field $\C$ of complex numbers.
Our goal is to decompose $\fH_{g,p}^b(\ell;\C)$ into subrepresentations and show that each of
these subrepresentations appears in $\fH_{g,p}^b(H;\C)$ for an appropriate symplectic subgroup $H$
of $\HH_1(\Sigma_{g,p}^b;\Z/\ell)$.

The vector space $\fH_{g,p}^b(\ell;\C)$ has actions of the following groups:
\begin{itemize}
\item The group $\cD \cong (\Z/\ell)^{2g}$, which acts on $S_{\cD}$ as the group of deck transformations.
\item The group $\Mod_{g,p}^b(\ell)$, which acts via the action obtained by lifting diffeomorphisms of $\Sigma_{g,p}^b$ to
diffeomorphisms of $S_{\cD}$ fixing all boundary components and punctures.\footnote{This is where we use the fact
that $p+b \geq 1$, so there is a fixed basepoint.  Otherwise, our lifts would only be defined up to the action of the deck group.}
\end{itemize}
These two actions commute, so the action of $\Mod_{g,p}^b$ on $\fH_{g,p}^b(\ell;\C)$ preserves
the $\cD$-isotypic components of $\fH_{g,p}^b(\ell;\C)$.

Since $\cD \cong (\Z/\ell)^{2g}$ is a finite abelian group, its irreducible $\C$-representations
are all $1$-dimensional and in bijection with characters $\chi\colon \cD \rightarrow \C^{\times}$.
Letting $\hcD$ be the abelian group of characters of $\cD$, the
irreducible representation corresponding to $\chi \in \hcD$ is a $1$-dimensional $\C$-vector space $\C_{\chi}$
with the action
\[d \cdot \vec{v} = \chi(d) \vec{v} \quad \text{for all $\vec{v} \in \C_{\chi}$ and $d \in \cD$}.\]
Let $\fH_{g,p}^b(\chi)$ be the $\C_{\chi}$-isotypic component of $\fH_{g,p}^b(\ell;\C)$.  By definition, this is
the set of all $\vec{w} \in \fH_{g,p}^b(\ell;\C)$ such that $d \cdot \vec{w} = \chi(d) \vec{w}$ for all
$d \in \cD$.  The action of $\Mod_{g,p}^b(\ell)$ on $\fH_{g,p}^b(\ell;\C)$ preserves
$\fH_{g,p}^b(\chi)$, and we have a direct sum decomposition
\[\fH_{g,p}^b(\ell;\C) = \bigoplus_{\chi \in \hcD} \fH_{g,p}^b(\chi)\]
of representations of $\Mod_{g,p}^b(\ell)$.  

\subsection{Intermediate covers}
\label{section:intermediate}

Let $H$ be a symplectic subgroup of $\HH_1(\Sigma_{g,p}^b;\Z/\ell)$, and let $S_H \rightarrow \Sigma_{g,p}^b$ 
be the cover corresponding to the surjective homomorphism
\[\pi_1(\Sigma_{g,p}^b) \rightarrow \HH_1(\Sigma_{g,p}^b;\Z/\ell) = H \oplus H^{\perp} \stackrel{\text{proj}}{\longrightarrow} H.\]
Since the homology classes of loops surrounding boundary components and punctures lie in $H^{\perp}$, this map factors through $\cD$, so this cover lies between $S_{\cD}$ and $\Sigma_{g,p}^b$ in the sense 
that we have a factorization
\[S_{\cD} \longrightarrow S_H \longrightarrow \Sigma_{g,p}^b.\]
Define $\fH_{g,p}^b(H;\C) = \HH_1(S_H;\C)$.  The partial mod-$\ell$ subgroup $\Mod_{g,p}^b(H)$ acts
on $\fH_{g,p}^b(H;\C)$ as before.  The deck group of $S_H \rightarrow \Sigma_{g,p}^b$ is $H$,
so again $\fH_{g,p}^b(H;\C)$ decomposes into a direct sum of $H$-isotypic components, indexed
by elements of the dual group $\hH$ of characters.

As we said above, the map $\pi_1(\Sigma_{g,p}^b) \rightarrow H$ factors through $\cD$, giving a surjection
$\cD \rightarrow H$.  This induces an inclusion $\hH \hookrightarrow \hcD$, and
we will identify $\hH$ with its image in $\hcD$.  An element of $\hcD$ lying
in $\hH$ is said to be compatible with $H$.  We then have the following:

\begin{lemma}
\label{lemma:intermediatecover}
Fix $g,p,b \geq 0$ and $\ell \geq 2$ with $p+b \geq 1$.  Let $H$ be a symplectic subgroup of $\HH_1(\Sigma_{g,p}^b;\Z/\ell)$.  Then
for all $\chi \in \hcD$ that are compatible with $H$, the $\C_{\chi}$-isotypic component of $\fH_{g,p}^b(H;\C)$ is naturally isomorphic\footnote{The meaning of ``natural'' here is
that the covering map $S_{\cD} \rightarrow S_H$ takes $\fH_{g,p}^b(\chi)$ isomorphically to the $\C_{\chi}$-isotypic
component of $\fH_{g,p}^b(H;\C)$.  In particular, the isomorphism is $\Mod_{g,p}^b(\ell)$-equivariant.} to $\fH_{g,p}^b(\chi)$, so in particular
\[\fH_{g,p}^b(H;\C) = \bigoplus_{\chi \in \hH} \fH_{g,p}^b(\chi).\]
\end{lemma}

Before proving Lemma \ref{lemma:intermediatecover}, we highlight one special case of it:

\begin{example}
\label{example:trivialcover}
Fix $g,p,b \geq 0$ and $\ell \geq 2$ with $p+b \geq 1$.  Let $H = 0$, so
$\fH_{g,p}^b(H;\C) = \HH_1(\Sigma_{g,p}^b;\C)$.  Letting $1$ be the trivial character
of $\cD$, we have $\hH = \{1\}$.  Lemma \ref{lemma:intermediatecover} therefore
implies that
\[\HH_1(\Sigma_{g,p}^b;\C) = \fH_{g,p}^b(H;\C) = \fH_{g,p}^b(1;\C). \qedhere\]
\end{example}

\begin{proof}[Proof of Lemma \ref{lemma:intermediatecover}]
Let $K$ be the kernel of the quotient map $\cD \twoheadrightarrow H$, so 
$S_H = S_{\cD} / K$.  A standard property of group actions (see, e.g., \cite[Theorem III.2.4]{BredonTransformation} 
or \cite[Proposition 1.1]{PutmanHomologyActionNote}) says that if a finite group $G$ acts smoothly on a smooth compact 
manifold with boundary\footnote{Or, more generally, a compact simplicial complex.}
$X$, then the $G$-coinvariants of the action of $G$ on $\HH_k(X;\C)$ are $\HH_k(X/G;\C)$.  Applying this
to the action of $K$ on $S_{\cD}$, we deduce\footnote{Strictly speaking, this does not apply if $p \geq 1$ since
then $\Sigma_{g,p}^b$ is not compact.  However, replacing each puncture with a boundary component does not
change the homology groups of the surface, so we can assume without loss of generality that $p=0$.}
that
\[\fH_{g,p}^b(H;\C) = \HH_1(S_H;\C) = \HH_1(S_{\cD};\C)_K = \fH_{g,p}^b(\ell;\C)_K,\]
where the subscripts indicate that we are taking the $K$-coinvariants.

Applying this to the decomposition
\[\fH_{g,p}^b(\ell;\C) = \bigoplus_{\chi \in \hcD} \fH_{g,p}^b(\chi),\]
we deduce that 
\[\fH_{g,p}^b(H;\C) = \bigoplus_{\chi \in \hcD} \fH_{g,p}^b(\chi)_K.\]
We claim that for $\chi \in \hcD$ we have
\[\fH_{g,p}^b(\chi)_K = \begin{cases}
\fH_{g,p}^b(\chi) & \text{if $\chi \in \hH$},\\
0                    & \text{if $\chi \notin \hH$}.
\end{cases}\]
For $k \in K$, the element $k$ acts on $\fH_{g,p}^b(\chi)$ as multiplication by $\chi(k)$.  If this is ever not $1$,
then taking the $K$-coinvariants of $\fH_{g,p}^b(\chi)$ reduces it to $0$.  Otherwise, if it is always $1$ then
taking the $K$-coinvariants of $\fH_{g,p}^b(\chi)$ does not change it.  Since $\hH$ is precisely the subgroup
of $\hcD$ consisting of characters that are identically $1$ on $K$, the claim follows.

We conclude that
\[\fH_{g,p}^b(H;\C) = \bigoplus_{\chi \in \hH} \fH_{g,p}^b(\chi).\]
It is immediate from the above that for $\chi \in \hH$, the action of $\cD$ on $\fH_{g,p}^b(\chi)$ factors through
$H = \cD/K$ and that for $h \in H$ and $\vec{v} \in \fH_{g,p}^b(\chi)$ we have
$h \cdot \vec{v} = \chi(h) \vec{v}$.  We conclude that this is exactly the decomposition into $H$-isotypic components, as desired.
\end{proof}

\begin{corollary}
\label{corollary:extendchi}
Fix $g,p,b \geq 0$ and $\ell \geq 2$ with $p+b \geq 1$.  Let $H$ be a symplectic subgroup of $\HH_1(\Sigma_{g,p}^b;\Z/\ell)$
and let $\chi \in \hcD$.  If $\chi$ is compatible with $H$, then 
the action of $\Mod_{g,p}^b(\ell)$ on $\fH_{g,p}^b(\chi)$ extends to an action of $\Mod_{g,p}^b(H)$.
\end{corollary}
\begin{proof}
Immediate.
\end{proof}

\subsection{Deleting punctures}
The following relates $\fH_{g,p+1}^b(\chi)$ and $\fH_{g,p}^b(\chi)$.  It uses
the convention from \S \ref{section:symplecticconventions}.

\begin{lemma}
\label{lemma:capboundary}
Fix $g,b,p \geq 0$ and $\ell \geq 2$ with $p+b \geq 1$.  Let $H$ be a symplectic subgroup of $\HH_1(\Sigma_{g,p+1}^b;\Z/\ell)$ and let
$\chi \in \hH$.  Let $x_0$ be a puncture of $\Sigma_{g,p+1}^b$.
We then have a short exact sequence
\[0 \longrightarrow \C \longrightarrow \fH_{g,p+1}^b(\chi) \longrightarrow \fH_{g,p}^b(\chi) \longrightarrow 0\]
of $\Mod_{g,p+1}^b(H)$-representations.  Here $\C$ is the trivial representation and
$\Mod_{g,p+1}^b(H)$ acts on $\fH_{g,p}^b(\chi)$ via the homomorphism $\Mod_{g,p+1}^b(H) \rightarrow \Mod_{g,p}^b(H)$
that deletes $x_0$.
\end{lemma}
\begin{proof}
Let $S'_H$ and $S_{H}$ be the covers used to define $\fH_{g,p+1}^b(H;\C)$ and $\fH_{g,p}^b(H;\C)$, respectively.  Let
$P$ be the set of punctures of $S'_H$ that project to $x_0$, so $S_{H}$ is obtained from $S'_H$ by deleting
all the punctures in $P$.  Since $b+p \geq 1$, deleting all the punctures in $P$ does not yield a closed surface.  Letting
$\C[P]$ be the set of formal $\C$-linear combinations of elements of $P$, we therefore get an injection
$\C[P] \hookrightarrow \HH_1(S'_H;\C)$ taking $p \in P$ to the homology class of a loop surrounding $p$, oriented
such that $p$ is to its right.  This fits into a short exact sequence
\begin{equation}
\label{eqn:fillinpunctures}
0 \longrightarrow \C[P] \longrightarrow \HH_1(S'_H;\C) \longrightarrow \HH_1(S_{H};\C) \longrightarrow 0.
\end{equation}
The deck group $H$ acts simply transitively on $P$, so as a representation of $H$ we have $\C[P] \cong \C[H]$.  It
follows that the $\C_{\chi}$-isotypic component of $\C[P]$ is $1$-dimensional.  Taking $\C_{\chi}$-isotypic
components in \eqref{eqn:fillinpunctures}, we therefore get a short exact sequence
\[0 \longrightarrow \C \longrightarrow \fH_{g,p+1}^b(\chi) \longrightarrow \fH_{g,p}^b(\chi) \longrightarrow 0\]
of $\Mod_{g,p+1}^b(H)$-representations.  That the actions of $\Mod_{g,p+1}^b(H)$ on the kernel and cokernel are
as described in the lemma is immediate.
\end{proof}

\subsection{Homological representations}

Let $\uchi = (\chi_1,\ldots,\chi_r)$ be an $r$-tuple of elements of $\hcD$.  We define the associated {\em homological representation}
of $\Mod_{g,p}^b(\ell)$ to be
\begin{equation}
\label{eqn:homologicalrep}
\fH_{g,p}^b(\uchi) = \fH_{g,p}^b(\chi_1) \otimes \cdots \otimes \fH_{g,p}^b(\chi_r).
\end{equation}
The number $r$ is the {\em size} of $\fH_{g,p}^b(\uchi)$.  If $H$ is a symplectic subgroup of $\HH_1(\Sigma_{g,p}^b;\Z/\ell)$ and each
$\chi_i$ is compatible with $H$, then we say that $\fH_{g,p}^b(\uchi)$ is {\em compatible} with $H$.  By
Corollary \ref{corollary:extendchi}, this implies that the action of $\Mod_{g,p}^b(\ell)$ on $\fH_{g,p}^b(\uchi)$
extends to an action of $\Mod_{g,p}^b(H)$.  This is a stronger statement if $H$ is smaller, and the following lemma
will allow us to bound how large of an $H$ we must take:

\begin{lemma}
\label{lemma:extendhomological}
Fix $g,p,b \geq 0$ and $\ell \geq 2$ with $p+b \geq 1$, and let $\fH_{g,p}^b(\uchi)$ be a homological representation of $\Mod_{g,p}^b(\ell)$
of size $r$ that is compatible with a symplectic subgroup $H$ of $\HH_1(\Sigma_{g,p}^b;\Z/\ell)$.
Then there exists a symplectic subgroup $H'$ of $\HH_1(\Sigma_{g,p}^b;\Z/\ell)$ of genus at most $r$ with $H' \subset H$ such
that $\fH_{g,p}^b(\uchi)$ is compatible with $H'$.
\end{lemma}
\begin{proof}
Let $h$ be the genus of $H$.  If $h \leq r$ then there is nothing to prove, so assume that $h > r$.  Write
\[\fH_{g,p}^b(\uchi) = \fH_{g,p}^b(\chi_1) \otimes \cdots \otimes \fH_{g,p}^b(\chi_r).\]
Let $C_i \subset \C^{\times}$ 
be the image of $\chi_i$, so $C_i$ is a possibly trivial finite cyclic group.  Regard each $\chi_i$ as a map
$H \rightarrow C_i$.  Set $A = C_1 \times \cdots \times C_r$, so $A$ is an abelian group
of rank\footnote{By definition, the rank of an abelian group is the minimal cardinality of a generating set for it.} 
at most $r$.  Let $\mu\colon H \rightarrow A$ be $\mu = \chi_1 \times \cdots \times \chi_r$.  
By\footnote{This reference is about maps $\Z^{2h} \rightarrow A$ rather than $(\Z/\ell)^{2h} \rightarrow A$,
but the same proof works in our situation.  Alternatively, apply it to the composition $\Z^{2h} \rightarrow (\Z/\ell)^{2h} \stackrel{\mu}{\rightarrow} A$ and then map the resulting symplectic subspace of $\Z^{2h}$ to $(\Z/\ell)^{2h}$.} \cite[Lemma 3.5]{PutmanPartialTorelli}, we can find a genus 
$h-r$ symplectic subgroup $U$ of $H$ such that $\mu$ vanishes on $U$.  Let $H' \subset H$ be the orthogonal
complement of $U$ in $H$, so $H'$ is a genus $r$ symplectic subspace of $H$ such that
each $\chi_i$ factors through the projection of $H$ to $H'$.  This implies that $\fH_{g,p}^b(\uchi)$ is compatible with $H'$.
\end{proof}

\section{The Reidemeister pairing and the point-pushing subgroup}
\label{section:reidemeister}

This section describes an important bilinear pairing on the Prym representation.  It goes back
to work of Reidemeister \cite{Reidemeister1, Reidemeister2}, and has since appeared in
many places.  

\subsection{Reidemeister pairing}
\label{section:reidemeisterdefinition}

Fix some $g,b,p \geq 0$ with $b+p \geq 1$.  Let $\ell \geq 2$ and let $H$ be a symplectic subgroup of $\HH_1(\Sigma_{g,p+1}^b;\Z/\ell)$.
Let $\bbk$ be a commutative ring and let $\omega_H(-,-)$ be the algebraic intersection pairing on $\fH_{g,p}^b(H;\bbk) = \HH_1(S_H;\bbk)$.  The group
$H$ acts on $\fH_{g,p}^b(H;\bbk)$ via its action on $S_H$ by deck transformations.
The Reidemeister pairing on $\fH_{g,p}^b(H;\bbk)$ is the map
\[\Romega{H}\colon \fH_{g,p}^b(H;\bbk) \times \fH_{g,p}^b(H;\bbk) \longrightarrow \bbk[H]\]
defined by the formula
\[\Romega{H}(x,y) = \sum_{d \in H} \omega_{H}(x,d y) d \quad \quad \text{for all $x,y \in \fH_{g,p}^b(H;\bbk)$}.\]

\subsection{Point-pushing subgroup}
We now connect the Reidemeister pairing to the point-pushing subgroup.
Fix some $g,p,b \geq 0$ with $\pi_1(\Sigma_{g,p}^b)$ nonabelian and $p+b \geq 1$.  Let $x_0$ be a puncture of $\Sigma_{g,p+1}^b$.
Let $\ell \geq 2$ and let $H$ be a symplectic subgroup of $\HH_1(\Sigma_{g,p+1}^b;\Z/\ell)$.
Using the conventions from \S \ref{section:symplecticconventions},
Theorem \ref{theorem:birmanpartial} gives a Birman exact sequence
\[1 \longrightarrow \PP_{x_0}(\Sigma_{g,p}^b,H) \longrightarrow \Mod_{g,p+1}^b(H) \longrightarrow \Mod_{g,p}^b(H) \longrightarrow 1.\]
Here the point-pushing subgroup $\PP_{x_0}(\Sigma_{g,p}^b,H)$ is the kernel of the map
\[\PP_{x_0}(\Sigma_{g,p}^b) \cong \pi_1(\Sigma_{g,p}^b) \longrightarrow \HH_1(\Sigma_{g,p}^b) = H \oplus H^{\perp} \stackrel{\text{proj}}{\longrightarrow} H.\]
Let $S'_H \rightarrow \Sigma_{g,p+1}^b$ and $S_H \rightarrow \Sigma_{g,p}^b$ be the covers used to define $\fH_{g,p+1}^b(H;\bbk)$ and $\fH_{g,p}^b(H;\bbk)$,
respectively.
Fixing a commutative ring $\bbk$, we want to understand the action of $\PP_{x_0}(\Sigma_{g,p}^b,H)$ on $\fH_{g,p+1}^b(H;\bbk) = \HH_1(S'_H;\bbk)$.
By the above, $S_H$ is the cover of $\Sigma_{g,p}^b$ corresponding to the subgroup $\PP_{x_0}(\Sigma_{g,p}^b,H)$ of
$\PP_{x_0}(\Sigma_{g,p}^b) \cong \pi_1(\Sigma_{g,p}^b)$.  We can thus identify $\PP_{x_0}(\Sigma_{g,p}^b,H)$ with $\pi_1(S_H)$.  The following lemma
shows how the action we are trying to understand is encoded by the Reidemeister pairing on
$\fH_{g,p}^b(H;\bbk) = \HH_1(S_H;\bbk)$.

\begin{lemma}
\label{lemma:pointpushreidemeister}
Let the notation be as above.  Let $\rho_1\colon \fH_{g,p+1}^b(H;\bbk) \rightarrow \fH_{g,p}^b(H;\bbk)$ be the map induced by
filling in $x_0$ and $\rho_2\colon \PP_{x_0}(\Sigma_{g,p}^{b},H) \rightarrow \fH_{g,p}^b(H;\bbk)$ be the composition
\[\PP_{x_0}(\Sigma_{g,p}^b,H) \cong \pi_1(S_{H}) \longrightarrow \fH_{g,p}^b(H;\bbk).\]
Let $\zeta \in \fH_{g,p+1}^b(H;\bbk)$ be the homology class of
a loop around the puncture of $S'_H$ that is used as the basepoint
in the identification of $\PP_{x_0}(\Sigma_{g,p}^b,H)$ with $\pi_1(S_{H})$, oriented such that the puncture is to its right.
Then for $\gamma \in \PP_{x_0}(\Sigma_{g,p}^b,H)$ and $z \in \fH_{g,p+1}^b(H;\bbk)$, we have
\[\gamma(z) = z + \Romega{H}(\rho_1(z),\rho_2(\gamma)) \cdot \zeta.\]
\end{lemma}
\begin{proof}
The action of $\gamma \in \PP_{x_0}(\Sigma_{g,p}^b,H)$ on $\fH_{g,p+1}^b(H;\bbk) = \HH_1(S'_H;\bbk)$
comes from simultaneously pushing all the punctures projecting to $x_0$ around paths
in $S_{H}$.  These punctures and paths are all $H$-orbits of the basepoint puncture
and the lift of $\gamma$ to that basepoint puncture.  The lemma is thus immediate
from Lemma \ref{lemma:pointpushh1}.
\end{proof}

\begin{remark}
For $\chi \in \hH$, the action of $\Mod_{g,p+1}^b(H)$ on $\fH_{g,p+1}^b(H;\C)$ preserves the subspace
$\fH_{g,p+1}^b(\chi)$.  It thus follows from Lemma \ref{lemma:pointpushreidemeister} that
for all $x,y \in \fH_{g,p}^b(H;\C)$ with $x \in \fH_{g,p}^b(\chi)$, the element
$\Romega{H}(x,y) \in \C[H]$ lies in the $\C_{\chi}$-isotypic subspace of $\C[H]$.  It is
enlightening to prove this directly.
\end{remark}

\subsection{Point-pushing coinvariants}
We next study the action of the point-pushing subgroup 
$\PP_{x_0}(\Sigma_{g,p}^b,H)$ from Theorem \ref{theorem:birmanpartial} on tensor
powers of $\fH_{g,p+1}^b(H;\bbk)$.  In the following lemma, the subscript indicates that we are taking
coinvariants.  The statement uses the conventions from \S \ref{section:symplecticconventions}

\begin{lemma}
\label{lemma:pushtensorpowers}
Fix some $g,p,b \geq 0$ such that $\pi_1(\Sigma_{g,p}^b)$ is nonabelian and $p+b \geq 1$, and let
$x_0$ be a puncture of $\Sigma_{g,p+1}^b$.  Let $\ell \geq 2$ and let
$H$ be a genus-$h$ symplectic subgroup of $\HH_1(\Sigma_{g,p+1}^b;\Z/\ell)$.  
Let $r \geq 0$ be such that $g \geq h+r$.  Then for all finite-index subgroups
$G$ of $\PP_{x_0}(\Sigma_{g,p}^b,H)$ and all fields $\bbk$ of characteristic $0$, we have
\[\left(\fH_{g,p+1}^b\left(H;\bbk\right)^{\otimes r}\right)_G \cong \fH_{g,p}^b\left(H;\bbk\right)^{\otimes r}.\]
\end{lemma}
\begin{proof}
The representations (and hence the lemma) are trivial if $r=0$, so we can assume that $r \geq 1$.  Let $\rho_1\colon \fH_{g,p+1}^b(H;\bbk) \rightarrow \fH_{g,p}^b(H;\bbk)$
be the map induced by filling in $x_0$.  The map 
$\rho_1^{\otimes r}\colon \fH_{g,p+1}^b(H;\bbk)^{\otimes r} \rightarrow \fH_{g,p}^b(H;\bbk)^{\otimes r}$
is surjective and factors through the $G$-coinvariants.  What we must show is that
all elements of the kernel of $\rho_1^{\otimes r}$ die in the $G$-coinvariants.  We divide
the proof of this into three steps.

\begin{step}{1}
We find generators for the kernel of 
$\rho_1^{\otimes r}\colon \fH_{g,p+1}^b(H;\bbk)^{\otimes r} \rightarrow \fH_{g,p}^b(H;\bbk)^{\otimes r}$.
\end{step}

This requires carefully constructing the relevant covers.
Let $T$ be a subsurface of $\Sigma_{g,p+1}^b$ with the following two properties:
\begin{itemize}
\item $T \cong \Sigma_{h,p}^{b+1}$ and does not contain the puncture $x_0$, and
\item $H$ is contained in the image of the map $\HH_1(T;\Z/\ell) \rightarrow \HH_1(\Sigma_{g,p+1}^b;\Z/\ell)$.
\end{itemize}
It follows that $T' = \Sigma_{g,p+1}^b \setminus \Interior(T)$ satisfies
$T' \cong \Sigma_{g-h,1}^1$.  Let $T''$ be a subsurface of $T'$ with $T'' \cong \Sigma_{g-h}^1$.  See the following figure, which depicts the surface $\Sigma_{g,p+1}^b = \Sigma_{8,4}^2$ with $h=5$:\\
\centerline{\psfig{file=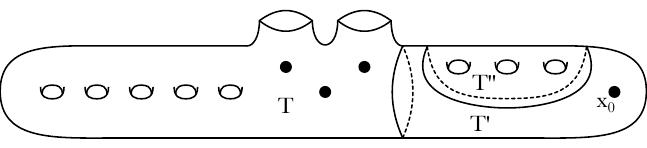,scale=1}}

Let $\pi\colon S'_H \rightarrow \Sigma_{g,p+1}^b$ be the cover used to define $\fH_{g,p+1}^b(H)$, and
let $\tT = \pi^{-1}(T)$ and $\tT' = \pi^{-1}(T')$.  Both $\tT \rightarrow T$ and $\tT' \rightarrow T'$
are finite regular covers with deck group $H$.  The second condition above implies that
$\tT$ is connected and that $\tT'$ is the disjoint union of $|H|$ components
each of which projects homeomorphically
to $T'$.  Letting $\tT'_0$ be one of these components, we have\footnote{Here the reader should think that the $d$
used to denote elements of $H$ stands for ``deck group''.}
\[\tT' = \bigsqcup_{d \in H} d \tT'_0.\]
Let $\tT''_0$ be the component of $\pi^{-1}(T'')$ lying in $\tT'_0$.  Both $\HH_1(\tT;\bbk)$ and
$\HH_1(\tT_0'';\bbk)$ inject into $\HH_1(S'_H;\bbk)$, and
\[\fH_{g,p+1}^b(H;\bbk) = \HH_1(S'_H;\bbk) = \HH_1(\tT;\bbk) \oplus \bigoplus_{d \in H} d \HH_1(\tT_0'';\bbk).\]
It follows that the $\bbk$-module $\fH_{g,p+1}^b(H;\bbk)^{\otimes r}$ is spanned 
by elements of the form
$\vec{v}_1 \otimes \cdots \otimes \vec{v}_r$, where each $\vec{v}_i$ lies in either
$\HH_1(\tT;\Z)$ or in $d \HH_1(\tT_0'';\Z)$ for some $d \in H$.  We emphasize
that the $\vec{v}_i$ are integral classes

Let $\zeta$ be the homology class of a loop around the puncture in $\tT'_0$, oriented such that the puncture lies to its right.
Note that $\zeta \in \HH_1(\tT;\Z)$; indeed, $\zeta$ is homologous to one of the boundary components
of $\tT$.  More generally, for $d \in H$ we have $d \zeta \in \HH_1(\tT;\Z)$.
To construct the cover $S_{H} \rightarrow \Sigma_{g,p}^b$ used to define $\fH_{g,p}^b(H;\bbk)$, you delete
the puncture lying in $d \tT'_0$ for each $d \in H$.  Letting $S_H \rightarrow \Sigma_{g,p}^b$ be the cover
used to define $\fH_{g,p}^b(H)$, it follows that the kernel of 
\[\rho_1\colon \fH_{g,p+1}^b(H;\bbk) = \HH_1(S'_H;\bbk) \rightarrow \HH_1(S_{H};\bbk) = \fH_{g,p}^b(H;\bbk)\]
is generated by the $d \zeta$ for $d \in H$.  Taking
the $r^{\text{th}}$ tensor power, we deduce that the kernel of the map
$\rho_1^{\otimes r}\colon \fH_{g,p+1}^b(H;\bbk)^{\otimes r} \rightarrow \fH_{g,p}^b(H;\bbk)^{\otimes r}$ is generated by elements of the form
$\vec{v}_1 \otimes \cdots \otimes \vec{v}_r$, where the $\vec{v}_i$ satisfy the following:
\begin{itemize}
\item Each $\vec{v}_i$ lies in either $\HH_1(\tT;\Z)$ or in $d \HH_1(\tT_0'';\Z)$ for some $d \in H$.
\item At least one of the $\vec{v}_i$ equals $d \zeta$ for some $d \in H$.
\end{itemize}
To prove the lemma, we must show that such elements die in the $G$-coinvariants.  

\begin{step}{2}
Consider one of the generators $\vec{v}_1 \otimes \cdots \otimes \vec{v}_r \in \fH_{g,p+1}^b(H;\bbk)^{\otimes r}$ for the kernel of $\rho_1^{\otimes r}$ identified in
Step 1, so the following hold:
\begin{itemize}
\item Each $\vec{v}_i$ lies in either $\HH_1(\tT;\Z)$ or in $d \HH_1(\tT_0'';\Z)$ for some $d \in H$.
\item At least one of the $\vec{v}_i$ equals $d \zeta$ for some $d \in H$.
\end{itemize}
Let $\Romega{H}(-,-)$ be the Reidemeister pairing on $\fH_{g,p}^b(H;\bbk)$.
We construct elements $\vec{a},\vec{b} \in \fH_{g,p+1}^b(H;\Z)$ such that the following hold.
\begin{itemize}
\item $\Romega{H}(\rho_1(\vec{a}),\rho_1(\vec{b})) = 1$.
\item $\Romega{H}(\rho_1(\vec{v}_i),\rho_1(\vec{b})) = 0$ for $1 \leq i \leq r$.
\end{itemize}
\end{step}

Let $\omega_{H}(-,-)$ be the algebraic intersection pairing on $\fH_{g,p+1}^b(H;\bbk)$.  The conditions above
on $\vec{a}$ and $\vec{b}$ are equivalent to the following:\footnote{In this formula, the element $d = 0$ is the identity in $H$,
so for this $d$ we have $d \vec{b} = \vec{b}$.  In other words, do not confuse $0 \in H$ with $0 \in \bbk$.}
\begin{itemize}
\item For all $d \in H$, we have 
\[\omega_H(\vec{a},d \vec{b}) = \begin{cases}
1 & \text{if $d = 0$},\\
0 & \text{if $d \neq 0$}.
\end{cases}\]
\item For all $d \in H$ and $1 \leq i \leq r$, we have $\omega_H(\vec{v}_i,d \vec{b}) = 0$.
\end{itemize}
This is the form in which we will verify them.

Let $\vec{w}_1,\ldots,\vec{w}_s \in \HH_1(\tT_0'';\Z)$ and $d_1,\ldots,d_s \in H$ be such
that the elements of $\{\vec{v}_1,\ldots,\vec{v}_r\}$ that lie in some $H$-translate of $\HH_1(\tT_0'';\Z)$
are precisely $\{d_1 \vec{w}_1,\ldots,d_s \vec{w}_s\}$.  
Since at least one of the $\vec{v}_i$ is of the form $d \zeta$ for some $d \in H$ and
thus does not lie in some $H$-translate of $\HH_1(\tT_0'';\Z)$, 
we have $s \leq r-1$.  Since $g \geq h+r$ this implies that $s < g-h$.  Recalling that
$\tT_0'' \cong T'' \cong \Sigma_{g-h}^1$, we can thus find\footnote{This is standard.  One source that
proves something equivalent is \cite[Proposition 3.4]{PutmanPartialTorelli}.  
This is where it is important that we are working with integral classes.}
a subsurface $\tT_0'''$ of $\tT_0''$ such that the following hold. 
\begin{itemize}
\item Each $\vec{w}_i$ lies in $\HH_1(\tT_0''';\Z)$.
\item $\tT_0''' \cong \Sigma_s^1$.
\end{itemize}
Since $\tT_0''$ has genus $g-h$ and $\tT_0'''$ has genus $s$ and $g-h > s$, we can find
$\vec{a} \in \HH_1(\tT_0'';\Z)$ and $\vec{b} \in \HH_1(\tT_0'';\Z)$ with the following two properties:
\begin{itemize}
\item $\omega_H(\vec{a},\vec{b}) = 1$.
\item For all $z \in \HH_1(\tT_0''';\bbk)$ we have
 $\omega_H(z,\vec{a}) = \omega_H(z,\vec{b}) = 0$.  In particular,
$\omega_H(\vec{w}_i,\vec{b}) = 0$ for all $1 \leq i \leq s$.
\end{itemize}
See here:\\
\centerline{\psfig{file=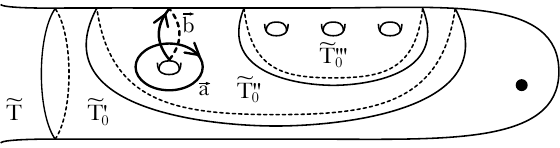,scale=1}}
Since each $\vec{v}_j$ either lies in $\HH_1(\tT;\Z)$ or is of the form
$d_i \vec{w}_i$ with $d_i \in H$, the second condition above implies that
$\omega_H(\vec{v}_i,d \vec{b}) = 0$ for all $d \in H$ and $1 \leq i \leq r$, as desired.

\begin{step}{3}
Consider one of the generators $\vec{v}_1 \otimes \cdots \otimes \vec{v}_r \in \fH_{g,p+1}^b(H;\bbk)^{\otimes r}$ for the kernel of $\rho_1^{\otimes r}$ identified in
Step 1, so the following hold:
\begin{itemize}
\item Each $\vec{v}_i$ lies in either $\HH_1(\tT;\Z)$ or in $d \HH_1(\tT_0'';\Z)$ for some $d \in H$.
\item At least one of the $\vec{v}_i$ equals $d \zeta$ for some $d \in H$.
\end{itemize}
We prove that $\vec{v}_1 \otimes \cdots \otimes \vec{v}_r$ dies in the $G$-coinvariants, where we recall that $G$ is a finite-index subgroup of $\PP_{x_0}(\Sigma_{g,p}^b,H)$.
\end{step}

To simplify our notation, we will give the details for when
$\vec{v}_1 = d_0 \zeta$ for some $d_0 \in H$.  The other cases are handled similarly.  By the previous step,
we can find $\vec{a},\vec{b} \in \fH_{g,p+1}^b(H;\Z)$ such that the following hold:
\begin{itemize}
\item $\Romega{H}(\rho_1(\vec{a}),\rho_1(\vec{b})) = 1$.
\item $\Romega{H}(\rho_1(\vec{v}_i),\rho_1(\vec{b})) = 0$ for $1 \leq i \leq r$.
\end{itemize}
Since $\vec{b}$ is an integral homology class, we can pick $\gamma \in \PP_{x_0}(\Sigma_{g,p}^b,H)$ that projects to $\rho_1(\vec{b})$ under the map
\[\rho_2\colon \PP_{x_0}(\Sigma_{g,p}^b,H) \cong \pi_1(S_{H}) \longrightarrow \HH_1(S_{H};\bbk) = \fH_{g,p}^b(H;\bbk),\]
where the $\cong$ uses the basepoint on $S_{H}$ that is surrounded by the loop in whose homology class
$\zeta \in \fH_{g,p+1}^b(H;\bbk)$ is.  Define
\[\kappa = \left(d_0 \vec{a} \right) \otimes \vec{v}_2 \otimes \cdots \otimes \vec{v}_r \in \fH_{g,p+1}^b(H;\bbk)^{\otimes r}.\]
Using Lemma \ref{lemma:pointpushreidemeister}, we have
\begin{align*}
\gamma\left(\kappa\right) &= \gamma\left(d_0 \vec{a}\right) \otimes \gamma\left(\vec{v}_2\right) \otimes \cdots \otimes \gamma\left(\vec{v}_r\right) \\
                          &= \left(d_0 \vec{a} + \Romega{H}\left(\rho_1(d_0 \vec{a}), \rho_2(\gamma) \right) \zeta\right)
                             \otimes 
                             \left(\vec{v}_2 + \Romega{H}\left(\rho_1(\vec{v}_2), \rho_2(\gamma) \right) \zeta\right)\\
                          &\quad\quad\quad   \otimes \cdots \otimes
                             \left(\vec{v}_r + \Romega{H}\left(\rho_1(\vec{v}_r), \rho_2(\gamma) \right) \zeta\right) \\
                          &= \left(d_0 \vec{a} + d_0 \zeta\right) \otimes \vec{v}_2 \otimes \cdots \otimes \vec{v}_r.
\end{align*}
Iterating this, we see that for all $m \geq 1$ we have
\[\gamma^m\left(\kappa\right) = \left(d_0 \vec{a} + m d_0 \zeta\right) \otimes \vec{v}_2 \otimes \cdots \otimes \vec{v}_r,\]
and thus
\[\gamma^m\left(\kappa\right) - \kappa = m \left(d_0 \zeta\right) \otimes \vec{v}_2 \otimes \cdots \otimes \vec{v}_r = m \vec{v}_1 \otimes \cdots \otimes \vec{v}_r.\]
Since $G$ is a finite-index subgroup of $\PP_{x_0}(\Sigma_{g,p}^b,H)$, we can pick $m \geq 1$ such that $\gamma^m \in G$, so
$m \vec{v}_1 \otimes \cdots \otimes \vec{v}_r$ dies in the $G$-coinvariants.
  Since $\bbk$ is a field of characteristic $0$, the element 
$\vec{v}_1 \otimes \cdots \otimes \vec{v}_r$ also dies
in the $G$-coinvariants, as desired.
\end{proof}

This has the following corollary:

\begin{corollary}
\label{corollary:pushtensorpowerschi}
Fix some $g,p,b \geq 0$ such that $\pi_1(\Sigma_{g,p}^b)$ is nonabelian and $p+b \geq 1$, and let
$x_0$ be a puncture of $\Sigma_{g,p+1}^b$.  Let $\ell \geq 2$ and let
$\fH_{g,p+1}^b(\uchi)$ be a homological representation of $\Mod_{g,p+1}^b(\ell)$
of size $r$ that is compatible with a genus-$h$ symplectic subgroup $H$ of $\HH_1(\Sigma_{g,p+1}^b;\Z/\ell)$.
Assume that $g \geq h+r$.  Then for all finite-index subgroups
$G$ of $\PP_{x_0}(\Sigma_{g,p}^b,H)$, we have
\[\left(\fH_{g,p+1}^b\left(\uchi\right)\right)_G \cong \fH_{g,p}^b\left(\uchi\right).\]
\end{corollary}
\begin{proof}
Let $\rho_1\colon \fH_{g,p+1}^b(H;\C) \rightarrow \fH_{g,p}^b(H;\C)$ be the map induced by filling in $x_0$.
Lemma \ref{lemma:intermediatecover} gives decompositions
\[\fH_{g,p+1}^b(H;\C) = \bigoplus_{\chi \in \hH} \fH_{g,p+1}^b(\chi;\C) \quad \text{and} \quad \fH_{g,p}^b(H;\C) = \bigoplus_{\chi \in \hH} \fH_{g,p}^b(\chi;\C),\]
and $\rho_1$ respects these direct sum decompositions.  It follows that $\fH_{g,p+1}^b(H;\C)^{\otimes r}$ and $\fH_{g,p}^b(H;\C)^{\otimes r}$ are
direct sums of the different $H$-compatible size-$r$ homological representations of $\Mod_{g,p+1}^b(\ell)$ and $\Mod_{g,p}^b(\ell)$, respectively, and
$\rho_1^{\otimes r}$ respects these direct sum decompositions.  This reduces the corollary to 
Lemma \ref{lemma:pushtensorpowers}.
\end{proof}

\section{Stability for the partial mod-\texorpdfstring{$\ell$}{l} subgroups}
\label{section:partialstability}

In \cite{PutmanPartialTorelli}, the author proved a homological stability theorem that applies
to the partial level-$\ell$ subgroups.  In this section, we explain how to generalize this to
incorporate tensor powers of the partial Prym representations.
Our theorem is as follows.  Its statement uses the conventions from \S \ref{section:symplecticconventions}. 

\begin{theorem}
\label{theorem:stability}
Let $\iota\colon \Sigma_g^b \rightarrow \Sigma_{g'}^{b'}$ be an orientation-preserving embedding
between surfaces with nonempty boundary.  For some $\ell \geq 2$, let $H$ be a genus-$h$ symplectic
subgroup of $\HH_1(\Sigma_g^b;\Z/\ell)$.  Fix some $k,r \geq 0$, and assume that $g \geq (2h+2)(k+r)+(4h+2)$.  Then
for all commutative rings $\bbk$ the induced map
\[\HH_k(\Mod_g^b(H);\fH_g^b(H;\bbk)^{\otimes r}) \rightarrow \HH_k(\Mod_{g'}^{b'}(H);\fH_{g'}^{b'}(H;\bbk)^{\otimes r})\]
is an isomorphism.
\end{theorem}
\begin{proof}
For $r=0$, this just asserts that the map
\[\HH_k(\Mod_g^b(H);\bbk) \rightarrow \HH_k(\Mod_{g'}^{b'}(H);\bbk)\]
is an isomorphism if $g \geq (2h+2)k+(4h+2)$, which is a special case of \cite[Theorem F]{PutmanPartialTorelli}.  To connect
our notation to that of \cite[Theorem F]{PutmanPartialTorelli}, we make the following remarks:
\begin{itemize}
\item First, the statement of \cite[Theorem F]{PutmanPartialTorelli} involves partitions $\cP$ and $\cP'$ of the components
of $\partial \Sigma_g^b$ and $\partial \Sigma_{g'}^{b'}$, respectively.  
Our result corresponds to the partition where all components of the boundary lie in a single partition
element.  With this convention, the map $(\Sigma_g^b,\cP) \rightarrow (\Sigma_{g'}^{b'},\cP')$ is a
``${\tt PSurf}$-morphism'', and the fact that $\Sigma_{g'}^{b'}$ has nonempty boundary implies that it
is ``partition bijective''.  Every time we refer to something in \cite{PutmanPartialTorelli} in
this proof, we implicitly use this choice of partition.
\item The statement of \cite[Theorem F]{PutmanPartialTorelli} also refers to an $A$-homology marking
$\mu\colon \HH_1^{\cP}(\Sigma_g^b) \rightarrow A$.  With the choice of partition from the previous bullet point,
we have $\HH_1^{\cP}(\Sigma_g^b) = \HH_1(\Sigma_g^b,\partial \Sigma_g^b)$.  Our marking has $A = H$, and
is the homomorphism $\mu\colon \HH_1(\Sigma_g^b,\partial \Sigma_g^b) \rightarrow H$ that equals the composition
\[\HH_1(\Sigma_g^b,\partial \Sigma_g^b) \cong \HH_1(\Sigma_g^b) \longrightarrow \HH_1(\Sigma_g^b;\Z/\ell) = H \oplus H^{\perp} \stackrel{\text{proj}}{\longrightarrow} H.\]
Here the first map comes from Poincar\'{e} duality.  With this marking, in the notation of
\cite[Theorem F]{PutmanPartialTorelli} we have
\[\Torelli(\Sigma_g^b,\cP,\mu) = \Mod_g^b(H).\]
The marking $\mu'$ on $\Sigma_{g'}^{b'}$ in \cite[Theorem F]{PutmanPartialTorelli}
is defined similarly.  The fact that $H$ is a symplectic subgroup implies that
our marking is ``supported on a symplectic subsurface''.
\end{itemize}
When $r \geq 1$, our theorem can be proven by following the proof of \cite[Theorem F]{PutmanPartialTorelli} word-for-word,
substituting the twisted homological theorem \cite[Theorem 5.2]{PutmanTwistedStability} for the ordinary
homological stability theorem, which appears as \cite[Theorem 3.1]{PutmanPartialTorelli}.

We briefly discuss some of the details of this.  We remark that the proof structure here is
inspired by a beautiful approach to homological stability for the whole mapping class
group due to Hatcher--Vogtmann \cite{HatcherVogtmannTethers}.  The proof of \cite[Theorem F]{PutmanPartialTorelli} has
two parts.  The first appears in \cite[\S 5.2-5.4]{PutmanPartialTorelli}.  These sections 
reduce the proof to what are called ``double boundary stabilizations'', i.e., where the
map $\Sigma_{g}^b \rightarrow \Sigma_{g'}^{b'}$ is as pictured here:\\
\centerline{\psfig{file=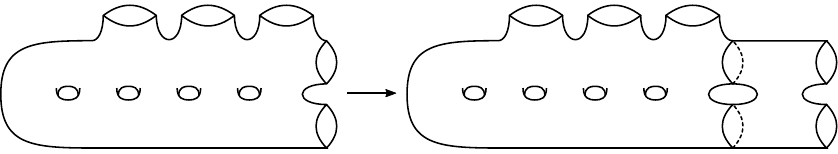,scale=1}}
This reduction does not use the homological stability machine, and no changes are needed
for the twisted version of it.

The double boundary stabilizations are handled in \cite[\S 6.8]{PutmanPartialTorelli} using
the homological stability machine.  This requires a semisimplicial set\footnote{Actually, in
the language of \S \ref{section:semisimplicial} it is an ordered simplicial complex} called
the ``complex of order-preserving double-tethered vanishing loops''.  We refer
to \cite[\S 6]{PutmanPartialTorelli} for the lengthy definition of this.  The changes that
need to be made here are as follows:
\begin{itemize}
\item As we said, the twisted homological stability theorem \cite[Theorem 5.2]{PutmanTwistedStability}
should be substituted for the ordinary homological stability theorem \cite[Theorem 3.1]{PutmanPartialTorelli}.
\item This requires constructing $\Mod_g^b(H)$-equivariant augmented coefficient systems $\cM_g^b(H;\bbk)$
on the complex of order-preserving double-tethered vanishing loops with
\[\cM_g^b(H;\bbk)\emptysimp = \fH_g^b(H;\bbk).\]
The definition of $\cM_g^b(H)$
is identical to the definition of the coefficient system $\cH_g^1(H;\bbk)$ we discussed in \S \ref{section:coefficientsystem},
and the proof that it is strongly polynomial of degree $1$ is essentially identical to the proof
of Lemma \ref{lemma:hstrongpoly}.  Using Lemma \ref{lemma:tensorpolynomial}, its tensor
power $\cM_g^b(H;\bbk)^{\otimes r}$ is strongly polynomial of degree $r$.
\item This allows you to apply Theorem \ref{theorem:vanishing} above (which is
\cite[Theorem 6.4]{PutmanTwistedStability}) to $\cM_g^b(H;\bbk)^{\otimes r}$ and
deduce that the homology of the complex of order-preserving double-tethered vanishing loops
with coefficients in $\cM_g^b(H;\bbk)^{\otimes r}$ vanishes in a range.  The needed Cohen--Macaulay
result is \cite[Theorem 6.13]{PutmanPartialTorelli}.
\item This verifies the one condition of \cite[Theorem 5.2]{PutmanTwistedStability} that
is different from \cite[Theorem 3.1]{PutmanPartialTorelli}.  The remainder of the proof
in \cite[\S 6.8]{PutmanPartialTorelli} needs no changes.\qedhere
\end{itemize}
\end{proof}

\section{Proof of main theorem for non-closed surfaces}
\label{section:proof}

We finally turn to proving our main theorems.  The following will be our main result, at
least for non-closed surfaces.  We will deal with closed surfaces later in \S \ref{section:closed}.

\begin{maintheorem}
\label{maintheorem:generalcase}
Let $g,p,b \geq 0$ and $\ell \geq 2$ be such that $p+b \geq 1$.
Let $\fH_{g,p}^b(\uchi)$ be a size-$r$ homological representation
of $\Mod_{g,p}^b(\ell)$ and let $H$ be a symplectic subgroup of $\HH_1(\Sigma_{g,p}^b(\ell);\Z/\ell)$
that is compatible with $\fH_{g,p}^b(\uchi)$.  Assume that $g \geq 2(k+r)^2+7k+6r+2$.  Then the map
\[\HH_k\left(\Mod_{g,p}^b\left(\ell\right);\fH_{g,p}^b(\uchi)\right) \rightarrow \HH_k\left(\Mod_{g,p}^b(H);\fH_{g,p}^b(\uchi)\right)\]
induced by the inclusion $\Mod_{g,p}^b(\ell) \hookrightarrow \Mod_{g,p}^b(H)$ is an isomorphism.
\end{maintheorem}

This implies Theorems \ref{maintheorem:mod} and \ref{maintheorem:modstd} for non-closed surfaces
in the following way:
\begin{itemize}
\item A size-$0$ homological representation of $\Mod_{g,p}^b(\ell)$ is simply the trivial representation
$\C$.  This is compatible with the symplectic subgroup $H=0$, for which $\Mod_{g,p}^b(H) = \Mod_{g,p}^b$.
Theorem \ref{maintheorem:generalcase} thus says that the map
\[\HH_k\left(\Mod_{g,p}^b\left(\ell\right);\C\right) \rightarrow \HH_k\left(\Mod_{g,p}^b;\C\right)\]
is an isomorphism for $g \geq 2k^2+7k+2$.  The universal coefficients theorem now implies that this
is also true with $\C$ replaced by $\Q$, which is exactly Theorem \ref{maintheorem:mod}.
\item Letting $V = \HH_1(\Sigma_{g,p}^b;\C)$, the tensor power $V^{\otimes r}$ is a size-$r$
homological representation that is compatible with $H = 0$ (see Example \ref{example:trivialcover}).  
Theorem \ref{maintheorem:generalcase} thus says that the map
\[\HH_k\left(\Mod_{g,p}^b\left(\ell\right);V^{\otimes r}\right) \rightarrow \HH_k\left(\Mod_{g,p}^b;V^{\otimes r}\right)\]
is an isomorphism for $g \geq 2(k+r)^2+7k+6r+2$.  The universal coefficients theorem now implies that this
is also true with the $\C$ in $V = \HH_1(\Sigma_{g,p}^b;\C)$ replaced by $\Q$, which is exactly Theorem \ref{maintheorem:modstd}.
\end{itemize}
We remark that Theorem \ref{maintheorem:modprym} will be a consequence of part of our proof, and we
will point out when this happens in a footnote (see the footnote on the paragraph right before Claim \ref{claim:42iso}). 

\begin{proof}[Proof of Theorem \ref{maintheorem:generalcase}]
We divide the proof into five steps.  Since the proof is organized around several interlocking inductions, we had
to write it in a certain order to make sure it was clear that the reasoning was not circular.  However, 
some of the intermediate steps might seem unmotivated upon first reading.  We thus suggest reading the steps
in the following order:
\begin{itemize}  
\item Steps \ref{step:induction}--\ref{step:reducegenuspunctures} set up the induction and make
some reductions.  They should be read first.  
\item Step \ref{step:main} is the main step 
that was sketched in \S \ref{section:proofsketch}.  We suggest reading it next.  
\item Doing this
will motivate Step \ref{step:capboundary}, whose proof depends on a calculation in Step \ref{step:pointpushing}.  It
is in Step \ref{step:pointpushing} that it becomes essential to work with general homological representations, even
though ultimately we are most interested in the trivial one.  
\end{itemize}
We remark that throughout the proof, we will
constantly use the conventions regarding symplectic subspaces from \S \ref{section:symplecticconventions}.  Also,
the number $\ell \geq 2$ will never change, but all the other parameters ($g,p,b,r,\uchi,H$, etc) will change
constantly, so we will try to be explicit about what is allowed for them at each stage of the proof.

\begin{step}{1}[Set up induction]
\label{step:induction}
We show that as an inductive hypothesis we can assume the following:
\begin{itemize}
\item[(a)] $r \geq 0$ and $k \geq 1$.
\item[(b)] We have already proved the theorem for $\HH_{i}$ for all $i < k$.
\item[(c)] For $\HH_k$, we have already proved the theorem for all homological representations of size less than $r$.  This
is vacuous if $r=0$.
\end{itemize}
\end{step}

Our proof is by induction\footnote{We could start our
induction with the trivial case $k=-1$, which would avoid having to prove any base case at
all.  However, this would lead to worse bounds.} 
on $k \geq 0$ and $r \geq 0$.  To be able to assume 
(a)-(c) as inductive hypotheses, we need to prove the theorem
for $k=0$ and general $r \geq 0$.  
So for some $g,p,b \geq 0$ with $p+b \geq 1$ let $\fH_{g,p}^b(\uchi)$ be a size-$r$ homological representation
of $\Mod_{g,p}^b(\ell)$ and let $H$ be a symplectic subgroup of $\HH_1(\Sigma_{g,p}^b(\ell);\Z/\ell)$
that is compatible with $\fH_{g,p}^b(\uchi)$.  Assume that $g \geq 2r^2+6r+2$ (which
is our claimed bound for $k=0$).
For a group $G$ acting on an abelian group $M$, the group $\HH_0(G;M)$ is the coinvariants
$M_G$, so what we have to prove is that
\[\left(\fH_{g,p}^b(\uchi)\right)_{\Mod_{g,p}^b(\ell)} \cong \left(\fH_{g,p}^b(\uchi)\right)_{\Mod_{g,p}^b(H)}.\]
Since $\Mod_{g,p}^b(\ell)$ is a subgroup of $\Mod_{g,p}^b(H)$, the right hand side is a quotient of the left hand side.
Decreasing $H$ makes $\Mod_{g,p}^b(H)$ larger, so our claim is stronger if $H$ is smaller.
By Lemma \ref{lemma:extendhomological}, the symplectic subgroup $H$ contains a symplectic subgroup that is
compatible with $\uchi$ and has genus at most $r$, so we can shrink $H$ and assume that it has genus at most $r$.

If $r=0$, then $\fH_{g,p}^b(\uchi)$ is the trivial representation and there is nothing to prove.
We can thus assume that $r \geq 1$, in which case our bound on $g$ implies that
$g \geq \max(2r,3)$, which is the bound we will actually use.
Lemma \ref{lemma:intermediatecover} implies that $\fH_{g,p}^b(\uchi)$ is a direct summand of
$\fH_{g,p}^b(H;\C)^{\otimes r}$, so it is enough to prove that
\[\left(\fH_{g,p}^b(H;\C)^{\otimes r}\right)_{\Mod_{g,p}^b(\ell)} \cong \left(\fH_{g,p}^b(H;\C)^{\otimes r}\right)_{\Mod_{g,p}^b(H)}.\]
Since $\fH_{g,p}^b(H;\C) \cong \fH_{g,p+b}(H;\C)$ (see Remark \ref{remark:boundarypunctures}) and 
the action of $\Mod_{g,p}^b(H)$ on these representations factors through $\Mod_{g,p+b}(H)$, we can assume
without loss of generality that $b = 0$, so since $b+p \geq 1$ we have $p \geq 1$.  

Assume first that $p \geq 2$, and let $x_0$ be one of the punctures.  Since $g \geq \max(2r,3)$, we have in
particular that $g \geq r+r$.  Lemma \ref{lemma:pushtensorpowers} then implies that
\[\left(\fH_{g,p}(H;\C)^{\otimes r}\right)_{\PP_{x_0}(\Sigma_{g,p-1},\ell)} \cong \fH_{g,p-1}(H;\C)^{\otimes r}.\]
This implies that
\[\left(\fH_{g,p}(H;\C)^{\otimes r}\right)_{\Mod_{g,p}(\ell)} \cong \left(\fH_{g,p-1}(H;\C)^{\otimes r}\right)_{\Mod_{g,p-1}(\ell)}.\]
A similar identity holds for $\Mod_{g,p}(H)$.  Applying this repeatedly, we reduce ourselves to the case $p=1$.  

Let $S_H \rightarrow \Sigma_g$ be the regular cover with deck group $H$ corresponding to the homomorphism
\begin{equation}
\label{eqn:commutator}
\pi_1(\Sigma_g) \longrightarrow \HH_1(\Sigma_g) = H \oplus H^{\perp} \stackrel{\text{proj}}{\longrightarrow} H.
\end{equation}
Define $\fH_g(H;\C) = \HH_1(S_H;\C)$.  What we would like to do is apply the above argument again and reduce
ourselves to the case $p=0$.  However, there is a problem: the group $\Mod_g(H)$ does not act on
$\fH_g(H;\C)$ since there is not a fixed basepoint to allow us to consistently choose a lift of a mapping
class on $\Sigma_g$ to $S_H$.  This is related to the fact that Lemma \ref{lemma:pushtensorpowers} does
not apply to the case $p=0$, and also to the fact that by Theorem \ref{theorem:birmanpartial} we have
\[\PP_{x_0}(\Sigma_g,\ell) = \PP_{x_0}(\Sigma_g,H) = \pi_1(\Sigma_g).\]
However, let $K \lhd \pi_1(\Sigma_g)$ be the kernel of the map \eqref{eqn:commutator}.  Regard
$K$ as a subgroup of $\PP_{x_0}(\Sigma_g,\ell)$.
The proof of Lemma \ref{lemma:pushtensorpowers} goes through without changes\footnote{The 
only place in the proof where the assumption $p \geq 1$ is used is in the invocation
of Lemma \ref{lemma:pointpushreidemeister}.  This lemma works for elements of $K$, which are exactly
the loops in the base space that lift to loops in the cover.}
that
\[\left(\fH_{g,1}(H;\C)^{\otimes r}\right)_{K} \cong \fH_{g}(H;\C)^{\otimes r}.\]
The groups $\Gamma(\ell) = \Mod_{g,1}(\ell)/K$ and $\Gamma(H) = \Mod_{g,1}(H)/K$ thus act on $\fH_{g}(H;\C)$,
and we are reduced to proving that
\[\left(\fH_{g}(H;\C)^{\otimes r}\right)_{\Gamma(\ell)} \cong \left(\fH_{g}(H;\C)^{\otimes r}\right)_{\Gamma(H)}.\]
Since $\PP_{x_0}(\Sigma_g,\ell) / K \cong H$ and $\Mod_{g,1}(\ell)$ acts trivially on $H$, 
the Birman exact sequence for $\Mod_{g,1}(\ell)$ from
Theorem \ref{theorem:birmanpartial} quotients down to a central extension
\[1 \longrightarrow H \longrightarrow \Gamma(\ell) \longrightarrow \Mod_g(\ell) \longrightarrow 1.\]
The action of the central subgroup $H$ on $\fH_g(H;\C)$ come from the action of $H$ on
$S_H$ as deck transformations.  There is a similar exact sequence for $\Gamma(H)$.

Since $g \geq \max(2r,3)$, we in particular have $g \geq 3$.  In that case,
Looijenga \cite{LooijengaPrym} proved that the action
of $\Gamma(H)$ on $\fH_g(H;\C)$ comes from a representation of $\Gamma(H)$ into a connected semisimple
$\R$-algebraic group $\bG$ without compact factors, and the image of $\Gamma(H)$ in $\bG(\R)$ is a lattice.
The Borel density theorem \cite{BorelDensity} says that lattices in such Lie groups are Zariski
dense, which implies that
\[\left(\fH_{g}(H;\C)^{\otimes r}\right)_{\Gamma(H)} \cong \left(\fH_{g}(H;\C)^{\otimes r}\right)_{\bG(\R)}.\]
The group $\Gamma(\ell)$ is a finite-index subgroup of $\Gamma(H)$, and thus its image in
$\bG(\R)$ is also a lattice and
\[\left(\fH_{g}(H;\C)^{\otimes r}\right)_{\Gamma(\ell)} \cong \left(\fH_{g}(H;\C)^{\otimes r}\right)_{\bG(\R)}.\]
The desired result follows.

\begin{step}{2}[Initial reductions]
\label{step:reducegenuspunctures}
For some $k$ and $r$, make the inductive hypotheses (a)-(c) from Step \ref{step:induction}:
\begin{itemize}
\item[(a)] $r \geq 0$ and $k \geq 1$.
\item[(b)] We have already proved the theorem for $\HH_{i}$ for all $i < k$.
\item[(c)] For $\HH_k$, we have already proved the theorem for all homological representations of size less than $r$.  This
is vacuous if $r=0$.
\end{itemize}
We show that to prove the theorem for $\HH_k$ and homological representations of size $r$, it is enough to prove it under the following simplifying assumptions:
\begin{itemize}
\item[($\dagger$)] The surface has no punctures.
\item[($\dagger\dagger$)] The symplectic subgroup $H$ has genus at most $r$.
\end{itemize}
\end{step}

Throughout this step, all surfaces $\Sigma_{g,p}^b$ satisfy the hypothesis $p+b \geq 1$ from the statement of the theorem.
We start by showing that we can assume ($\dagger$):

\begin{claim}{2.1}
\label{claim:reducepunctures}
Assume that for $\HH_k$ and homological representations of size $r$, the theorem is true for surfaces without punctures.  Then for $\HH_k$ and
homological representations of size $r$ it is true in general.
\end{claim}
\begin{proof}[Proof of claim]
The proof is by induction on the number of punctures.  The base case $p=0$ is our assumption.
For the inductive step, assume that for $\HH_k$ the theorem is true for surfaces with $p$ punctures.  We will prove it for surfaces with $(p+1)$ punctures
as follows.  

Consider a size-$r$ homological representation $\fH_{g,p+1}^b(\uchi)$ and a symplectic subgroup
$H$ of $\HH_1(\Sigma_{g,p+1}^b;\Z/\ell)$ that is compatible with $\fH_{g,p+1}^b(\uchi)$.  Assume that
$g \geq 2(k+r)^2+7k+6r+2$.  Since $\HH_1(\Sigma_{g,p+1}^b;\Z/\ell) = \HH_1(\Sigma_{g,p}^{b+1};\Z/\ell)$,
we can identify $H$ with a symplectic subgroup of $\HH_1(\Sigma_{g,p}^{b+1};\Z/\ell)$.  Using
Proposition \ref{proposition:cappartial} and Remark \ref{remark:caplevel}, we have a commutative diagram of central extensions
\begin{equation}
\label{eqn:cappartial}
\begin{tikzcd}[row sep=scriptsize]
1 \arrow{r} & \Z \arrow{r} \arrow{d}{=} & \Mod_{g,p}^{b+1}(\ell) \arrow{r} \arrow{d} & \Mod_{g,p+1}^b(\ell) \arrow{r} \arrow{d} & 1 \\
1 \arrow{r} & \Z \arrow{r}              & \Mod_{g,p}^{b+1}(H)    \arrow{r}           & \Mod_{g,p+1}^b(H)    \arrow{r}           & 1
\end{tikzcd}
\end{equation}
whose central $\Z$ is generated by the Dehn twist
about a boundary component $\partial$ of $\Sigma_{g,p}^{b+1}$.  We have
$\fH_{g,p+1}^b(\uchi) \cong \fH_{g,p}^{b+1}(\uchi)$ (c.f.\ Remark \ref{remark:boundarypunctures}), and
$T_{\partial}$ acts trivially on $\fH_{g,p}^{b+1}(\uchi)$.  Let $V = \fH_{g,p+1}^b(\uchi) = \fH_{g,p}^{b+1}(\uchi)$.  The two-row Hochschild--Serre spectral
sequences associated to the short exact sequences in \eqref{eqn:cappartial} turn into long
exact Gysin sequences, and we have a map between these Gysin sequences containing the following.
To save horizontal space we have written $M$ rather than $\Mod$ and also omitted the coefficients, which should all be $V$.
\[\begin{tikzcd}[font=\footnotesize, column sep=small, row sep=scriptsize]
\HH_{k-1}(M_{g,p+1}^b(\ell)) \arrow{r} \arrow{d}{f_1} & \HH_{k}(M_{g,p}^{b+1}(\ell)) \arrow{r} \arrow{d}{f_2} & \HH_{k}(M_{g,p+1}^{b}(\ell)) \arrow{r} \arrow{d}{f_3} & \HH_{k-2}(M_{g,p+1}^b(\ell)) \arrow{r} \arrow{d}{f_4} & \HH_{k-1}(M_{g,p}^{b+1}(\ell)) \arrow{d}{f_5} \\
\HH_{k-1}(M_{g,p+1}^b(H))    \arrow{r}                & \HH_{k}(M_{g,p}^{b+1}(H))    \arrow{r}                & \HH_{k}(M_{g,p+1}^{b}(H))                   \arrow{r} & \HH_{k-2}(M_{g,p+1}^b(H))    \arrow{r}                & \HH_{k-1}(M_{g,p}^{b+1}(H))
\end{tikzcd}\]
Our inductive hypothesis (b) implies that $f_1$ and $f_4$ and $f_5$ are isomorphisms, and our induction on $p$ implies that $f_2$ is an isomorphism.
The five-lemma now implies that $f_3$ is an isomorphism, as desired.
\end{proof}

We next show that we can assume ($\dagger\dagger$):

\begin{claim}{2.2}
\label{claim:reducegenus}
Assume that for $\HH_k$ and homological representations of size $r$, the theorem is true when the symplectic subgroup $H$ has genus at most $r$.  
Then for $\HH_k$ and homological representations of size $r$ it is true in general.
\end{claim}
\begin{proof}[Proof of claim]
Consider a size-$r$ homological representation $\fH_{g,p}^b(\uchi)$ of $\Mod_{g,p}^b(\ell)$ that is compatible with a symplectic subgroup $H$ of $\HH_1(\Sigma_{g,p}^b;\Z/\ell)$.
Assume that $g \geq 2(k+r)^2+7k+6r+2$.
By Lemma \ref{lemma:extendhomological}, we can find a genus at most $r$ symplectic subgroup $H'$ of $\HH_1(\Sigma_{g,p}^b;\Z/\ell)$
with $H' < H$ that is compatible with $\fH_{g,p}^b(\uchi)$.  The group $\Mod_{g,p}^b(H)$ is a finite-index subgroup of $\Mod_{g,p}^b(H')$, so by the transfer map lemma (Lemma \ref{lemma:transfer})
we see that each map in
\[\HH_k(\Mod_{g,p}^b(\ell);\fH_{g,p}^b(\uchi)) \rightarrow \HH_k(\Mod_{g,p}^b(H);\fH_{g,p}^b(\uchi)) \rightarrow \HH_k(\Mod_{g,p}^b(H');\fH_{g,p}^b(\uchi))\]
is a surjection.  Our assumption says that $\HH_k(\Mod_{g,p}^b(\ell);\fH_{g,p}^b(\uchi)) \rightarrow \HH_k(\Mod_{g,p}^b(H');\fH_{g,p}^b(\uchi))$ is an isomorphism,
so we deduce that $\HH_k(\Mod_{g,p}^b(\ell);\fH_{g,p}^b(\uchi)) \rightarrow \HH_k(\Mod_{g,p}^b(H);\fH_{g,p}^b(\uchi))$ is an isomorphism.  The
claim follows.
\end{proof}

\begin{step}{3}[Point-pushing coefficients]
\label{step:pointpushing}
For some $k$ and $r$, make the inductive hypotheses (a)-(c) from Step \ref{step:induction}:
\begin{itemize}
\item[(a)] $r \geq 0$ and $k \geq 1$.
\item[(b)] We have already proved the theorem for $\HH_{i}$ for all $i < k$.
\item[(c)] For $\HH_k$, we have already proved the theorem for all homological representations of size less than $r$.  This
is vacuous if $r=0$.
\end{itemize}
We study their consequences for the point-pushing subgroup.
\end{step}

It will take a bit of work to state the result we will derive from (a)--(c).
Fix some $g \geq 0$ and $b \geq 1$ such that $\pi_1(\Sigma_g^b)$ is nonabelian, and let $x_0$ be the puncture
of $\Sigma_{g,1}^b$.  Let $H$ be a symplectic subgroup of $\HH_1(\Sigma_{g,1}^b;\Z/\ell)$.  
Theorem \ref{theorem:birmanpartial} gives a Birman exact sequence
\[1 \longrightarrow \PP_{x_0}(\Sigma_g^b,H) \longrightarrow \Mod_{g,1}^b(H) \longrightarrow \Mod_g^b(H) \longrightarrow 1.\]
If $U$ is a representation of $\Mod_{g,1}^b(H)$, then
the action of $\Mod_{g,1}^b(H)$ on $U$ along with the conjugation action of $\Mod_{g,1}^b(H)$ on $\PP_{x_0}(\Sigma_g^b,H)$
gives an action of $\Mod_{g,1}^b(H)$ on $\hU = \HH_1(\PP_{x_0}(\Sigma_g^b,H);U)$.  Since inner automorphisms
act trivially on homology (see, e.g., \cite[Proposition III.8.1]{BrownCohomology}), this descends to 
an action of $\Mod_g^b(H)$ on $\hU$.  This action can be restricted
to $\Mod_g^b(\ell)$.  Our goal in this step is
to prove that our inductive hypotheses can be applied to show that for certain $U$, the map
\[\HH_i(\Mod_g^b(\ell);\hU) \rightarrow \HH_i(\Mod_g^b(H);\hU)\]
is an isomorphism once $g$ is sufficiently large.  What we prove is as follows.  We state it in terms of $\HH_{i-1}$ rather
than $\HH_i$ since that will be how we use it in the next step, and this will make it easier
to verify the genus bounds in its hypotheses.  The numbers $k \geq 1$ and $r \geq 0$ in the statement
of this claim are from the inductive hypotheses (a)-(c).

\begin{claim}{3.1}
\label{claim:31ppu}
For some $g \geq 0$ and $b \geq 1$, let $U$ be a homological representation of $\Mod_{g,1}^b(\ell)$ of size $r' \geq 0$ that is compatible with a symplectic subgroup
$H$ of $\HH_1(\Sigma_{g,1}^b;\Z/\ell)$.  Fix some $1 \leq i \leq k$, and if $i=k$ then assume that $r' \leq r$.  Let $x_0$ be the puncture
of $\Sigma_{g,1}^b$ and let $\hU = \HH_1(\PP_{x_0}(\Sigma_g^b,H);U)$.  Assume
that $g \geq 2(i+r')^2+7i+6r'+1$, and also that $g \geq r' + h$ with $h$ the genus of $H$.  Then the map
\[\HH_{i-1}(\Mod_g^b(\ell);\hU) \rightarrow \HH_{i-1}(\Mod_g^b(H);\hU)\]
is an isomorphism.
\end{claim}
\begin{proof}[Proof of claim]
To allow an inductive proof, we will allow $U$ to be a bit more general.
Say that a $\Mod_{g,1}^b(\ell)$-representation $V$ is an {\em extended homological representation} of size $r'$ if it can be written as $V = V_1 \otimes \cdots \otimes V_{r'}$, where
each $V_i$ is either $\fH_{g,1}^b(\chi_i)$ or $\fH_g^b(\chi_i)$ for some character $\chi_i$
of $\cD = \HH_1(\Sigma_g;\Z/\ell)$.  Here $\Mod_{g,1}^b(\ell)$ acts on $\fH_g^b(\chi_i)$ via the projection
$\Mod_{g,1}^b(\ell) \rightarrow \Mod_g^b(\ell)$ that fills in the puncture $x_0$.  

The number $s'$ of
tensor factors of the form $\fH_{g,1}^b(\chi_i)$ will be called the {\em nonextended size} of $V$.  Thus
$s' \leq r'$, with equality precisely when $V$ is a normal homological representation.
We say that $V$ is compatible with our symplectic subspace $H$ of $\HH_1(\Sigma_{g,1}^b;\Z/\ell)$ if each
$\chi_i$ is compatible with $H$.  This implies that the action of $\Mod_{g,1}^b(\ell)$ on $V$ extends
to $\Mod_{g,1}^b(H)$.

We will prove the claim for an extended homological representation $U$.
Let $0 \leq s' \leq r'$ be the nonextended size of $U$.  The proof will be by induction on $r'$, and for a fixed $r'$ will be
by induction on $s'$.  The base case is $r' \geq 0$ arbitrary (subject to the conditions in the claim!) and $s'=0$.
In this case, the action of $\Mod_{g,1}^{b}(H)$ on $U$ factors through $\Mod_g^b(H)$ and $U$ is a homological
representation of $\Mod_g^b(\ell)$ of size $r'$ that is compatible with $H$.  It follows that the action on $U$ of the kernel
$\PP_{x_0}(\Sigma_g^b,H)$ of $\Mod_{g,1}^{b}(H) \rightarrow \Mod_g^b(H)$ is trivial, so
\[\hU = \HH_1(\PP_{x_0}(\Sigma_g^b,H);U) \cong \HH_1(\PP_{x_0}(\Sigma_g^b,H);\C) \otimes U \cong \fH_g^b(H;\C) \otimes U.\]
Here we are using the fact from Theorem \ref{theorem:birmanpartial} that $\PP_{x_0}(\Sigma_g^b,H)$ is the kernel of the map
\[\PP_{x_0}(\Sigma_g^b) \cong \pi_1(\Sigma_g^b,x_0) \rightarrow \HH_1(\Sigma_g^b) = H \oplus H^{\perp} \stackrel{\text{proj}}{\longrightarrow} H,\]
so it can be identified with the fundamental group of the cover $S_H$ of $\Sigma_g^b$ used to define $\fH_g^b(H;\C) = \HH_1(S_H;\C)$.
Using Lemma \ref{lemma:intermediatecover}, the representation $\fH_g^b(H;\C) \otimes U$ is a direct sum
of homological representations of $\Mod_g^b(H)$ of size $r'+1$.  Since $i \leq k$, we have $i-1 < k$, so
we can apply our inductive hypothesis (b) to deduce that the map
\[\HH_{i-1}(\Mod_{g}^b(\ell);\fH_g^b(H;\C) \otimes U) \rightarrow \HH_{i-1}(\Mod_{g}^b(H);\fH_g^b(H;\C) \otimes U)\]
is an isomorphism.  Here we are using the fact that our genus assumption is
\[g \geq 2(i+r')^2+7i+6r'+1 = 2((i-1)+(r'+1))^2+7(i-1) + 6(r'+1) +2.\]
We remark that this is the origin of the bound in this claim.  This completes the proof of the base case.

We can now assume that $1 \leq s' \leq r'$ and that the claim is true whenever either $r'$ or $s'$ is smaller.
Reordering the tensor
factors of $U$ if necessary, we can write $U = U' \otimes \fH_{g,1}^b(\chi)$ for some extended homological representation
$U'$ of size $r'-1$ and nonextended size $s'-1$ and some character $\chi$ that is compatible with $H$.  Lemma \ref{lemma:capboundary} gives a short
exact sequence
\[0 \longrightarrow \C \longrightarrow \fH_{g,1}^b(\chi) \longrightarrow \fH_g^b(\chi) \longrightarrow 0\]
of $\Mod_{g,1}^b(H)$-representations.  Define $U'' = U' \otimes \fH_{g}^b(\chi)$, so
$U''$ is an extended homological representation of size $r'$ and nonextended size $s'-1$.  Tensoring
our exact sequence with $U'$, we get a short exact sequence
\begin{equation}
\label{eqn:21ushort}
0 \longrightarrow U' \longrightarrow U \longrightarrow U'' \longrightarrow 0
\end{equation}
of $\Mod_{g,1}^b(H)$-representations.  There is an associated long exact sequence in
$\PP_{x_0}(\Sigma_g^b,H)$-homology.  Since $\PP_{x_0}(\Sigma_g^b,H)$ is a free group, this involves
homology in degrees $0$ and $1$.  As notation, let
\[\hU = \HH_1(\PP_{x_0}(\Sigma_g^b,H);U),\quad \hU' = \HH_1(\PP_{x_0}(\Sigma_g^b,H);U'),\quad \hU'' = \HH_1(\PP_{x_0}(\Sigma_g^b,H);U'')\]
and
\[\oU = \HH_0(\PP_{x_0}(\Sigma_g^b,H);U),\quad \oU' = \HH_0(\PP_{x_0}(\Sigma_g^b,H);U'),\quad \oU'' = \HH_0(\PP_{x_0}(\Sigma_g^b,H);U'').\]
The long exact sequence in $\PP_{x_0}(\Sigma_g^b,H)$-homology associated to \eqref{eqn:21ushort} is thus of the form
\[0 \longrightarrow \hU' \longrightarrow \hU \longrightarrow \hU'' \longrightarrow \oU' \longrightarrow \oU \longrightarrow \oU'' \longrightarrow 0.\]
One of our genus assumptions is that $g \geq h+r'$ where $h$ is the genus of $H$.
Thus Corollary \ref{corollary:pushtensorpowerschi} implies\footnote{To
apply this to extended homological representations, we factor out the extended part.  For instance,
if $U = \fH_{g}^b(\chi_1) \otimes \fH_g^b(\chi_2) \otimes \fH_{g,1}^b(\chi_3) \otimes \fH_{g,1}^b(\chi_4)$, then
Corollary \ref{corollary:pushtensorpowerschi} implies that $\oU = U_{\PP_{x_0}(\Sigma_g^b,H)}$ equals
\[\fH_{g}^b(\chi_1) \otimes \fH_g^b(\chi_2) \otimes \left(\fH_{g,1}^b(\chi_3) \otimes \fH_{g,1}^b(\chi_4)\right)_{\PP_{x_0}(\Sigma_g^b,H)} = \fH_{g}^b(\chi_1) \otimes \fH_g^b(\chi_2) \otimes \fH_{g}^b(\chi_3) \otimes \fH_{g}^b(\chi_4).\]
Here we are using the fact that $\PP_{x_0}(\Sigma_g^b,H)$ acts trivially on the first two factors.}
that $\oU \cong \oU''$.  Letting $Q$ be the image of the map $\hU \rightarrow \hU''$, this implies that we
have short exact sequences
\begin{equation}
\label{eqn:21qseq1}
0 \longrightarrow \hU' \longrightarrow \hU \longrightarrow Q \longrightarrow 0
\end{equation}
and
\begin{equation}
\label{eqn:21qseq2}
0 \longrightarrow Q \longrightarrow \hU'' \longrightarrow \oU' \longrightarrow 0.
\end{equation}
To simplify our notation, let $M(\ell) = \Mod_{g}^b(\ell)$ and $M(H) = \Mod_{g}^b(H)$.
Our goal is to prove that the map
\begin{equation}
\label{eqn:21biggoal}
\HH_{i-1}(M(\ell);\hU) \rightarrow \HH_{i-1}(M(H);\hU)
\end{equation}
is an isomorphism. Both \eqref{eqn:21qseq1} and \eqref{eqn:21qseq2} induce
exact sequences in the homology of $M(\ell)$ and
$M(H)$, and also a map between these long exact sequences.

For \eqref{eqn:21qseq2}, this contains the segment
\[\begin{tikzcd}[font=\footnotesize, column sep=small, row sep=scriptsize]
\HH_{i}(M(\ell);\hU'') \arrow{r} \arrow{d}{f_1} & \HH_{i}(M(\ell);\oU') \arrow{r} \arrow{d}{f_2} & \HH_{i-1}(M(\ell);Q) \arrow{r} \arrow{d}{f_3} & \HH_{i-1}(M(\ell);\hU'') \arrow{r} \arrow{d}{f_4} & \HH_{i-1}(M(\ell);\oU') \arrow{d}{f_5} \\
\HH_{i}(M(H);\hU'')    \arrow{r}                & \HH_{i}(M(H);\oU')    \arrow{r}                & \HH_{i-1}(M(H);Q)    \arrow{r}                & \HH_{i-1}(M(H);\hU'')    \arrow{r}                & \HH_{i-1}(M(H);\oU')
\end{tikzcd}\]
We can understand the maps $f_i$ as follows:
\begin{itemize}
\item The transfer map lemma (Lemma \ref{lemma:transfer}) implies that $f_1$ is a surjection.
\item By construction, $U'$ is an extended homological representation of $\Mod_{g,1}^{b}(H)$ of size $r'-1$, so
just like above we can use Corollary \ref{corollary:pushtensorpowerschi} to see that
$\oU'$ is a homological representation of
$\Mod_g^b(H)$ of size $r'-1$.  Recall that we are assuming that $i \leq k$ and that if $i = k$ then $r' \leq r$ (so $r'-1 < r$).
We can therefore apply our inductive hypotheses (b) and (c) to see that $f_2$ and $f_5$ are isomorphisms.
\item By construction, $U''$ is an extended homological representation of size $r'$ and nonextended size $s'-1$.
By our induction on the nonextended size, we see that $f_4$ is an isomorphism.
\end{itemize}
Applying the five-lemma, we deduce that $f_3$ is an isomorphism.

We now turn to the long exact sequences in $M(\ell)$ and $M(H)$ homology induced by \eqref{eqn:21qseq1}.  These
contain
\[\begin{tikzcd}[font=\footnotesize, column sep=small, row sep=scriptsize]
\HH_{i}(M(\ell);Q) \arrow{r} \arrow{d}{f_6} & \HH_{i-1}(M(\ell);\hU') \arrow{r} \arrow{d}{f_7} & \HH_{i-1}(M(\ell);\hU) \arrow{r} \arrow{d}{f_8} & \HH_{i-1}(M(\ell);Q) \arrow{r} \arrow{d}{f_3} & \HH_{i-2}(M(\ell);\hU') \arrow{d}{f_9} \\
\HH_{i}(M(H);Q)    \arrow{r}                & \HH_{i-1}(M(H);\hU')    \arrow{r}                & \HH_{i-1}(M(H);\hU)    \arrow{r}                & \HH_{i-1}(M(H);Q)    \arrow{r}                & \HH_{i-2}(M(H);\hU')
\end{tikzcd}\]
Note that the map $f_3$ here is the same as the one from the previous diagram.  We can understand these new maps $f_i$ as follows:
\begin{itemize}
\item The transfer map lemma (Lemma \ref{lemma:transfer}) implies that $f_6$ is a surjection.
\item By construction, $U'$ is an extended homological representation of $\Mod_{g,1}^{b}(H)$ of size $r'-1$, so by our induction
on $r'$ we see that the maps $f_7$ and $f_9$ are isomorphisms.
\item We proved above that $f_3$ is an isomorphism.
\end{itemize}
Applying the five-lemma, we deduce that $f_8$ is an isomorphism.  This is exactly the map \eqref{eqn:21biggoal} we were supposed
to prove is an isomorphism, so this completes the proof of the claim.
\end{proof}

\begin{step}{4}[Capping the boundary]
\label{step:capboundary}
For some $k$ and $r$, make the inductive hypotheses (a)-(c) from Step \ref{step:induction}:
\begin{itemize}
\item[(a)] $r \geq 0$ and $k \geq 1$.
\item[(b)] We have already proved the theorem for $\HH_{i}$ for all $i < k$.
\item[(c)] For $\HH_k$, we have already proved the theorem for all homological representations of size less than $r$.  This
is vacuous if $r=0$.
\end{itemize}
Let $g \geq 0$ and $b \geq 1$.  Let $\fH_{g}^{b+1}(\uchi)$ be a size-$r$ homological representation
of $\Mod_{g}^{b+1}(\ell)$ and let $H$ be a symplectic subgroup of $\HH_1(\Sigma_{g}^{b+1}(\ell);\Z/\ell)$
that is compatible with $\fH_{g}^{b+1}(\uchi)$ and has genus at most $r$.  
Assume\footnote{Note that this is $1$ less than the bound we are trying to prove for $\HH_k$.} that $g \geq 2(k+r)^2+7k+6r+1$.
Then we prove that the map
\begin{equation}
\label{eqn:captoprove}
\HH_k\left(\Mod_{g}^{b+1}\left(\ell\right);\fH_{g}^{b+1}(\uchi)\right) \rightarrow \HH_k\left(\Mod_{g}^b\left(\ell\right);\fH_{g}^b(\uchi)\right)
\end{equation}
induced by gluing a disc to a boundary component $\partial$ of $\Sigma_g^{b+1}$ is an isomorphism.
\end{step}

We will need some of our initial calculations in this step to hold more generally when $H$ has
genus at most $r+1$, so for the moment we only impose this weaker condition.  At the very end
we will re-impose the condition that the genus of $H$ is at most $r$.

To simplify our notation, let $V = \fH_{g}^{b+1}(\uchi)$ and $W = \fH_{g}^b(\uchi)$.  The map \eqref{eqn:captoprove} fits into a commutative diagram
\begin{equation}
\label{eqn:captoprove2}
\begin{tikzcd}[row sep=scriptsize]
\HH_k\left(\Mod_{g}^{b+1}\left(\ell\right);V\right) \arrow{r} \arrow{d} & \HH_k\left(\Mod_{g}^b\left(\ell\right);W\right) \arrow{d} \\
\HH_k\left(\Mod_{g}^{b+1}\left(H\right);V)\right)   \arrow{r}           & \HH_k\left(\Mod_{g}^b\left(H\right);W\right).
\end{tikzcd}
\end{equation}
By Lemma \ref{lemma:intermediatecover}, the representations $V = \fH_{g}^{b+1}(\uchi)$ and
$W = \fH_g^b(\uchi)$ are direct summands of $\fH_g^{b+1}(H;\C)^{\otimes r}$ and
$\fH_g^b(H;\C)^{\otimes r}$, respectively.  Since $H$ has genus at most $r+1$, Theorem \ref{theorem:stability} 
implies that the bottom horizontal
map in \eqref{eqn:captoprove2} is an isomorphism if $g \geq (2r+4)(k+r)+(4r+6)$.  We will derive 
that the top horizontal map is an isomorphism from this, 
which first requires verifying that our genus assumption $g \geq 2(k+r)^2+7k+6r+1$ implies that
we are in the stable range $g \geq (2r+4)(k+r)+(4r+6)$ from Theorem \ref{theorem:stability}:
\begin{align*}
2(k+r)^2+7k+6r+1 &=    2(k+r+1)(k+r) + 5k + 4r + 1 \\
                 &\geq 2(r+2)(k+r)+ 4r + (5k+1) \\
                 &\geq (2r+4)(k+r) + 4r + 6.
\end{align*}
Here both inequalities use the fact that $k \geq 1$, which is our inductive hypothesis (a).

Since $b \geq 1$, the maps
$\Mod_g^{b+1}(\ell) \rightarrow \Mod_g^b(\ell)$ and $\Mod_g^{b+1}(H) \rightarrow \Mod_g^b(H)$ 
induced by gluing a disc to $\partial$ split via maps induced by an embedding $\Sigma_g^b \hookrightarrow \Sigma_g^{b+1}$ as follows:\\
\centerline{\psfig{file=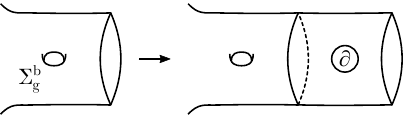,scale=1}}
A similar map gives a splitting of the map $V \rightarrow W$ induced by gluing a disc to $\partial$.  These
give compatible splitting of the top and bottom rows of \eqref{eqn:captoprove2}, which in particular
imply that they are surjections (as we already know for the bottom row).  We thus must prove that
the top row of \eqref{eqn:captoprove2} is an injection.

We can factor the horizontal maps in \eqref{eqn:captoprove2} as follows:
\[\begin{tikzcd}[row sep=scriptsize]
\HH_k\left(\Mod_{g}^{b+1}\left(\ell\right);V\right) \arrow{r}{\phi}  \arrow{d} & \HH_k\left(\Mod_{g,1}^b\left(\ell\right);V\right) \arrow{r}{\psi}  \arrow{d} & \HH_k\left(\Mod_{g}^b\left(\ell\right);W\right) \arrow{d} \\
\HH_k\left(\Mod_{g}^{b+1}\left(H\right);V)\right)   \arrow{r}{\ophi}           & \HH_k\left(\Mod_{g,1}^b\left(H\right);V\right)    \arrow{r}{\opsi}           & \HH_k\left(\Mod_{g}^b\left(H\right);W\right).
\end{tikzcd}\]
Here we are using the fact that $V = \fH_g^{b+1}(\uchi) = \fH_{g,1}^b(\uchi)$ (see Remark \ref{remark:boundarypunctures}).
To prove that the top horizontal map in \eqref{eqn:captoprove2} is an injection, it
is enough to prove that $\ker(\phi)=0$ and $\ker(\psi) \cap \Image(\phi) = 0$.  We will
derive this from the fact that $\opsi \circ \ophi$ is an isomorphism (Theorem \ref{theorem:stability}, as noted above), which
implies in particular that $\ker(\ophi) = 0$ and that $\ker(\opsi) \cap \Image(\ophi) = 0$.
We start by showing that $\ker(\phi)=0$:

\begin{claim}{4.1}
$\ker(\phi)=0$.
\end{claim}
\begin{proof}[Proof of claim]
Proposition \ref{proposition:cappartial} and Remark \ref{remark:caplevel} 
give a commutative diagram of central extensions
\begin{equation}
\label{eqn:centraltogysin}
\begin{tikzcd}[row sep=scriptsize]
1 \arrow{r} & \Z \arrow{r} \arrow{d}{=} & \Mod_{g}^{b+1}(\ell) \arrow{r} \arrow{d} & \Mod_{g,1}^b(\ell) \arrow{r} \arrow{d} & 1 \\
1 \arrow{r} & \Z \arrow{r}              & \Mod_{g}^{b+1}(H)    \arrow{r}           & \Mod_{g,1}^b(H)    \arrow{r}           & 1 
\end{tikzcd}
\end{equation}
where the central $\Z$ is generated by the Dehn twist $T_{\partial}$.  Consider the two
associated Hochschild--Serre spectral sequences with coefficients in $V$ associated to the
central extensions in \eqref{eqn:centraltogysin}.  These spectral sequences have two
potentially nonzero rows, so they encode long exact Gysin sequences.  There is a map
between these Gysin sequences, which contains the following:
\[\begin{tikzcd}[font=\footnotesize, column sep=small, row sep=scriptsize]
\HH_{k+1}(\Mod_{g,1}^b(\ell);V) \arrow{r}{\phi'}  \arrow{d}{f_1} & \HH_{k-1}(\Mod_{g,1}^b(\ell);V) \arrow{r} \arrow{d}{f_2} & \HH_k(\Mod_g^{b+1}(\ell);V) \arrow{r}{\phi}  \arrow{d} & \HH_k(\Mod_{g,1}^b(\ell);V) \arrow{d} \\
\HH_{k+1}(\Mod_{g,1}^b(H);V)    \arrow{r}{\ophi'}                & \HH_{k-1}(\Mod_{g,1}^b(H);V)    \arrow{r}                & \HH_k(\Mod_g^{b+1}(H);V)    \arrow{r}{\ophi}           & \HH_k(\Mod_{g,1}^b(H);V).
\end{tikzcd}\]
To prove that $\ker(\phi)=0$, it is enough to prove that $\phi'$ is a surjection.  We know that $\ker(\ophi)=0$,
so $\ophi'$ is a surjection.   
Our inductive hypothesis (b) 
implies that $f_2$ is an isomorphism, and using the transfer map lemma (Lemma \ref{lemma:transfer})
we see that $f_1$ is a surjection.  It
follows that $\phi'$ is a surjection, as desired.
\end{proof}

The proof that $\ker(\psi) \cap \Image(\phi) = 0$ is a little more complicated.  A proof identical to the one
in the above claim shows that the maps
\[\HH_i(\Mod_{g}^{b+1}(\ell);V) \rightarrow \HH_{i}(\Mod_{g,1}^b(\ell);V) \quad \text{and} \quad
\HH_i(\Mod_{g}^{b+1}(H);V) \rightarrow \HH_{i}(\Mod_{g,1}^b(H);V)\]
are injections for $0 \leq i \leq k$.  It follows that up to degree $k$, the Gysin sequences
discussed in the proof of the above claim break up into short exact sequences.  In particular,
we have the following commutative diagram with exact rows:
\begin{equation}
\label{eqn:gysinbreak}
\begin{tikzcd}[column sep=1em, row sep=1em]
0 \arrow{r} & \HH_k(\Mod_g^{b+1}(\ell);V) \arrow{r}{\phi}  \arrow{d} & \HH_k(\Mod_{g,1}^b(\ell);V) \arrow{r}{\mu}  \arrow{d} & \HH_{k-2}(\Mod_{g,1}^b(\ell);V) \arrow{r} \arrow{d}{\cong} & 0 \\
0 \arrow{r} & \HH_k(\Mod_g^{b+1}(H);V)    \arrow{r}{\ophi}           & \HH_k(\Mod_{g,1}^b(H);V)    \arrow{r}{\omu}           & \HH_{k-2}(\Mod_{g,1}^b(H);V)    \arrow{r}                  & 0
\end{tikzcd}
\end{equation}
The isomorphism on the right-most vertical arrow comes from our inductive hypothesis (b).

We know that $\opsi \circ \ophi$ is an isomorphism, so
\[\HH_k(\Mod_{g,1}^b(H);V) = \Image(\ophi) \oplus \ker(\opsi).\]
Combining this with the bottom exact sequence in \eqref{eqn:gysinbreak}, we see that the map
\begin{equation}
\label{eqn:identifykeropsi}
\omu|_{\ker{\opsi}}\colon \ker(\opsi) \longrightarrow \HH_{k-2}(\Mod_{g,1}^b(H);V)
\end{equation}
is an isomorphism.
To prove that $\ker(\psi) \cap \Image(\phi) = 0$, it is enough to prove that the restriction of $\mu$ to $\ker(\psi)$ is
also an isomorphism.  To do that, since the right-hand vertical arrow in \eqref{eqn:gysinbreak} is an isomorphism
it is enough to prove that the map $\HH_k(\Mod_{g,1}^b(\ell);V) \rightarrow \HH_k(\Mod_{g,1}^b(H);V)$ restricts
to an isomorphism from $\ker(\psi)$ to $\ker(\opsi)$.

To do this, we must identify $\ker(\psi)$ and $\ker(\opsi)$.
Let $x_0$ be the puncture of $\Sigma_{g,1}^b$.  In light of Remark \ref{remark:birmanlevel}, Theorem
\ref{theorem:birmanpartial} gives a commutative diagram of Birman exact sequences
\begin{equation}
\label{eqn:birmandiagram}
\begin{tikzcd}[row sep=scriptsize]
1 \arrow{r} & \PP_{x_0}(\Sigma_g^b,\ell) \arrow{r} \arrow{d} & \Mod_{g,1}^b(\ell) \arrow{r} \arrow{d} & \Mod_g^b(\ell) \arrow{r} \arrow{d} & 1 \\
1 \arrow{r} & \PP_{x_0}(\Sigma_g^b,H)    \arrow{r}           & \Mod_{g,1}^b(H)    \arrow{r}           & \Mod_g^b(H)    \arrow{r}           & 1
\end{tikzcd}
\end{equation}
Since $b \geq 1$, the maps
$\Mod_{g,1}^{b}(\ell) \rightarrow \Mod_g^b(\ell)$ and $\Mod_{g,1}^{b}(H) \rightarrow \Mod_g^b(H)$
induced by filling in $x_0$ split via maps induced by an embedding $\Sigma_g^b \hookrightarrow \Sigma_{g,1}^{b}$ as follows:\\
\centerline{\psfig{file=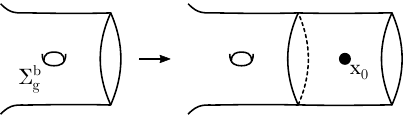,scale=1}}
It follows that all differentials coming out of the bottom rows of the Hochschild--Serre spectral sequences with
coefficients in $V$ associated to the rows of \eqref{eqn:birmandiagram} must vanish.  Since $\PP_{x_0}(\Sigma_g^b,\ell)$ 
and $\PP_{x_0}(\Sigma_g^b,H)$ are subgroups of the free group $\pi_1(\Sigma_g^b,x_0)$, 
these spectral sequence only have two potentially nonzero rows, so they degenerate at the $E^2$ page.  They thus break
up into a bunch of short exact sequences.

To identify these short exact sequences, we must identify the $E^2$-pages of our spectral sequences.  For
$\Mod_{g,1}^b(\ell)$, the two potentially nonzero rows of our spectral sequence are as follows:
\begin{itemize}
\item $E^2_{p0} = \HH_p(\Mod_{g}^b(\ell);\HH_0(\PP_{x_0}(\Sigma_g^b,\ell);V))$.  Recall that at the beginning
of this step we imposed the condition that
$H$ has genus at most $r+1$.  Our genus assumptions imply that
$g \geq (r+1) + r$, so Corollary \ref{corollary:pushtensorpowerschi} implies that
\begin{align*}
\HH_0(\PP_{x_0}(\Sigma_g^b,\ell);V) = \fH_{g,1}^b(\uchi)_{\PP_{x_0}(\Sigma_g^b,\ell)} \cong \fH_g^b(\uchi) = W.
\end{align*}
We conclude that $E^2_{p0} = \HH_p(\Mod_g^b(\ell);W)$.
\item $E^2_{p1} = \HH_p(\Mod_{g}^b(\ell);\HH_1(\PP_{x_0}(\Sigma_g^b,\ell);V))$.
\end{itemize}
The short exact sequence encoded by the $p+q=k$ diagonal of this spectral sequence is thus of the form
\[\begin{tikzcd}[column sep=0.7em, row sep=1em]
0 \arrow{r} & \HH_{k-1}(\Mod_{g}^b(\ell);\HH_1(\PP_{x_0}(\Sigma_g^b,\ell);V)) \arrow{r} & \HH_k(\Mod_{g,1}^b(\ell);V) \arrow{r}{\psi} & \HH_k(\Mod_g^b(\ell);W) \arrow{r} & 0.
\end{tikzcd}\]
A similar analysis holds for $\Mod_{g,1}^b(H)$, yielding an analogous short exact sequence.  There is a map
between these short exact sequences of the form  
\begin{equation}
\label{eqn:identifypsi}
\begin{tikzcd}[font=\footnotesize, column sep=0.7em, row sep=1em]
0 \arrow{r} & \HH_{k-1}(\Mod_{g}^b(\ell);\HH_1(\PP_{x_0}(\Sigma_g^b,\ell);V)) \arrow{r} \arrow{d} & \HH_k(\Mod_{g,1}^b(\ell);V) \arrow{r}{\psi}  \arrow{d} & \HH_k(\Mod_g^b(\ell);W) \arrow{r} \arrow{d} & 0 \\
0 \arrow{r} & \HH_{k-1}(\Mod_{g}^b(H);\HH_1(\PP_{x_0}(\Sigma_g^b,H);V))       \arrow{r}           & \HH_k(\Mod_{g,1}^b(H);V)    \arrow{r}{\opsi}           & \HH_k(\Mod_g^b(H);W)    \arrow{r}           & 0.
\end{tikzcd}
\end{equation}
It follows that to prove that the map $\ker(\psi) \rightarrow \ker(\opsi)$ is an isomorphism, we must prove that the map
\[\HH_{k-1}(\Mod_{g}^b(\ell);\HH_1(\PP_{x_0}(\Sigma_g^b,\ell);V)) \rightarrow \HH_{k-1}(\Mod_{g}^b(H);\HH_1(\PP_{x_0}(\Sigma_g^b,H);V))\]
is an isomorphism.\footnote{This will imply Theorem \ref{maintheorem:modprym}.  Indeed, consider the case where $V \cong \C$ is the
trivial representation of size $r=0$ and $H=0$.  We then have $\HH_1(\PP_{x_0}(\Sigma_g^b,\ell);V) \cong \fH_g^b(\ell;\C)$ and
$\HH_1(\PP_{x_0}(\Sigma_g^b,H);V) = \fH_g^b(\C)$.  The fact that this map
is an isomorphism thus becomes the fact that the map $\HH_{k-1}(\Mod_g^b(\ell);\fH_g^b(\ell;\C)) \rightarrow \HH_{k-1}(\Mod_g^b;\fH_g^b(\C))$
is an isomorphism.  The universal coefficients theorem says this is also true with the $\C$ replaced with a $\Q$, which is
exactly Theorem \ref{maintheorem:modprym} in the special case that the surface has no punctures.  The case where the surface has punctures
can be derived from this exactly like in Claim \ref{claim:reducepunctures}.}  This map factors as the composition of the maps
\begin{equation}
\label{eqn:toproveisomorphism1}
\HH_{k-1}(\Mod_{g}^b(\ell);\HH_1(\PP_{x_0}(\Sigma_g^b,\ell);V)) \rightarrow \HH_{k-1}(\Mod_{g}^b(\ell);\HH_1(\PP_{x_0}(\Sigma_g^b,H);V))
\end{equation}
and
\begin{equation}
\label{eqn:toproveisomorphism2}
\HH_{k-1}(\Mod_{g}^b(\ell);\HH_1(\PP_{x_0}(\Sigma_g^b,H);V)) \rightarrow \HH_{k-1}(\Mod_{g}^b(H);\HH_1(\PP_{x_0}(\Sigma_g^b,H);V)).
\end{equation}
Claim \ref{claim:31ppu} says that \eqref{eqn:toproveisomorphism2} is an isomorphism.  This claim
includes the assumption that $g$ is at least the genus of $H$ plus $r$, which follows from the condition
that the genus of $H$ is at most $r+1$ that we imposed at the beginning of this step.  

It therefore
remains to prove that \eqref{eqn:toproveisomorphism1} is an isomorphism.
At this point in the
proof, we re-impose the assumption that the genus of $H$ is at most $r$.  As we noted
at the beginning, it is enough to verify this step in that case; however, we will use
some of the above calculations for other symplectic subspaces of genus $r+1$.

\begin{claim}{4.2}
\label{claim:42iso}
The map \eqref{eqn:toproveisomorphism1} is an isomorphism.
\end{claim}
\begin{proof}[Proof of claim]
Let $\cD = \HH_1(\Sigma_g;\Z/\ell)$.  By Remark \ref{remark:birmanlevel}, we have a short
exact sequence
\begin{equation}
\label{eqn:ppseq}
1 \longrightarrow \PP_{x_0}(\Sigma_g^b,\ell) \longrightarrow \PP_{x_0}(\Sigma_g^b) \longrightarrow \cD \longrightarrow 1.
\end{equation}
Regard $H$ as a subspace of $\cD$ via the map $\HH_1(\Sigma_g^b;\Z/\ell) \rightarrow \HH_1(\Sigma_g;\Z/\ell)$, and
let $C \subset \cD$ be its orthogonal complement with respect to the algebraic intersection pairing on $\cD$.
By Theorem \ref{theorem:birmanpartial}, the group $\PP_{x_0}(\Sigma_g^b,H)$ is the kernel of the map
\[\PP_{x_0}(\Sigma_g^b) \longrightarrow \cD = H \oplus C \stackrel{\text{proj}}{\longrightarrow} H.\]
Combining this with \eqref{eqn:ppseq}, we get a short exact sequence
\[1 \longrightarrow \PP_{x_0}(\Sigma_g^b,\ell) \longrightarrow \PP_{x_0}(\Sigma_g^b,H) \longrightarrow C \longrightarrow 1.\]
The group $\PP_{x_0}(\Sigma_g^b,H)$ acts by conjugation on $\PP_{x_0}(\Sigma_g^b,\ell)$ and on $V$
since $V$ is compatible with $H$.  Since inner automorphisms act 
trivially on homology (see, e.g., \cite[Proposition III.8.1]{BrownCohomology}), this
induces an action of the finite abelian group $C$ on $\HH_1(PP_{x_0}(\Sigma_g^b,\ell);V)$.  
Using the transfer map lemma (Lemma \ref{lemma:transfer}),
we deduce that
\begin{equation}
\label{eqn:coinvariantsc}
\HH_1(\PP_{x_0}(\Sigma_g^b,\ell);V)_{C} = \HH_1(\PP_{x_0}(\Sigma_g^b,H);V).
\end{equation}
The group $\Mod_g^b(\ell)$ also acts on $\HH_1(\PP_{x_0}(\Sigma_g^b,\ell);V)$.  Since $\Mod_g^b(\ell)$
acts trivially on $\cD$, the actions of $\Mod_g^b(\ell)$ and $C$ on $\HH_1(\PP_{x_0}(\Sigma_g^b,\ell);V)$
commute.  It follows that $\Mod_g^b(\ell)$ preserves the decomposition of
$\HH_1(\PP_{x_0}(\Sigma_g^b,\ell);V)$ into $C$-isotypic components.  

Just
like we talked about for the Prym representation in \S \ref{section:decompositionprym},
the irreducible representations of the finite abelian group $C$ are in
bijection with characters $\chi \in \hC$.  For $\chi \in \hC$, let $U_{\chi}$ denote
the corresponding isotypic component of $\HH_1(\PP_{x_0}(\Sigma_g^b,\ell);V)$.  We thus
have
\begin{equation}
\label{eqn:decomposeisotypic}
\HH_{k-1}(\Mod_g^b(\ell);\HH_1(\PP_{x_0}(\Sigma_g^b,\ell);V)) = \bigoplus_{\chi \in \hC} \HH_{k-1}(\Mod_g^b(\ell);U_{\chi}).
\end{equation}
We now return to \eqref{eqn:coinvariantsc}.  Taking the $C$-coinvariants like in
\eqref{eqn:coinvariantsc} kills exactly the $U_{\chi}$ such that $\chi$ is nontrivial
(c.f.\ the proof of Lemma \ref{lemma:intermediatecover}).  Letting $1 \in \hC$ denote
the trivial character, we thus see from \eqref{eqn:coinvariantsc} and \eqref{eqn:decomposeisotypic} that
\[\HH_{k-1}(\Mod_g^b(\ell);\HH_1(\PP_{x_0}(\Sigma_g^b,H);V)) = \HH_{k-1}(\Mod_g^b(\ell);U_1).\]
The map
\[\HH_{k-1}(\Mod_{g}^b(\ell);\HH_1(\PP_{x_0}(\Sigma_g^b,\ell);V)) \rightarrow \HH_{k-1}(\Mod_{g}^b(\ell);\HH_1(\PP_{x_0}(\Sigma_g^b,H);V))\]
we are trying to prove is an isomorphism can therefore be identified with the projection
\[\bigoplus_{\chi \in \hC} \HH_{k-1}(\Mod_g^b(\ell);U_{\chi}) \rightarrow \HH_{k-1}(\Mod_g^b(\ell);U_1).\]
Fixing some nontrivial $\chi_0 \in \hC$, we deduce that to prove the claim, it
suffices to prove that $\HH_{k-1}(\Mod_g^b(\ell);U_{\chi_0}) = 0$.

Recall that the genus of $H \subset \HH_1(\Sigma_{g,1}^b;\Z/\ell)$ is at most $r$.  Let $H'$ be a symplectic subspace
of $\HH_1(\Sigma_{g,1}^b;\Z/\ell)$ of genus at most $r+1$ with the following two properties:
\begin{itemize}
\item $H \subset H'$.  Regarding $H'$ as a subgroup of $\cD$ and letting $C'$ be its orthogonal complement, we
thus have $C' \subset C$.
\item The character $\chi_0\colon C \rightarrow \C^{\ast}$ vanishes on $C'$.
\end{itemize}
Since the image of $\chi_0$ is a finite cyclic group, such an $H'$ can be constructed using
an argument similar to \cite[Lemma 3.5]{PutmanPartialTorelli}.  We have
$\PP_{x_0}(\Sigma_g^b,H') \subset \PP_{x_0}(\Sigma_g^b,H)$, and arguments
identical to the ones we gave above show that
\[\HH_1(\PP_{x_0}(\Sigma_g^b,\ell);V)_{C'} = \HH_1(\PP_{x_0}(\Sigma_g^b,H');V).\]
Continuing just like above, since $\hC'$ is the set of $\chi \in \hC$ with $\chi|_{C'} = 1$ we deduce that
\[\HH_{k-1}(\Mod_g^b(\ell);\HH_1(\PP_{x_0}(\Sigma_g^b,H');V)) = \bigoplus_{\substack{\chi \in \hC \\ \chi|_{C'} = 1}} \HH_{k-1}(\Mod_g^b(\ell);U_{\chi}).\]
The map
\begin{equation}
\label{eqn:projectionisomorphism}
\HH_{k-1}(\Mod_g^b(\ell);\HH_1(\PP_{x_0}(\Sigma_g^b,H');V)) \rightarrow \HH_{k-1}(\Mod_g^b(\ell);\HH_1(\PP_{x_0}(\Sigma_g^b,H);V))
\end{equation}
can thus be identified with the projection
\[\bigoplus_{\substack{\chi \in \hC \\ \chi|_{C'} = 1}} \HH_{k-1}(\Mod_g^b(\ell);U_{\chi}) \rightarrow \HH_{k-1}(\Mod_g^b(\ell);U_1).\]
The summand $\HH_{k-1}(\Mod_g^b(\ell);U_{\chi_0})$ we are trying to prove is $0$ therefore appears
in the kernel of \eqref{eqn:projectionisomorphism}.

To prove the claim, we are therefore reduced to proving that \eqref{eqn:projectionisomorphism} is injective.  In fact, we will show that it is an isomorphism.
Our earlier work allows us to simplify its domain and codomain.  We start with the codomain.  By \eqref{eqn:identifykeropsi}, we have an isomorphism
\[\HH_{k-1}(\Mod_g^b(\ell);\HH_1(\PP_{x_0}(\Sigma_g^b,H);V)) \stackrel{\cong}{\rightarrow} \HH_{k-1}(\Mod_{g}^b(H);\HH_1(\PP_{x_0}(\Sigma_g^b,H);V)).\]
By \eqref{eqn:identifypsi}, the right hand side fits into a short exact sequence
\[\begin{tikzcd}[font=\small, column sep=0.7em, row sep=1em]
0 \arrow{r} & \HH_{k-1}(\Mod_{g}^b(H);\HH_1(\PP_{x_0}(\Sigma_g^b,H);V))       \arrow{r}           & \HH_k(\Mod_{g,1}^b(H);V)    \arrow{r}{\opsi}           & \HH_k(\Mod_g^b(H);W)    \arrow{r}           & 0.
\end{tikzcd}\]
Finally, \eqref{eqn:identifykeropsi} gives an isomorphism
\[\omu|_{\ker{\opsi}}\colon \ker(\opsi) \stackrel{\cong}{\longrightarrow} \HH_{k-2}(\Mod_{g,1}^b(H);V).\]
Combining the previous three facts, we get an isomorphism
\begin{equation}
\label{eqn:codomain}
\HH_{k-1}(\Mod_g^b(\ell);\HH_1(\PP_{x_0}(\Sigma_g^b,H);V)) \cong \HH_{k-2}(\Mod_{g,1}^b(H);V).
\end{equation}
This identifies the codomain of \eqref{eqn:projectionisomorphism}.

As for the domain, we can also run the above argument with $H'$ instead of $H$ since we were careful at the beginning of this step to allow
the genus to be at most $r+1$ instead of just $r$, and we see that
\begin{equation}
\label{eqn:domain}
\HH_{k-1}(\Mod_g^b(\ell);\HH_1(\PP_{x_0}(\Sigma_g^b,H');V)) \cong \HH_{k-2}(\Mod_{g,1}^b(H');V).
\end{equation}
Using the isomorphisms \eqref{eqn:codomain} and \eqref{eqn:domain}, we can identify identify the map \eqref{eqn:projectionisomorphism} we are trying
to prove is an isomorphism with the map
\[\HH_{k-2}(\Mod_{g,1}^b(H');V) \rightarrow \HH_{k-2}(\Mod_{g,1}^b(H);V).\]
But this is an isomorphism; indeed, using our inductive hypothesis (b)
we see that both the first map and the composition in
\[\HH_{k-2}(\Mod_{g,1}^b(\ell);V) \rightarrow \HH_{k-2}(\Mod_{g,1}^b(H');V) \rightarrow \HH_{k-2}(\Mod_{g,1}^b(H);V)\]
are isomorphisms.  This completes the proof of the claim.
\end{proof}

\begin{step}{5}[Completing the induction]
\label{step:main}
For some $k$ and $r$, make the inductive hypotheses (a)-(c) from Step \ref{step:induction}:
\begin{itemize}
\item[(a)] $r \geq 0$ and $k \geq 1$.
\item[(b)] We have already proved the theorem for $\HH_{i}$ for all $i < k$.
\item[(c)] For $\HH_k$, we have already proved the theorem for all homological representations of size less than $r$.  This
is vacuous if $r=0$.
\end{itemize}
Let $g \geq 0$ and $b \geq 1$.  Let $\fH_{g}^{b}(\uchi)$ be a size-$r$ homological representation
of $\Mod_{g}^{b}(\ell)$ and let $H$ be a symplectic subgroup of $\HH_1(\Sigma_{g}^{b}(\ell);\Z/\ell)$
that is compatible with $\fH_{g}^{b}(\uchi)$ and has genus at most $r$.  
Assume that $g \geq 2(k+r)^2+7k+6r+2$.  Then we prove that the map
\[\HH_k\left(\Mod_{g}^b\left(\ell\right);\fH_{g}^b(\uchi)\right) \rightarrow \HH_k\left(\Mod_{g}^b\left(H\right);\fH_{g}^b(\uchi)\right)\]
is an isomorphism.  In light of Step \ref{step:reducegenuspunctures}, this proves the
theorem for $\HH_k$ and homological representations of size $r$ in general, completing our induction.
\end{step}

By Lemma \ref{lemma:intermediatecover}, the $\Mod_g^b(H)$-representation $\fH_{g}^b(\uchi)$ is a direct summand
of $\fH_g^b(H;\C)^{\otimes r}$.  It follows that it is enough to prove that the map
\begin{equation}
\label{eqn:maintoprove}
\HH_k\left(\Mod_{g}^b\left(\ell\right);\fH_{g}^b(H;\C)^{\otimes r}\right) \rightarrow \HH_k\left(\Mod_{g}^b\left(H\right);\fH_{g}^b(H;\C)^{\otimes r}\right)
\end{equation}
is an isomorphism.  Step \ref{step:capboundary} implies that
\[\HH_k\left(\Mod_{g}^{b}\left(\ell\right);\fH_{g}^{b}(H;\C)\right) \cong \HH_k\left(\Mod_{g}^{1}\left(\ell\right);\fH_{g}^{1}(H;\C)\right).\]
Also, as we noted at the beginning of the proof of Step \ref{step:capboundary} our genus assumptions allow
us to use Theorem \ref{theorem:stability} to deduce that
\[\HH_k\left(\Mod_{g}^{b}\left(H\right);\fH_{g}^{b}(H;\C)\right) \cong \HH_k\left(\Mod_{g}^{1}\left(H\right);\fH_{g}^{1}(H;\C)\right).\]
We conclude that it is enough to prove that \eqref{eqn:maintoprove} is an isomorphism when $b=1$.

The transfer map lemma (Lemma \ref{lemma:transfer})
implies that
\[\HH_k\left(\Mod_{g}^1\left(H\right);\fH_{g}^1(H;\C)^{\otimes r}\right) = \left(\HH_k\left(\Mod_{g}^1\left(\ell\right);\fH_{g}^1(H;\C)^{\otimes r}\right)\right)_{\Mod_g^1(H)},\]
This reduces us to showing that
$\Mod_g^1(H)$ acts trivially on $\HH_k(\Mod_g^1(\ell);\fH_g^1(H;\C)^{\otimes r})$.  Lemma \ref{lemma:partialgen} says
that $\Mod_g^1(H)$ is generated by $\Mod_{g}^1(\ell)$ along with the set of all Dehn twists
$T_{\gamma}$ such that $[\gamma] \in H^{\perp}$.  Since inner automorphisms act
trivially on homology (see, e.g., \cite[Proposition III.8.1]{BrownCohomology}), this reduces us
to showing that such $T_{\gamma}$ act trivially.

Fix such a $\gamma$.  Say that an embedding $\Sigma_{g-1}^1 \hookrightarrow \Sigma_g^1$ is $H$-compatible if $H$ is contained
in the image of the induced map $\HH_1(\Sigma_{g-1}^1;\Z/\ell) \rightarrow \HH_1(\Sigma_g^1;\Z/\ell)$.  Fix
an $H$-compatible embedding $j\colon \Sigma_{g-1}^1 \hookrightarrow \Sigma_g^1$ such that
$\gamma$ is contained in the complement of the image of $j$:\\
\centerline{\psfig{file=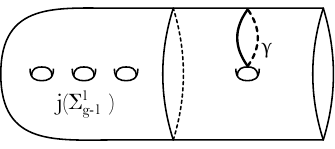,scale=1}}
Since $T_{\gamma}$ commutes with mapping classes supported on $\Sigma_{g-1}^1$,
it acts trivially on the image
of $\HH_k(\Mod_{g-1}^1(\ell);\fH_{g-1}^1(H;\C)^{\otimes r})$ in
$\HH_k(\Mod_g^1(\ell);\fH_g^1(H;\C)^{\otimes r}$.  It follows that it is enough to prove that the map
\[j_{\ast}\colon \HH_k(\Mod_{g-1}^1(\ell);\fH_{g-1}^1(H;\C)^{\otimes r}) \longrightarrow \HH_k(\Mod_{g}^1(\ell);\fH_{g}^1(H;\C)^{\otimes r})\]
is surjective.  For this, it is enough to prove the following two facts:
\begin{itemize}
\item The map
\[\bigoplus_{\substack{\Sigma_{g-1}^1 \hookrightarrow \Sigma_g^1 \\ \text{$H$-compatible}}} \HH_k(\Mod_{g-1}^1(\ell);\fH_{g-1}^1(H;\C)^{\otimes r}) \longrightarrow \HH_k(\Mod_{g}^1(\ell);\fH_{g}^1(H;\C)^{\otimes r})\]
is surjective.
\item Let $j_0\colon \Sigma_{g-1}^1 \hookrightarrow \Sigma_g^1$ and $j_1\colon \Sigma_{g-1}^1 \hookrightarrow \Sigma_g^1$ be two
$H$-compatible embeddings.  Let 
\[(j_i)_{\ast}\colon \HH_k(\Mod_{g-1}^1(\ell);\fH_{g-1}^1(H)^{\otimes r}) \longrightarrow \HH_k(\Mod_{g}^1(\ell);\fH_{g}^1(H)^{\otimes r})\]
be the induced map.  Then the images of $(j_0)_{\ast}$ and $(j_1)_{\ast}$ are the same.
\end{itemize}
These two facts are the subject of the following two claims:

\begin{claim}{5.1}
The map
\[\bigoplus_{\substack{\Sigma_{g-1}^1 \hookrightarrow \Sigma_g^1 \\ \text{$H$-compatible}}} \HH_k(\Mod_{g-1}^1(\ell);\fH_{g-1}^1(H;\C)^{\otimes r}) \longrightarrow \HH_k(\Mod_{g}^1(\ell);\fH_{g}^1(H;\C)^{\otimes r})\]
is surjective.
\end{claim}
\begin{proof}[Proof of claim]
Let $I$ be an open interval in $\partial \Sigma_g^1$, and consider the complex $\bbTT_g^1(I,H)$ of
$I$-tethered $H$-orthogonal tori in $\Sigma_g^1$ that we introduced in \S \ref{section:tetheredtoruscomplex}.  Given
a $p$-simplex $\sigma = [\iota_0,\ldots,\iota_p]$, let $X_{\sigma}$ be the subsurface of $\Sigma_g^1$ defined in
\S \ref{section:coefficientsystem}:
\[X_{\sigma} = \Sigma_g^1 \setminus \text{Nbhd}\left(\partial \Sigma_g^1 \cup \Image\left(\iota_0\right) \cup \cdots \cup \Image\left(\iota_p\right)\right).\]
See here:\\
\centerline{\psfig{file=CoefficientSystem,scale=1}}
Thus $X_{\sigma} \cong \Sigma_{g-p-1}^1$, and the inclusion $X_{\sigma} \hookrightarrow \Sigma_g^1$ is
$H$-compatible.  The $\Mod_g^1(\ell)$-stabilizer of $\sigma$ consists of all elements of 
$\Mod_g^1(\ell)$ supported on $X_{\sigma}$, so $\Mod_g^1(\ell)_{\sigma} \cong \Mod_{g-p-1}^1(\ell)$.

We now recall the definition of the augmented coefficient system $\cH_g^1(H;\C)$ on
$\bbT_g^1(I,H)$ we defined in \S \ref{section:coefficientsystem}.
Let $\pi\colon S_H \rightarrow \Sigma_g^1$ be the cover used to define $\fH_g^1(H;\C)$, and for a simplex
$\sigma$ of $\bbT_g^1(I,H)$ let $\tX_{\sigma} = \pi^{-1}(X_{\sigma})$.  For a $p$-simplex $\sigma$ of $\bbT_g^1(I,H)$, we then have
\[\cH_g^1(H;\C)(\sigma) = \HH_1(\tX_{\sigma};\C) \cong \fH_{g-p-1}^1(H;\C).\]
In particular, for the empty simplex $\emptysimp$ we have
\[\cH_g^1(H;\C)\emptysimp = \fH_g^1(H;\C).\]
It follows that to prove the claim, it is enough to prove that the map
\[\bigoplus_{v \in \bbT_g^1(I,H)^{(0)}} \HH_k(\Mod_g^1(\ell)_v;\cH_g^1(H;\C)^{\otimes r}(v)) \longrightarrow \HH_k(\Mod_g^1(\ell);\cH_g^1(H;\C)^{\otimes r}\emptysimp)\]
is surjective.  This will follow from Proposition \ref{proposition:stabilitymachine} once we verify its three hypotheses.
This requires manipulating our bound on $g$, so we introduce the notation
\[\bb(k,r) = 2(k+r)^2+7k+6r+2.\]
Thus our assumption is that $g \geq \bb(k,r)$.

Hypothesis (i) is that $\RH_i(\bbTT_g^1(I,H);\cH_g^1(H;\C)^{\otimes r}) = 0$ for $-1 \leq i \leq k-1$.
Lemma \ref{lemma:hstrongpoly} says that $\cH_g^1(H;\C)$ is strongly polynomial of degree $1$, so
Lemma \ref{lemma:tensorpolynomial} implies that $\cH_g^1(H;\C)^{\otimes r}$ is strongly polynomial of
degree $r$.  We have assumed that the genus of $H$ is at most $r$, so Corollary \ref{corollary:tetheredtoricm}
implies that 
$\bbTT_g^1(I,H)$ is weakly forward Cohen--Macaulay of dimension $\frac{g-(4r+3)}{2r+2}+1$.
Applying Theorem \ref{theorem:vanishing}, we deduce that 
$\RH_i(\bbTT_g^1(I,H);\cH_g^1(H;\C)^{\otimes r}) = 0$ for $-1 \leq i \leq \frac{g-(4r+3)}{2r+2}-r$.
We must prove that this is at least $k-1$.  Manipulating
\[\frac{g-(4r+3)}{2r+2}-r \geq k-1,\]
we see that it is equivalent to
\[g \geq (2r+2)(k+r-1) + (4r+3) = 2(r+1)(k+r)+2r+1.\]
We thus must prove that $\bb(k,r) \geq 2(r+1)(k+r)+2r+1$.  But at the beginning of Step \ref{step:capboundary}
we proved\footnote{Actually, what we proved was $\bb(k,r)-1 \geq 2(k+r)(r+2)+(4r+6)$.} that $\bb(k,r) \geq 2(k+r)(r+2)+(4r+7)$, which is even stronger.

Hypotheses (ii) is that $\RH_i(\bbTT_g^1(I,H)/\Mod_g^1(\ell)) = 0$ for $-1 \leq i \leq k$.
We assumed that the genus of $H$ is at most $r$, so
Corollary \ref{corollary:highconnectivityquotient} implies that $\bbTT_g^1(I,H)/\Mod_g^1(\ell)$
is at least $\frac{g-r-5}{2}$-connected.  We want to prove that this is at least $k$.  Manipulating this,
we see that it is equivalent to
\[g \geq 2k+r+5.\]
We thus must prove that $\bb(k,r) \geq 2k+r+5$.  For this, we calculate: 
\begin{align*}
\bb(k,r) &= 2(k+r)^2+7k+6r +2 = 2k + r + (2(k+r)^2+5k+5r+2) \\
         &\geq 2k+r+(2+5+0+2) \geq 2k+r + 5,
\end{align*}
as desired.  Here we are using our inductive hypothesis (a), which says
that $k \geq 1$ and $r \geq 0$.

Hypothesis (iii) is that if $\sigma$ is a simplex of $\bbTT_g^1(I,H)$ and $i \geq 1$, then the map
\[\HH_{k-i}(\Mod_g^1(\ell)_{\sigma};\cH_g^1(H;\C)^{\otimes r}(\sigma)) \longrightarrow \HH_{k-i}(\Mod_g^1(\ell);\cH_g^1(H;\C)^{\otimes r}\emptysimp)\]
is an isomorphism if $i-1 \leq \dim(\sigma) \leq i+1$.
By our description of the simplex
stabilizers and the values of $\cH_g^1(H;\C)$, this is equivalent to proving that 
for $i \geq 1$, the map
\[\HH_{k-i}(\Mod_{g-h}^1(\ell);\fH_{g-h}^1(H;\C)^{\otimes r}) \longrightarrow \HH_{k-i}(\Mod_g^1(\ell);\fH_g^1(H;\C)^{\otimes r})\]
is an isomorphism if $i \leq h \leq i+2$.  In fact, we will prove
that it is an isomorphism for $1 \leq h \leq i+2$.  The above map fits into a commutative diagram
\[\begin{tikzcd}[row sep=scriptsize]
\HH_{k-i}(\Mod_{g-h}^1(\ell);\fH_{g-h}^1(H;\C)^{\otimes r}) \arrow{r} \arrow{d} & \HH_{k-i}(\Mod_g^1(\ell);\fH_g^1(H;\C)^{\otimes r}) \arrow{d} \\
\HH_{k-i}(\Mod_{g-h}^1(H);\fH_{g-h}^1(H;\C)^{\otimes r})    \arrow{r}           & \HH_{k-i}(\Mod_g^1(H);\fH_g^1(H;\C)^{\otimes r}).
\end{tikzcd}\]
What we will do is use our inductive hypothesis (b) to show that both vertical arrows are isomorphisms
and Theorem \ref{theorem:stability} to prove that the bottom horizontal arrow is an isomorphism.

We start by using our inductive hypothesis to show that both vertical arrows are isomorphisms.  To
show this, since $g \geq \bb(k,r)$ it is enough to prove that
\begin{equation}
\label{eqn:verticaltoprove}
\bb(k,r) \geq \bb(k-i,r)+(i+2) \quad \text{for $1 \leq i \leq k$}.
\end{equation}
In fact, we will prove something stronger, namely that for all $j, r \geq 0$ we have
have
\begin{equation}
\label{eqn:verticaltoprove2}
\bb(j+1,r) \geq \bb(j,r) + 7.
\end{equation}
Iterating this gives an even better bound than \eqref{eqn:verticaltoprove}, namely
that $\bb(k,r) \geq \bb(k-i,r)+7i$.  To see \eqref{eqn:verticaltoprove2}, we
calculate as follows:
\[\bb(j+1,r) = 2(j+r+1)^2 + 7(j+1)+6r+2 \geq 2(j+r)^2 + 7(j+1) + 6r + 2= \bb(j,r)+7.\]
We next use Theorem \ref{theorem:stability} to prove that the bottom horizontal arrow is
an isomorphism.  The bound in that theorem for $\HH_j$ is $\bb'(j,r) = 2(j+r)(r+1)+(4r+2)$,
so since $g \geq \bb(k,r)$ what we have to show is that
\[\bb(k,r) \geq \bb'(k-i,r)+(i+2) \quad \text{for $1 \leq i \leq k$}.\]
To see this, it is enough to prove that $\bb(k,r) \geq \bb'(k,r)+7$ and that $\bb'(j+1,r) \geq \bb'(j,r)+2$
for all $j,r \geq 0$.
For $\bb(k,r) \geq \bb'(k,r)+7$, we calculate as follows:
\begin{align*}
\bb(k,r) &= 2(k+r)^2+7k+6r+2 \geq 2(k+r)(r+1) + 7k+6r+2 \\
         &\geq 2(k+r)(r+1) + 4r + (7k+2) \\
         &\geq 2(k+r)(r+1)+4r+9 = \bb'(k,r) + 7.
\end{align*}
For $\bb'(j+1,r) \geq \bb'(j,r)+2$, we calculate as follows:
\[\bb'(j+1,r) = 2(j+r+1)(r+1)+(4r+2) = 2(j+r)(r+1) + (4r+2) + (2r+2) \geq \bb'(j,r)+2.\qedhere\]
\end{proof}

\begin{claim}{5.2}
Let $j_0\colon \Sigma_{g-1}^1 \hookrightarrow \Sigma_g^1$ and $j_1\colon \Sigma_{g-1}^1 \hookrightarrow \Sigma_g^1$ be two
$H$-compatible embeddings.  Let 
\[(j_i)_{\ast}\colon \HH_k(\Mod_{g-1}^1(\ell);\fH_{g-1}^1(H)^{\otimes r}) \longrightarrow \HH_k(\Mod_{g}^1(\ell);\fH_{g}^1(H)^{\otimes r})\]
be the induced map.  Then the images of $(j_0)_{\ast}$ and $(j_1)_{\ast}$ are the same.
\end{claim}
\begin{proof}[Proof of claim]
The group $\Mod_g^1(H)$ acts transitively\footnote{This can be proved directly 
along the same lines as Lemma \ref{lemma:identifyquotient}.  Alternatively, since
all $H$-compatible embeddings come from vertices of $\bbTT_g^1(I,H)$, it follows
from the fact that $\Mod_g^1(H)$ acts transitively on such vertices.  See \cite[Lemma 3.9]{PutmanPartialTorelli}
for a proof of this in a much more general context.}
 on the set of $H$-compatible embeddings $\Sigma_{g-1}^1 \hookrightarrow \Sigma_g^1$.  We can thus
find some $\phi \in \Mod_g^1(H)$ such that $j_1 = \phi \circ j_0$.  

Let $h$ be the genus of $H$, so by our assumptions $h \leq r$.
Lemma \ref{lemma:partialgen}
says that $\Mod_g^1(H)$ is generated by $\Mod_g^1(\ell)$ along with any set $S$ of Dehn twists
about simple closed nonseparating curves $\gamma$ with $[\gamma] \in H^{\perp}$ such that $S$
maps to a generating set for $\Sp(H^{\perp}) \cong \Sp_{2(g-h)}(\Z/\ell)$.  In fact, as
Remark \ref{remark:specificgen} points out, we
can take
\[S = \{T_{\alpha_1},\ldots,T_{\alpha_{g-h}},T_{\beta_1},\ldots,T_{\beta_{g-h}},T_{\gamma_1},\ldots,T_{\gamma_{g-h-1}}\},\]
where the $\alpha_i$ and $\beta_i$ and $\gamma_i$ are as follows:\\
\centerline{\psfig{file=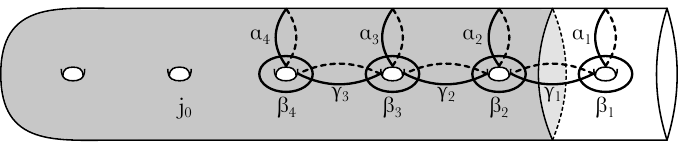,scale=1}}
The image of $j_0$ is the shaded region and $H$ consists of all elements of homology orthogonal to the
curves about whose twists are in $S$, so $H$ is supported on the handles on the left side of the figure
that have no $S$-curves around them.  

The element $\phi \in \Mod_g^1(H)$ with $j_1 = \phi \circ j_0$ from the first paragraph
can thus be written as $\phi = \phi_1 \cdots \phi_n$ with each $\phi_i$ either in $\Mod_g^1(\ell)$
or $S^{\pm 1}$.  We can therefore ``connect'' $j_0$ and $j_1$ by the sequence of $H$-compatible embeddings
\[j_0, \quad \phi_1 \circ j_0, \quad \phi_1 \phi_2 \circ j_0, \quad \ldots, \quad \phi_1 \cdots \phi_n \circ j_0 = j_1.\]
It is enough to prove that the maps on homology induced by consecutive embeddings $\phi_1 \cdots \phi_i \circ j_0$ and
$\phi_1 \cdots \phi_i \phi_{i+1} \circ j_0$ in this sequence have the same image.  Multiplying these
on the left by $(\phi_1 \cdots \phi_i)^{-1}$, we see that in fact it is enough to prove that the maps
on homology induced by $j_0$ and $\phi_i \circ j_0$ have the same image.  We remark that this
type of argument was systematized in \cite{PutmanConnectivityNote}, which has many examples of it. 

This is trivial if $\phi_i \in \Mod_g^1(\ell)$
since the images differ by an inner automorphism of $\Mod_g^1(\ell)$ and inner automorphisms act
trivially on homology (see, e.g., \cite[Proposition III.8.1]{BrownCohomology}).  It is also trivial
if $\phi_i$ is an element of $S^{\pm 1}$ that fixes the subsurface $j_0(\Sigma_{g-1}^1)$.  The
remaining case is where $\phi_i = T_{\gamma_{1}}^{\pm 1}$.  It is enough to deal
with the case where the sign is positive; indeed, if the maps on homology induced by $j_0$ and $T_{\gamma_{1}} \circ j_0$
have the same image, then we can multiply both by $T_{\gamma_{1}}^{-1}$ and deduce that the same 
is true for $T_{\gamma_{1}}^{-1} \circ j_0$ and $j_0$.

In summary, we have reduced ourselves to handling the case where $j_1 = T_{\gamma_{1}} \circ j_0$ as in
the following:\\
\centerline{\psfig{file=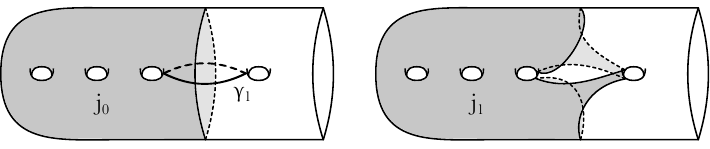,scale=1}}
We can find an embedding $\iota\colon \Sigma_{g-1}^2 \hookrightarrow \Sigma_g^1$ whose image
contains the images of $j_0$ and $j_1$:\\
\centerline{\psfig{file=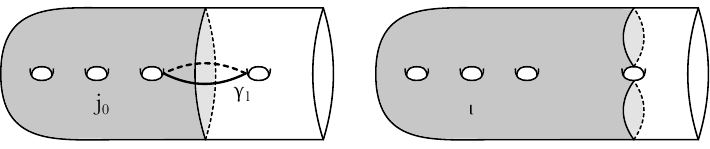,scale=1}}
Using the notation from \S \ref{section:subsurfacestabilizers}, the embedding $\iota$ induces
a homomorphism $\hMod_{g-1}^2(\ell) \rightarrow \Mod_g(\ell)$, and each $j_i$ factors as
\[\Mod_{g-1}^1(\ell) \stackrel{j'_i}{\longrightarrow} \hMod_{g-1}^2(\ell) \longrightarrow \Mod_g(\ell).\]
To prove that the images of $(j_0)_{\ast}$ and $(j_1)_{\ast}$ are the same, it is enough to prove
that the maps
\[(j'_i)_{\ast}\colon \HH_k(\Mod_{g-1}^1(\ell);\fH_{g-1}^1(H)^{\otimes r}) \longrightarrow \HH_k(\hMod_{g-1}^2(\ell);\fH_{g-1}^2(H)^{\otimes r})\]
are surjective.  In fact, we will prove they are isomorphisms.  Consider the composition
\begin{equation}
\label{eqn:composeidentity}
\Mod_{g-1}^1(\ell) \stackrel{j'_i}{\longrightarrow} \hMod_{g-1}^2(\ell) \stackrel{f}{\hookrightarrow} \Mod_{g-1}^2(\ell) \stackrel{f'}{\longrightarrow} \Mod_{g-1}^1(\ell),
\end{equation}
where the final map glues a disc to one of the components of $\partial \Sigma_{g-1}^1$.  This comes from a map 
$\Sigma_{g-1}^1 \rightarrow \Sigma_{g-1}^1$ that is homotopic to the identity:\\
\centerline{\psfig{file=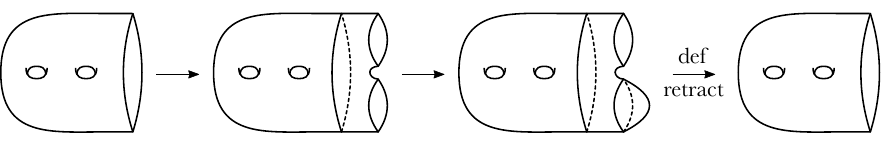,scale=1}}
Letting $W = \fH_{g-1}^1(H)^{\otimes r}$ and $V = \fH_{g-1}^2(H)^{\otimes r}$, it follows that the following composition
is the identity (and in particular, is an isomorphism):
\[\HH_k(\Mod_{g-1}^1(\ell);W) \stackrel{(j'_i)_{\ast}}{\rightarrow} \HH_k(\hMod_{g-1}^2(\ell);V) \stackrel{f_{\ast}}{\rightarrow} \HH_k(\Mod_{g-1}^2(\ell);V) \stackrel{(f')_{\ast}}{\rightarrow} \HH_k(\Mod_{g-1}^1(\ell);W).\]
Here for the final map we are using the map $V \rightarrow W$ induced by the map that fills in a boundary component.
Corollary \ref{corollary:hatcap} says that $f_{\ast}$ is an isomorphism, and Step \ref{step:capboundary} says that $(f')_{\ast}$ is an 
isomorphism.\footnote{This is why we needed the bound in Step \ref{step:capboundary} to be one better than the bound
we are proving for $\HH_k$.}  We
conclude that $(j'_i)_{\ast}$ is an isomorphism, as desired.
\end{proof}

This completes the proof of Theorem \ref{maintheorem:generalcase}.
\end{proof}

\section{Closed surfaces}
\label{section:closed}

We close by showing how to derive Theorems \ref{maintheorem:mod} and \ref{maintheorem:modstd} for closed surfaces.

\subsection{Alternate standard representation}
This first requires the following variant on Theorem \ref{maintheorem:modstd} for non-closed surfaces.  Let
$\Mod_{g,p}^b$ act on $\HH_1(\Sigma_g)$ via the homomorphism $\Mod_{g,p}^b \rightarrow \Mod_g$ that fills
in the punctures and glues discs to the boundary components.  Also, recall that
$\Mod_{g,p}^b[\ell]$ is the kernel of the action of $\Mod_{g,p}^b$ on
$\HH_1(\Sigma_g;\Z/\ell)$.

\begin{maintheorem}
\label{maintheorem:modstd2}
Let $g,p,b \geq 0$ and $\ell \geq 2$ be such that $p+b \geq 1$.  Then for $r \geq 0$, the maps
\begin{equation}
\label{eqn:modstd2toprove}
\HH_k\left(\Mod_{g,p}^b\left(\ell\right);\HH_1\left(\Sigma_{g};\Q\right)^{\otimes r}\right) \rightarrow \HH_k\left(\Mod_{g,p}^b;\HH_1\left(\Sigma_{g};\Q\right)^{\otimes r}\right)
\end{equation}
and
\[\HH_k\left(\Mod_{g,p}^b\left[\ell\right];\HH_1\left(\Sigma_{g};\Q\right)^{\otimes r}\right) \rightarrow \HH_k\left(\Mod_{g,p}^b;\HH_1\left(\Sigma_{g};\Q\right)^{\otimes r}\right)\]
are isomorphisms if $g \geq 2(k+r)^2+7k+6r+2$.
\end{maintheorem}
\begin{proof}
The transfer map lemma (Lemma \ref{lemma:transfer}) implies that both maps in the composition
\[\begin{tikzcd}[font=\small, column sep=1em]
\HH_k\left(\Mod_{g,p}^b\left(\ell\right);\HH_1\left(\Sigma_{g};\Q\right)^{\otimes r}\right) \arrow{r} &
\HH_k\left(\Mod_{g,p}^b\left[\ell\right];\HH_1\left(\Sigma_{g};\Q\right)^{\otimes r}\right) \arrow{r} &
\HH_k\left(\Mod_{g,p}^b;\HH_1\left(\Sigma_{g};\Q\right)^{\otimes r}\right)
\end{tikzcd}\]
are surjections.  It is thus enough to prove that \eqref{eqn:modstd2toprove} is an isomorphism.
If $p+b = 1$, then $\HH_1(\Sigma_{g,p}^b;\Q) \cong \HH_1(\Sigma_g;\Q)$ and this reduces to Theorem \ref{maintheorem:modstd}.  We can thus
assume that $p+b \geq 2$.

Choose some arbitrary ordering on the punctures and boundary components of $\Sigma_{g,p}^b$, and for $0 \leq p' \leq p$ and $0 \leq b' \leq b$ let
$\Mod_{g,p}^b$ act on $\HH_1(\Sigma_{g,p'}^{b'};\Q)$ by filling in the first $p-p'$ punctures and gluing discs to the first $b-b'$ boundary components.
For sequences $\up = (p_1,\ldots,p_r)$ with $0 \leq p_i \leq p$ and $\ub = (b_1,\ldots,b_r)$ with $0 \leq b_i \leq b$, define
\[U(\up,\ub) = \HH_1(\Sigma_{g,p_1}^{b_1};\Q) \otimes \cdots \otimes \HH_1(\Sigma_{g,p_r}^{b_r};\Q).\]
We will prove more generally that the map
\[\HH_k\left(\Mod_{g,p}^b\left(\ell\right);U(\up,\ub)\right) \rightarrow \HH_k\left(\Mod_{g,p}^{b};U(\up,\ub)\right)\]
is an isomorphism if $g \geq 2(k+r)^2+7k+6r+2$.

If $r=0$, then our representation is the trivial representation and the theorem
reduces to Theorem \ref{maintheorem:modstd}.  Assume, therefore, that $r>0$.  For $\up$ and $\ub$ as above, define
\[d(\up,\ub) = \sum_{i=1}^r (p-p_i) + \sum_{i=1}^r (b-b_i) \geq 0.\]
The proof will be by induction on $d(\up,\ub)$.  The base case is when $d(\up,\ub) = 0$, in which case
$U(\up,\ub) = \HH_1(\Sigma_{g,p}^b;\Q)^{\otimes r}$ and the theorem follows from Theorem \ref{maintheorem:modstd}.
Assume, therefore, that $d(\up,\ub) > 0$ and that the theorem is true whenever this is smaller.

If for some $i$ we have $p_i = b_i = 0$, then increasing either $p_i$ or $b_i$ by $1$ does not change
$U(\up,\ub)$, so the theorem follows by induction.  We can therefore assume that for all $i$ we have
either $p_i >0$ or $b_i>0$.  Since $d(\up,\ub)>0$, there is some $i$ such that either $p_i<p$ or $b_i<b$ (or both).
We will give the details for when $b_i<b$.  The case where $p_i<p$ is similar.

Reordering the indices, we can assume that $b_r<b$.  Let
\[\ub' = (b_1,\ldots,b_{r-1}) \quad \text{and} \quad \ub'' = (b_1,\ldots,b_{r-1},b_r+1).\]
Since we do not have $b_r = p_r = 0$, we have a short exact sequence
\[0 \longrightarrow \Q \longrightarrow \HH_1(\Sigma_{g,p_r}^{b_r+1};\Q) \longrightarrow \HH_1(\Sigma_{g,p_r}^{b_r};\Q) \longrightarrow 0\]
of representations of $\Mod_{g,p}^b$.  Tensoring this with $U(\up,\ub')$, we get a short exact sequence
\[0 \longrightarrow U(\up,\ub') \longrightarrow U(\up,\ub'') \longrightarrow U(\up,\ub) \longrightarrow 0\]
of $\Mod_{g,p}^b$-representations.  This induces long exact sequences in both $\Mod_{g,p}^b(\ell)$ and
$\Mod_{g,p}^b(H)$ homology, and a map between these long exact sequences.  As notation,
let
\[U = U(\up,\ub) \quad \text{and} \quad U' = U(\up,\ub') \quad \text{and} \quad U'' = U(\up,\ub'')\]
and let $M(\ell) = \Mod_{g,p}^b(\ell)$ and $M(H) = \Mod_{g,p}^b(H)$.  This map between long exact
sequences contains the segment
\[\begin{tikzcd}[font=\footnotesize, column sep=small, row sep=scriptsize]
\HH_{k}(M(\ell);U') \arrow{r} \arrow{d}{f_1} & \HH_{k}(M(\ell);U'') \arrow{r} \arrow{d}{f_2} & \HH_{k}(M(\ell);U) \arrow{r} \arrow{d}{f_3} & \HH_{k-1}(M(\ell);U') \arrow{r} \arrow{d}{f_4} & \HH_{k-2}(M(\ell);U'') \arrow{d}{f_5} \\
\HH_{k}(M(H);U')    \arrow{r}                & \HH_{k}(M(H);U'')    \arrow{r}                & \HH_{k}(M(H);U)    \arrow{r}                & \HH_{k-1}(M(H);U')    \arrow{r}                & \HH_{k-2}(M(H);U'')
\end{tikzcd}\]
If $g \geq 2(k+r)^2+7k+6r+2$, then our inductive hypothesis implies that $f_1$ and $f_2$ and $f_4$ and $f_5$ are isomorphisms, so by the five-lemma $f_3$ is an isomorphism, as desired.
\end{proof}

\subsection{Closed surfaces}
The following theorem subsumes Theorems \ref{maintheorem:mod} and \ref{maintheorem:modstd} for closed surfaces. 
Theorem \ref{maintheorem:mod}, which concerns the trivial representation, is the case $r=0$.

\begin{theorem}
\label{theorem:closed}
Let $g \geq 0$ and $\ell \geq 2$.  Then for $r \geq 0$, the map
\[\HH_k\left(\Mod_{g}\left(\ell\right);\HH_1\left(\Sigma_{g};\Q\right)^{\otimes r}\right) \rightarrow \HH_k\left(\Mod_{g};\HH_1\left(\Sigma_{g};\Q\right)^{\otimes r}\right)\]
is an isomorphism if $g \geq 2(k+r)^2+7k+6r+2$.
\end{theorem}
\begin{proof}
To simplify our notation, let $V = \HH_1(\Sigma_g;\Q)^{\otimes r}$.
We will adapt to our situation a beautiful argument of Randal-Williams \cite{RandalWilliamsDiscs} for proving homological stability for mapping class groups of closed surfaces.
For $g \geq 3$ and $b \geq 0$, let $\fD_g^b = \Diff^{+}(\Sigma_g^b,\partial \Sigma_g^b)$ and let
$\fD_g^b[\ell]$ be the kernel of the action of $\fD_g^b$ on $\HH_1(\Sigma_g;\Z/\ell)$ obtained by gluing discs to all the boundary components.
We thus have
\[\Mod_g^b = \pi_0(\fD_g^b) \quad \text{and} \quad \Mod_g^b[\ell] = \pi_0(\fD_g^b[\ell]).\]
Since $g \geq 3$, theorems of Earle--Eells \cite{EarleEells} and Earle--Schatz \cite{EarleSchatz} say that the components of $\fD_g^b$ are all contractible.  This implies that
\[\HH_k(\Mod_g^b;V) \cong \HH_k(B\fD_g^b;V) \quad \text{and} \quad \HH_k(\Mod_g^b[\ell];V) \cong \HH_k(B\fD_g^b[\ell];V).\]
By Theorem \ref{maintheorem:modstd2}, for $b \geq 1$ the map
\[\HH_k(\Mod_g^b[\ell];V) \rightarrow \HH_k(\Mod_g^b;V)\]
is an isomorphism if $g \geq 2(k+r)^2+7k+6r+2$.  It follows that the map
\begin{equation}
\label{eqn:knowncases}
\HH_k(B\fD_g^b[\ell];V) \rightarrow \HH_k(B\fD_g^b;V)
\end{equation}
is also an isomorphism if $g \geq 2(k+r)^2+7k+6r+2$.  Our goal is to prove that the map
\[\HH_k(B\fD_g[\ell];V) \rightarrow \HH_k(B\fD_g;V)\]
is an isomorphism in that same range.

Assume that $g \geq 2(k+r)^2+7k+6r+2$.
Randal-Williams (\cite{RandalWilliamsDiscs}; see \cite[\S 5]{WahlHandbook} for an expository reference) introduced a semisimplicial space of discs embedded in $\Sigma_g$ and proved its geometric realization was contractible.  He then showed that this leads to a spectral sequence converging to the homology of $\fD_g$.
Though he worked with trivial coefficients, his exact same argument also works with the coefficient system $V$, for which the
spectral sequence in question has the form
\[E^1_{pq} = \HH_q(B \fD_g^{p+1};V) \Rightarrow \HH_{p+q}(B\fD_g;V).\]
The key fact that underlies the identification of this spectral sequence is the fact that for all $p$, the group $\fD_g$ acts transitively
on the set of orientation-preserving embeddings
\begin{equation}
\label{eqn:discembeddings}
\sqcup_{i=0}^p \bbD^2 \rightarrow \Sigma_g
\end{equation}
and the stabilizer of one of these embeddings is isomorphic to $\fD_g^{p+1}$.  The same thing is true for $\fD_g[\ell]$; indeed, even the identity component
of $\fD_g$ acts transitively on embeddings \eqref{eqn:discembeddings}.  We thus also get a spectral sequence with
\[(E')^1_{pq} = \HH_q(B \fD_g^{p+1}[\ell];V) \Rightarrow \HH_{p+q}(B\fD_g[\ell];V).\]
We remark that $\fD_g^{p+1}[\ell]$ appears here rather than $\fD_g^{p+1}(\ell)$ since the stabilizer only fixes
$\HH_1(\Sigma_g;\Z/\ell)$, not $\HH_1(\Sigma_g^{p+1};\Z/\ell)$.
There is a map $E' \rightarrow E$ between these spectral sequences, and by our discussion of \eqref{eqn:knowncases} above
the map $(E')^1_{pq} \rightarrow E^1_{pq}$ is an isomorphism for $q \leq k$ and all $p$.  It is also a surjection
for all $p$ and $q$ by the transfer map lemma (Lemma \ref{lemma:transfer}).  By the spectral sequence comparison theorem, we deduce
that the map
\[\HH_k(B\fD_g[\ell];V) \rightarrow \HH_k(B\fD_g;V)\]
is an isomorphism, as desired.
\end{proof}

\end{document}